\documentclass{siamltex}
\usepackage{amsmath}
\usepackage{latexsym}
\usepackage{algorithm}
\usepackage{algorithmic}
\usepackage{graphicx}
\usepackage{subfigure}

\newtheorem{assumption}{Assumption}
\newtheorem{Exa}{Example}[section]

\newcommand{\rhomax}{\rho_{\mbox{\rm\scriptsize max}}}

\def\tto{\;{\lower 1pt \hbox{$\rightarrow$}}\kern -10pt
           \hbox{\raise 2.8pt \hbox{$\rightarrow$}}\;}
\newcommand{\inT}{\mbox{\rm int}\,}

\newcommand{\beq}{\begin{equation}}
\newcommand{\eeq}{\end{equation}}
\newcommand{\IR}{\makebox{\sf I \hspace{-8.5 pt} R \hspace{-4.0pt}}}
\newcommand{\R}{\IR}

\newcommand{\cF}{{\cal F}}

\newcommand{\cN}{{\cal N}}

\newcommand{\cS}{{\cal S}}

\newcommand{\dom}{\mbox{\rm dom }}

\newcommand{\trace}{\mbox{\rm trace}}

\newcommand{\DC}{\Delta C}

\providecommand{\norm}[1]{\lVert#1\rVert} 

\def\eqnok#1{(\ref{#1})}

\def\noprint#1{}
\newcommand{\eat}[1]{}

\newcommand{\btab}{\begin{tabbing}
\ \ \= thenn \= thenn \= thenn \= thenn \= thenn \= \kill}
\newcommand{\etab}{\end{tabbing}}

\title{Packing Ellipsoids with Overlap%
  \thanks{Version of \today. Research supported by NSF Grants
    DMS-0914524 and DMS-0906818, and DOE Grant DE-SC0002283.}}

\author{Caroline Uhler%
\thanks{IST Austria, Am Campus 1, 3400 Klosterneuburg, Austria. \texttt{caroline.uhler@ist.ac.at}}
\and 
Stephen J. Wright %
\thanks{Computer Sciences Department, 1210
W. Dayton Street, University of Wisconsin, Madison, WI 53706, USA. \texttt{swright@cs.wisc.edu}}}

\begin{document}

\maketitle

\begin{abstract}
  The problem of packing ellipsoids of different sizes and shapes into
  an ellipsoidal container so as to minimize a measure of overlap
  between ellipsoids is considered. A bilevel optimization formulation
  is given, together with an algorithm for the general case and a
  simpler algorithm for the special case in which all ellipsoids are
  in fact spheres. Convergence results are proved and computational
  experience is described and illustrated. The motivating application
  --- chromosome organization in the human cell nucleus --- is
  discussed briefly, and some illustrative results are presented.
\end{abstract}



\section{Introduction} \label{sec:intro}

Shape packing problems have been a popular area of study in discrete
mathematics over many years. Typically, such problems pose the
question of how many uniform objects can be packed without overlap
into a larger container, or into a space of infinite extent with
maximum density. In ellipsoid packing problems, the smaller shapes are
taken to be ellipsoids of known size and shape. In three dimensions,
the ellipsoid packing problem has become well known in recent years,
due in part to colorful experiments involving the packing of M\&Ms
\cite{Donev}.
% In two dimensions, the
% ellipsoid packing problem is referred to as ellipse packing.

%, in such a way as to maximize density (the
%proportion of the container's space that is filled by the smaller
%shapes)

Finding densest ellipsoid packings is a difficult computational
problem. Most studies concentrate on the special case of sphere
packings, with spheres of identical size. Here, optimal densities have
been found for the infinite Euclidean space of dimensions two and
three. In two dimensions, the densest circle packing is given by the
hexagonal lattice (see \cite{Thue}), where each circle has six
neighbors. The density of this packing (that is, the proportion of the
space filled by the circles) is $\pi/\sqrt{12}$. In dimension three,
it has been proven recently by Hales~\cite{Hales} that the
face-centered cubic (FCC) lattice achieves the densest packing. In
this arrangement, every sphere has 12 neighboring spheres and the
density is $\pi/\sqrt{18}$. For dimensions higher than 3, the problem
of finding the densest sphere packing is still open.

A problem related to sphere packing is {\em sphere covering}. Here,
the goal is to find an arrangement that covers the space with a set of
uniform spheres, as economically as possible. Overlap is not only
allowed in these arrangements, but inevitable. The density is defined
similarly to sphere packing (that is, the total volume of the spheres
divided by the volume covered), but now we are interested in finding
an arrangement of {\em minimal} density. In two dimensions, as for
circle packing, the optimal circle covering is given by the regular
hexagonal arrangement. However, the thinnest sphere covering in
dimension 3 is given not by the FCC lattice, but by the body-centered
cubic (BCC) lattice. In this arrangement, every sphere intersects with
fourteen neighboring spheres; see for example \cite{Rogers}.

In this paper we study a problem that falls between ellipsoid
packing and covering. Given a set of ellipsoids of diverse size and
shape, and a finite enclosing ellipsoid, we seek an arrangement that
minimizes some measure of total overlap between ellipsoid pairs.

% in that (like packing) we require the ellipsoids to lie fully within
% the , and (like covering) we allow them to intersect. We seek an
% arrangement such that some measure of total overlap is minimized,
% when both, the collection of packed ellipsoids and the outer
% container, have fixed (but diverse) sizes and shapes.

Our formulation is motivated by chromosome organization in human cell
nuclei. In biological sciences, the study of chromosome arrangements
and their functional implications is an area of great current
interest.  The territory occupied by each chromosome can be modeled as
an ellipsoid, different chromosomes giving rise to ellipsoids of
different size. The enclosing ellipsoid represents a cell nucleus, the
size and shape of which differs across cell types. Overlap between
chromosome territories has biological significance: It allows for
interaction and co-regulation of different genes. Also of key
significance are the DNA-free interchromatin channels that allow
access by regulatory factors to chromosomes deep inside a cell
nucleus.  Smaller nuclei tend to have tighter packings, so that fewer
channels are available, and the chromosomes packed closest to the
center may not be accessible to regulatory factors.

The arrangement of chromosome territories is neither completely random
nor deterministic. Certain features of the arrangement are believed to
be conserved during evolution \cite{Cremer_primates}, but can change
during such processes as cell differentiation and cancer development
\cite{Berezney_cancer}.  In general, smaller and more
gene-dense chromosomes are believed to be found closer to the center
of the nucleus \cite{Cremer_radial}, and heterologous chromosomes tend
to be nearer to each other than homologous pairs
\cite{Khalil_heterologues}. For further background on chromosome arrangement properties,
see \cite{Cremer, Berezney_nonrandom}. 

A major goal of this paper is to determine whether the experimental
observations made to date about chromosome organization can be
explained in terms of simple geometrical principles, such as minimal
overlap. The minimum-overlap principle appears to be consistent with
the tendency of chromosome territories to exploit the whole volume of
the nucleus, to make the DNA-free channels as extensive as
possible. Our formulation also includes features to discourage close
proximity of homologous pairs. 
% With our biological application, our aim is to explore how
% thoroughly observed chromosome arrangements are explicable from
% simple geometric principles alone.

% First, the shape of the space each chromosome can occupy looks more
% like an ellipsoid than a sphere. Second, the volumes of the
% ellipsoids are non-identical due to the different sizes of the
% chromosomes. Finally, overlap among the different chromosomes is
% allowed, but at the same time the chromosomes usually don't fill the
% whole space. We are interested in finding an ellipsoidal arrangement
% with minimal overlap. This is neither the common packing nor the
% covering problem.

The remainder of the paper is organized as follows.  In
Section~\ref{sec:prob}, we outline the mathematical formulation,
define notation, and state a key technical result concerning algebraic
formulations of ellipsoidal containment. In Section~\ref{sec:sphere},
we study the special case of finding a minimal overlap configuration
of {\em spheres} inside an ellipsoidal container. We describe a simple
iterative procedure based on convex linearized approximations that
produces convergence to stationary points of the minimal-overlap
problem.  We show through simulations that our algorithm can be used
to recover known optimal circle and sphere packings.  In
Section~\ref{sec:ellipsoid}, we generalize our optimization procedure
to ellipsoid packing, introducing trust-region stabilization and
proving convergence results.  Section~\ref{sec:biology} describes the
application of our algorithms to chromosome arrangement.

\subsection*{Notation} When $A$ and $B$ are two symmetric matrices,
the relation $A \preceq B$ indicates that $B-A$ is positive
semidefinite, while $A \prec B$ denotes positive definiteness of
$B-A$. Similar definitions apply for $\succeq$ and $\succ$.

Let $X$ be a finite-dimensional vector space over the reals $\R$
endowed with inner product $\langle \cdot, \cdot \rangle$. (The usual
Euclidean space $\R^n$ with inner product $\langle x,y \rangle :=
x^Ty$ and the space of symmetric matrices ${\cal S} \R^{n \times n}$
with inner product $\langle X,Y\rangle := \trace (XY)$ are two examples
of particular interest in this paper.) Given a closed convex subset
$\Omega \subset X$, the normal cone to $\Omega$ at a point $x$ is
defined as
\beq \label{def.normal}
N_{\Omega}(x) := \{ v \in X \mid \langle v,y-x \rangle \le 0 \; \mbox{for all $y \in \Omega$} \}.
\eeq

We use $\partial f$ to denote the Clarke subdifferential of the
function $f: X \to \R$. In defining this quantity, we follow Borwein
and Lewis~\cite[p.~124]{BorL00} by assuming Lipschitz continuity of
$f$ at $x$, and defining the Clarke directional derivative as follows:
\[
f^{\circ}(x;h) := \limsup_{y \to x, t \downarrow 0} \, \frac{f(y+th)-f(y)}{t}.
\]
The Clarke subdifferential is then
\beq \label{def:clarke.sub}
\partial f(x) := \{ v \in X \mid \langle v,h \rangle \le f^{\circ} (x;h), \;
\mbox{for all $h \in X$} \}.
\eeq
When $f$ is convex (in addition to Lipschitz continuous), this
definition coincides with the usual subdifferential from convex
analysis, which is 
\[
\partial f(x) := \{ v \in X \mid f(y) \ge f(x) + \langle v,y-x \rangle \;
\mbox{for all $y \in \dom f$} \}
\]
(see \cite[Proposition~2.2.7]{Cla83}).

\section{Problem Description and Preliminaries} \label{sec:prob}

An ellipsoid $\mathcal{E} \subset \R^n$ can be specified in terms of
its center $c \in \R^n$ and a symmetric positive definite eccentricity
matrix $S \in S \R^{n \times n}$. We can write
\beq \label{eq:defE}
\mathcal{E}  := \{x\in\R^n \mid (x-c)^T S^{-2} (x-c)\leq 1 \}  =
\{ c + S u \mid \norm{u}_2 \le 1 \}.
\eeq
It is often convenient to work with the quantity $\Sigma := S^2$ (also
symmetric positive definite), and thus to rewrite the definition
\eqnok{eq:defE} as
\beq \label{eq:defSig,gen}
\mathcal{E}  := \{x\in\R^n \mid (x-c)^T \Sigma^{-1} (x-c)\leq 1 \}.
\eeq

For the remainder of this section, we assume that $n=3$, that is, the
ellipsoids are three-dimensional. The eigenvalues of $S$ are the
lengths of the principal semi-axes of $\mathcal{E}$; we denote these
by $r_1$, $r_2$, and $r_3$, and assume that these three positive
quantities are arranged in nonincreasing order.  It follows that the
eigenvalues of $\Sigma$ are $r_1^2$, $r_2^2$, and $r_3^2$, and that
the matrices $S$ and $\Sigma$ have the form
\[
S = Q \left[ \begin{matrix} r_1 & 0 & 0 \\ 0 & r_2 & 0 \\ 0 & 0 & r_3 \end{matrix}
\right] Q^T, \qquad
\Sigma = Q \left[ \begin{matrix} r_1^2 & 0 & 0 \\ 0 & r_2^2 & 0 \\ 0 & 0 & r_3^2 \end{matrix}
\right] Q^T,
\]
for some orthogonal matrix $Q$, which determines the orientation of
the ellipse.

In this paper, we are given the semi-axis lengths $r_{i1}$, $r_{i2}$,
and $r_{i3}$ for a collection of $N$ ellipsoids $\mathcal{E}_i$,
$i=1,2,\dotsc,N$. The goal is to specify centers $c_i$ and matrices
$S_i$ for these ellipsoids, such that
\begin{itemize}
\item[(a)] $\mathcal{E}_i \subset \mathcal{E}$, for some fixed
  ellipsoidal container $\mathcal{E}$;
\item[(b)] The eigenvalues of $S_i$ are $r_{i1}$, $r_{i2}$, and
  $r_{i3}$, for $i=1,2,\dotsc,N$;
\item[(c)] Some measure of volumes of the pairwise overlaps
  $\mathcal{E}_i \cap \mathcal{E}_j$, $i,j=1,2,\dotsc,N$, $i \neq j$,
  is minimized.
\end{itemize}
In the following subsections, we give more specific formulations of (c),
first for the case in which all $\mathcal{E}_i$ are spheres (that is,
$r_{i1}=r_{i2}=r_{i3}$, $i=1,2,\dotsc,N$) and then for the general
case. For now, we note that a crucial element in formulating these
problems is ellipsoidal containment, that is, algebraic conditions
that ensure that one given ellipsoid is contained in another. This is
the subject of the following lemma, which is a simple application of
the S-procedure (see \cite[Appendix B.2]{BoyV03}).

\begin{lemma}
\label{lem:S_procedure}
Define two ellipsoids as follows:
\begin{align*}
\mathcal{E} &= \{x\in\R^3 \mid (x-c)^T S^{-2} (x-c)\leq 1 \}  =
\{ c + S u \mid \norm{u}_2 \le 1 \},  \\
 \bar{\mathcal{E}} &= \{x\in\R^3 \mid (x-\bar{c})^T \bar{S}^{-2} (x-\bar{c})\leq 1 \}  =
\{\bar{c} + \bar{S}u \mid \norm{u}_2 \leq 1 \}.
\end{align*}
The containment condition $\bar{\mathcal{E}}\subset\mathcal{E}$ can be
represented as the following linear matrix inequality (LMI) in
parameters $\bar{c}$, $\bar{S}$, $c$, and $S^2$: There exists $\lambda
\in \R$ such that
\begin{equation}
\label{eq:LMI}
\begin{pmatrix} - \lambda I &0&\bar{S}\\0& \lambda-1&(\bar{c}-c)^T\\ \bar{S}&\bar{c}-c&-S^2 \end{pmatrix}\preceq 0.
\end{equation}
\end{lemma}
\begin{proof}
  The condition $\bar{\mathcal{E}} \subset \mathcal{E}$ can be
  expressed as
\[
(\bar{c} + \bar{S}u-c)^TS^{-2} (\bar{c} + \bar{S}u-c)\leq 1 \quad \textrm{for all } u  \textrm{ such that }  \norm{u}_2 \leq 1.
\]
By multiplying out this inequality we get
\[
u^T\bar{S}S^{-2}\bar{S}u+2u^T\bar{S}S^{-2} (\bar{c}-c)+(\bar{c}-c)^TS^{-2}(\bar{c}-c)-1\leq 0
\]
for all $u$ such that $u^Tu -1\leq 0$. By applying the S-procedure, we
find that this is equivalent to the existence of $\lambda > 0$ such
that
\begin{equation}
\label{eq:sproc}
\begin{pmatrix} \bar{S}S^{-2} \bar{S}&\bar{S}S^{-2} (\bar{c}-c)\\(\bar{c}-c)^TS^{-2}\bar{S}& (\bar{c}-c)^TS^{-2} (\bar{c}-c)-1\end{pmatrix}
\preceq\lambda\begin{pmatrix} I & 0\\0& -1\end{pmatrix}.
\end{equation}
This expression is not linear in the variables $\bar{c}$, $\bar{S}$,
$c$, and $S^2$, but an elementary Schur complement argument shows
equivalence to the linear matrix inequality \eqnok{eq:LMI}, completing
the proof.
\end{proof}

As one special case, the condition $\mathcal{E}_i \subset
\mathcal{E}$, where $\mathcal{E}_i$ is a sphere with center $c_i$ and
radius $r_i$ and $\mathcal{E}$ is an ellipsoid centered at $0$ with
matrix $S$, can be represented as the LMI:
\beq \label{eq:S.circ}
\begin{pmatrix} - \lambda_i I &0&r_i I\\0& \lambda_i-1&c_i^T\\ r_i I&c_i&-S^2\end{pmatrix}\preceq 0.
% , \qquad \lambda_i\geq 0.
\eeq
The more general case of $\mathcal{E}_i \subset \mathcal{E}$, where
$\mathcal{E}_i$ is an ellipsoid with center $c_i$ and matrix $S_i$ and
$\mathcal{E}$ is an ellipsoid centered at $0$ with matrix $S$, can be
represented as the LMI:
\beq \label{eq:S.ell}
\begin{pmatrix} - \lambda_i I &0&S_i \\0& \lambda_i-1&c_i^T\\ S_i &c_i &-S^2 \end{pmatrix}\preceq 0.
\eeq

\section{Sphere Packing} \label{sec:sphere}

We give a problem formulation for the case in which all enclosed
shapes are spheres (of arbitrary dimension), and present a successive
approximation algorithm that is shown to accumulate or converge to a
stationary point of the formulation. Some examples of results obtained
with this approach are described at the end of the section.

\subsection{Formulation and Algorithm} \label{sec:sphere:alg}

When the inscribed objects are spheres, the variables in the problem
are the centers $c_i \in R^m$, $i=1,2,\dots ,N$, which we
aggregate as follows:
\beq \label{eq:defc}
c := (c_1,c_2, \dotsc,c_N).
\eeq
The radii $r_i$, $i=1,2,\dotsc,N$ are given.  We express the
containment condition for each sphere as follows:
\beq \label{eq:defOmi}
\mathcal{E}_i \subset \mathcal{E}  \;\; \Leftrightarrow \;\;
c_i \in \Omega_i,
\eeq
where $\Omega_i$ is a closed, bounded, convex set with nonempty
interior. When $\mathcal{E}$ is a sphere of radius $R$ centered at
$0$, we have $\Omega_i := \{ c_i \, : \, \|c_i \| \le
R-r_i\}$. Otherwise, we can define $\Omega_i$ implicitly by
Lemma~\ref{lem:S_procedure}; see in particular \eqnok{eq:S.circ}.

A simple measure for the overlap between two spheres $\mathcal{E}_i$
and $\mathcal{E}_j$ is the diameter of the largest sphere inscribed
into the intersection, which we denote by an auxiliary variable $\xi_{ij}$:
\beq \label{eq:defxi} 
\xi_{ij} := \max (0, (r_i
+r_j)-\norm{c_i-c_j}_2), \qquad \xi := (\xi_{ij})_{1 \le i < j \le N}.
\eeq
Our minimum-overlap problem can thus be formulated as follows:
\begin{subequations}
\label{prob:min_overlap}
\begin{align}
\min_{c,\xi} \qquad&
H(\xi) 
\label{prob:min_overlap.1}
\\
\mbox{subject to} \qquad &
(r_i + r_j) - \norm{c_i - c_j}_2 \leq \xi_{ij}
&& \mbox{for }  1\leq i< j \leq N  
\label{prob:min_overlap.2}
\\
& 0 \leq \xi, 
\label{prob:min_overlap.3}
\\
& c_i \in \Omega_i, && 
% & \mbox{$\mathcal{E}_i \subset \mathcal{E}$} &&
\mbox{for }  i = 1,\dots , N,  
\label{prob:min_overlap.4}
\end{align}
\end{subequations}
where \eqnok{prob:min_overlap.3} denotes the entrywise condition
$\xi_{ij} \ge 0$, $1 \le i < j \le N$.  The objective $H:
\R^{n(n-1)/2}_+ \to \R_+$ satisfies the following assumption.
\begin{assumption} \label{ass:H} The function $H: \R^{n(n-1)/2}_+ \to
  \R_+$ is convex and continuous, with the following additional
  properties:
\begin{itemize}
\item[(a)] $H(0)=0$;
\item[(b)]  $H(\xi)>0$ whenever $\xi \neq 0$;
\item[(c)] $0 \le \bar{\xi} \le \xi \Rightarrow H(\bar{\xi}) \le H(\xi)$.
\end{itemize}
\end{assumption}
% \noindent
% (We denote by $\bar{\xi} \le \xi$ the entrywise condition
% $\bar{\xi}_{ij} \le \xi_{ij}$ for $1 \le i < j \le N$.)

Assumption~\ref{ass:H} is satisfied, for example, by the norms
$H(\xi)=\|\xi\|_1$, $H(\xi)=\|\xi\|_2$, and $H(\xi) = \|
\xi\|_{\infty} = \max_{1 \le i < j \le N} \, |\xi_{ij}|$. In the
application to be discussed below, we prefer the overlaps in the
overlapping ellipsoids to be roughly the same size; for this purpose,
the $\ell_2$ and $\ell_{\infty}$ norms are the most appropriate. 
% We discuss the choice of $H$ further in later sections.

Although the objective \eqnok{prob:min_overlap.1} and containment
constraints \eqnok{prob:min_overlap.4} are convex, the problem
\eqnok{prob:min_overlap} is nonconvex, due to the constraints
\eqnok{prob:min_overlap.2}. A point is {\em Clarke-stationary} for
\eqnok{prob:min_overlap} if the following conditions are satisfied,
for some $\lambda_{ij} \in \R$, $1 \le i < j \le N$:
\begin{subequations}
\label{prob:min_overlap_stat}
\begin{align}
\label{prob:min_overlap_stat.1}
0 \le g_{ij} - \lambda_{ij}  \perp \xi_{ij} \ge 0 \;\; \mbox{for some} \; g_{ij} 
& \in \partial_{\xi_{ij}} H(\xi), && 1 \le i < j \le N, 
\\
\label{prob:min_overlap_stat.2}
\sum_{j=i+1}^{N} \lambda_{ij} w_{ij} - 
\sum_{j=1}^{i-1} \lambda_{ji} w_{ji} & \in N_{\Omega_i}(c_i), && i=1,2,\dotsc,N,
\\
\label{prob:min_overlap_stat.3}
0 \le \xi_{ij} + \norm{c_i-c_j}-(r_i+r_j) & \perp \lambda_{ij} \ge 0, 
&& 1 \le i < j \le N, 
\\
\label{prob:min_overlap_stat.4}
\mbox{where} \; \norm{w_{ij}}_2 \le 1, \quad \mbox{with} \;\; 
w_{ij} = \frac{c_i-c_j}{\norm{c_i-c_j}_2} \; & \mbox{when $c_i \neq c_j$},
&& 1 \le i < j \le N.
\end{align}
\end{subequations}
Condition \eqnok{prob:min_overlap_stat.4} defines $w_{ij}$ to be in 
the subdifferential of $\norm{c_i-c_j}_2$ with respect to $c_i$. See
\eqnok{def.normal} for the definition of the normal cone in
\eqnok{prob:min_overlap_stat.2}.

We now develop an algorithm that seeks a local solution of
\eqnok{prob:min_overlap}, by formulating a sequence of convex
approximations in which the key feature is linearization of the
nonconvex constraint \eqnok{prob:min_overlap.2} around the current
iterate. Because of the special properties of this problem, we need
not apply the usual safeguards for this successive approximation
approach, such as trust regions or line searches. Decrease of the
objective at each iteration and accumulation of the iteration sequence
at first-order points of the problem \eqnok{prob:min_overlap} can be
proved in the absence of these features. However, for purposes of stabilizing
the iterates generated by the method, it may be desirable to place a
uniform bound on the length of each step. This can be done without
complicating the analysis, and we do so in our implementations. 
% (For the general ellipsoid packing problem discussed in the next
% section, however, a trust-region framework is needed to obtain
% similar convergence properties.)

The linearization of \eqnok{prob:min_overlap} around the current
iterate $c^-$ is defined as follows:
\begin{subequations}
\label{prob:min_overlap_lin}
\begin{align}
P(c^-) := \min_{c,\bar{\xi}} \; &
H(\bar{\xi}) 
\label{prob:min_overlap_lin.1}
\\
\mbox{subject to} \qquad &
(r_i + r_j) - z_{ij}^T (c_i - c_j) \leq \bar{\xi}_{ij},
\qquad \mbox{for }  1\leq i< j \leq N,
\label{prob:min_overlap_lin.2}
\\
& 0 \leq \bar{\xi}, 
\label{prob:min_overlap_lin.3}
\\
& c_i \in \Omega_i, \qquad
\mbox{for }  i = 1,\dots , N,  
\label{prob:min_overlap_lin.4}
\\
\mbox{where } \;\;
& z_{ij} := 
\begin{cases}
{(c_i^--c_j^-)^T}/{\norm{c_i^--c_j^-}} & \;\; \mbox{when $c_i^- \neq c_j^-$} \\
0 & \;\; \mbox{otherwise.}
\end{cases}
\label{prob:min_overlap_lin.5}
\end{align}
\end{subequations}
This problem is convex, with affine constraints except for the
inclusion \eqnok{prob:min_overlap_lin.4}, which can be satisfied
strictly when each $\Omega_i$ is closed, bounded, and convex, with
nonempty interior. Hence (see for example
\cite[Theorem~28.2, Corollary~28.3.1]{Roc70}), its solutions are
characterized by the following KKT conditions: There exist
$\lambda_{ij}$, $1 \le i < j \le N$ such that
\begin{subequations}
\label{prob:min_overlap_lin_kkt}
\begin{align}
\label{prob:min_overlap_lin_kkt.1}
0 \le g_{ij} - \lambda_{ij}  \perp \bar{\xi}_{ij} \ge 0 \;\; \mbox{for some} \; g_{ij} 
& \in \partial_{\bar{\xi}_{ij}} H(\bar{\xi}), && 1 \le i < j \le N, 
\\
\label{prob:min_overlap_lin_kkt.2}
\sum_{j=i+1}^{N} \lambda_{ij} z_{ij} - 
\sum_{j=1}^{i-1} \lambda_{ji} z_{ji} & \in N_{\Omega_i}(c_i), && i=1,2,\dotsc,N,
\\
\label{prob:min_overlap_lin_kkt.3}
0 \le \bar{\xi}_{ij} + z_{ij}^T(c_i-c_j)-(r_i+r_j) & \perp \lambda_{ij} \ge 0, 
&& 1 \le i < j \le N.
\end{align}
\end{subequations}
We can use a compactness argument to verify that solutions to
\eqnok{prob:min_overlap_lin} are attained. The vector of feasible
centers $c$ is restricted to a compact set, by the assumed properties
of $\Omega_1,\Omega_2,\dotsc,\Omega_N$. By using
\eqnok{prob:min_overlap_lin.2} we can define effective upper bounds on
the variables $\bar{\xi}_{ij}$ as follows:
\[
\bar{\xi}_{ij}' := \max  \left( 0,
\sup_{c_i \in \Omega_i, \, c_j \in \Omega_j} \,
(r_i + r_j) - z_{ij}^T (c_i - c_j) \right).
\] 
(For any feasible $c$, and given any $\bar{\xi}$ satisfying
\eqnok{prob:min_overlap_lin.2}, we can always replace $\bar{\xi}$ by
an alternative feasible point $\bar{\xi}'' \in [0,\bar{\xi}']$ without
increasing the value of $H$, by property (b) of
Assumption~\ref{ass:H}.) Thus, the problem
\eqnok{prob:min_overlap_lin} reduces to minimization of a continuous
convex function over a compact set, for which existence of a solution
is guaranteed.

\begin{algorithm}[h]
\caption{Packing Spheres by Minimizing Overlap\label{alg:circ}}
\begin{algorithmic}
\STATE Given $r_i$, $i=1,2,\dotsc,N$ and $\Omega_i$ closed, convex, bounded with nonempty interior;
\STATE Choose $c^0 \in \Omega_1 \times \Omega_2 \times \cdots \times \Omega_N$;
\FOR{$k=0,1,2,\dotsc$}
\STATE Solve $P(c^k)$ defined by \eqnok{prob:min_overlap_lin} 
to obtain $(c^{k+1},\bar{\xi}^{k+1})$;
\IF{$H(\bar{\xi}^{k+1})=H(\xi^k)$}
\STATE {\bf stop} and return $c^k$;
\ENDIF
\STATE Set $\xi_{ij}^{k+1} = \max(0, (r_i+r_j)-\norm{c_i^{k+1}-c_j^{k+1}} )$ for $1\le i < j \le N$;
\ENDFOR
\end{algorithmic}
\end{algorithm}

Algorithm~\ref{alg:circ} is the simple algorithm based on the
subproblem \eqnok{prob:min_overlap_lin}. To analyze convergence
properties of this method, we start with basic results about
stationary points and about the changes in $H$ at each iteration of
Algorithm~\ref{alg:circ}.
\begin{lemma} \label{lem:cstat} Suppose that the sets $\Omega_i$ in
  \eqnok{prob:min_overlap} are closed, bounded, and convex, with a
  nonempty interior, and that Assumption~\ref{ass:H} holds.  Then the
  following claims are true.
\begin{itemize}
\item[(i)] If the point $(c^k,\xi^k)$ satisfies the optimality
  conditions \eqnok{prob:min_overlap_lin_kkt} for the subproblem
  $P(c^k)$ defined by \eqnok{prob:min_overlap_lin}, then $(c^k,\xi^k)$
  satisfies the stationarity conditions \eqnok{prob:min_overlap_stat}
  for the problem \eqnok{prob:min_overlap}.
\item[(ii)] If the point $(c^k,\xi^k)$ satisfies the stationarity
  conditions \eqnok{prob:min_overlap_stat} for the problem
  \eqnok{prob:min_overlap} and in addition $c_i^k \neq c_j^k$ for all
  $1\le i < j \le N$, then $(c^k,\xi^k)$ satisfies the optimality
  conditions \eqnok{prob:min_overlap_lin_kkt} for the subproblem
  $P(c^k)$ defined by \eqnok{prob:min_overlap_lin}.
\item[(iii)] If $(c^k,\xi^k)$ does not satisfy the stationarity
  conditions \eqnok{prob:min_overlap_stat}, then
  $H(\bar{\xi}^{k+1}) < H(\xi^k)$.
\item[(iv)] For each $k$ we have $H(\xi^k) \le H(\bar{\xi}^k)$.
\end{itemize}
\end{lemma}
{\em Proof}.
\begin{itemize}
\item[(i)] If $(c^k,\xi^k)$ satisfies the optimality conditions
  \eqnok{prob:min_overlap_lin_kkt} for $P(c^k)$, then by setting
  $w_{ij}=z_{ij}$ in \eqnok{prob:min_overlap_stat}, we see that these
  conditions are also satisfied with the same values of $g_{ij}$ and
  $\lambda_{ij}$. (We have made the particular choice $w_{ij}=0$ when
  $c_i^k=c_j^k$.) 
\item[(ii)] If the conditions
\eqnok{prob:min_overlap_stat} are satisfied at
  $(c^k,\xi^k)$ with $c_i^k \neq c_j^k$ for all $i$, $j$ with $1\le i
  < j \le N$, then $w_{ij} = (c_i^k-c_j^k)/\norm{c_i^k-c_j^k}$ for all
  such $i$, $j$. Thus by noting that $z_{ij}=w_{ij}$ for all such $i$,
  $j$, we can verify using the same values of $g_{ij}$ and
  $\lambda_{ij}$ that $(c^k,\xi^k)$ satisfies the optimality
  conditions \eqnok{prob:min_overlap_lin_kkt}, and therefore is a solution
  of $P(c^k)$.
\item[(iii)] Note that the point $(c^k,\xi^k)$ is feasible for the
  subproblem \eqnok{prob:min_overlap_lin} (with $c^-=c^k$), so its
  optimal objective satisfies $H(\bar{\xi}^{k+1}) \le H(\xi^k)$. Since
  $(c^k,\xi^k)$ does not satisfy the stationarity conditions
  \eqnok{prob:min_overlap_stat}, however, part (i) implies that it
  cannot be a {\em solution} of \eqnok{prob:min_overlap_lin}, which
  implies that in fact $H(\bar{\xi}^{k+1}) < H(\xi^k)$, as claimed.
\item[(iv)] By using the fact that $\norm{z_{ij}^{k-1}}_2 \le 1$, we have for
  all $k \ge 1$ and all $i$, $j$ with $1 \le i < j \le N$ that
\[
\xi_{ij}^k = \max(r_i+r_j-\norm{c_i^k-c_j^k}_2,0)
\le \max(r_i+r_j - (z_{ij}^{k-1})^T(c_i^k-c_j^k),0) 
\le \bar{\xi}_{ij}^k.
\]
The result now follows immediately from Assumption~\ref{ass:H}(c). \endproof
\end{itemize}
% \qedhere

Note that in the case of coinciding centers, i.e.~$c_i=c_j$ for some
$i \neq j$, the stationarity conditions for \eqnok{prob:min_overlap}
and \eqnok{prob:min_overlap_lin} are not equivalent. This observation
yields the intriguing property --- unusual in algorithms based on
linear approximations --- that Algorithm~\ref{alg:circ} may be able to
move away from a stationarity point for \eqnok{prob:min_overlap}. That
is, if $(c^k,\xi^k)$ satisfies \eqnok{prob:min_overlap_stat} but there
is some pair $(i,j)$ with $i \neq j$ and $c_i^k=c_j^k$, then by
setting $z_{ij}=0$, the subproblem \eqnok{prob:min_overlap_lin_kkt}
may yield a solution $(c^{k+1},\bar{\xi}^{k+1})$ with
$H(\bar{\xi}^{k+1}) < H(\xi^k)$, and thus (by Lemma~\ref{lem:cstat}
(iv)) the next iterate will satisfy $H(\xi^{k+1}) < H(\xi^k)$.  Note
too that the proof of Lemma~\ref{lem:cstat} (iv) still holds if
$z_{ij}^k$ is chosen to be {\em any} vector with $\norm{z_{ij}^k}_2
\le 1$ when $c_i^k=c_j^k$. Hence, random choices for $z_{ij}^k$ in
this situation could be used in place of our choice $z_{ij}^k=0$
above, leading to some interesting algorithmic possibilities for
avoiding coincident centers and moving away from stationary points.
Since coincident centers rarely arise in the cases of interest,
however, we do not pursue these possibilities.

We now prove the main convergence result for Algorithm~\ref{alg:circ}.

\begin{theorem} \label{th:circ} Suppose that the sets $\Omega_i$ in
  \eqnok{prob:min_overlap} are closed, bounded, and convex, with a
  nonempty interior, and that Assumption~\ref{ass:H} holds. Then
  Algorithm~\ref{alg:circ} either terminates at a stationary point for
  \eqnok{prob:min_overlap}, or else  generates an infinite sequence
  $\{ c^k \}$ for which all accumulation points $\hat{c}$ are either
  stationary points for \eqnok{prob:min_overlap}, or else have
  $\hat{c}_i=\hat{c}_j$ for some pair $(i,j)$ with $1 \le i < j \le
  N$.
\end{theorem}
\begin{proof}
  Lemma~\ref{lem:cstat} (iii) says that termination can occur only if
  $(c^k,\xi^k)$ satisfies the stationarity conditions
  \eqnok{prob:min_overlap_stat}.  Hence, we need to consider only the
  case of an infinite sequence of iterates $\{ c^k \}$.  Suppose for
  contradiction that there is an accumulation point $\hat{c}$ for this
  sequence such that $\hat{c}_i \neq \hat{c}_j$ for all $(i,j)$ but
  $\hat{c}$ is not stationary for
  \eqnok{prob:min_overlap}. Considering the problem $P(\hat{c})$
  defined by \eqnok{prob:min_overlap_lin}, we have by
  Lemma~\ref{lem:cstat} (iii) that $\epsilon := H(\hat{\xi}) -
  H(\bar{\xi})>0$ (strict inequality), where $\hat{\xi}_{ij} = \max
  (0, r_i+r_j-\|\hat{c}_i-\hat{c}_j\|)$. Moreover, we can identify a
  neighborhood $\cN$ of $\bar{c}$ such that for all $c^k \in \cN$, we
  have
\beq \label{eq:hdecr}
H(\xi^{k+1}) \le H(\bar{\xi}^{k+1}) < H(\xi^k) - \epsilon/2, 
\eeq
This claim follows from Lemma~\ref{lem:cstat} (iv) and the observation that
the optimal objective in \eqnok{prob:min_overlap_lin} is a continuous
function of $c^-$, for $c^-$ near $\hat{c}$. The face that $\hat{c}_i \neq
\hat{c}_j$ for all $(i,j)$ ensures that the $z_{ij}$ are continuous
functions of $c^-$, while $H$ itself is continuous by
Assumption~\ref{ass:H}. Since there is a subsequence $\cS$ with
$\lim_{k \in \cS} c^k=\hat{c}$, we have from \eqnok{eq:hdecr} and
monotonicity of the full sequence $\{ H(\xi^k) \}$ that $H(\xi^k)
\downarrow -\infty$. This is impossible, however, since $H$ is bounded
below by $0$. We conclude therefore that all accumulation points
$\hat{c}$ are either stationary or else have $\hat{c}_i=\hat{c}_j$ for
some pair $(i,j)$, as required. 
\end{proof}

As noted above, the case in which accumulation points have coincident
centers is exceptional, so Theorem~\ref{th:circ} shows that the
algorithm usually either terminates or accumulates at stationary
points.

\subsection{Examples} \label{sec:sphere:examples}

We present several examples showing results obtained with
Algorithm~\ref{alg:circ} on various problems, and compare them with
known results. To begin, a simple example to demonstrate the
existence of local minima that are not global minima.

\begin{figure}[!t]
\centering
\subfigure[Global Solution: $o=.4122147478$]{\includegraphics[scale=0.32]{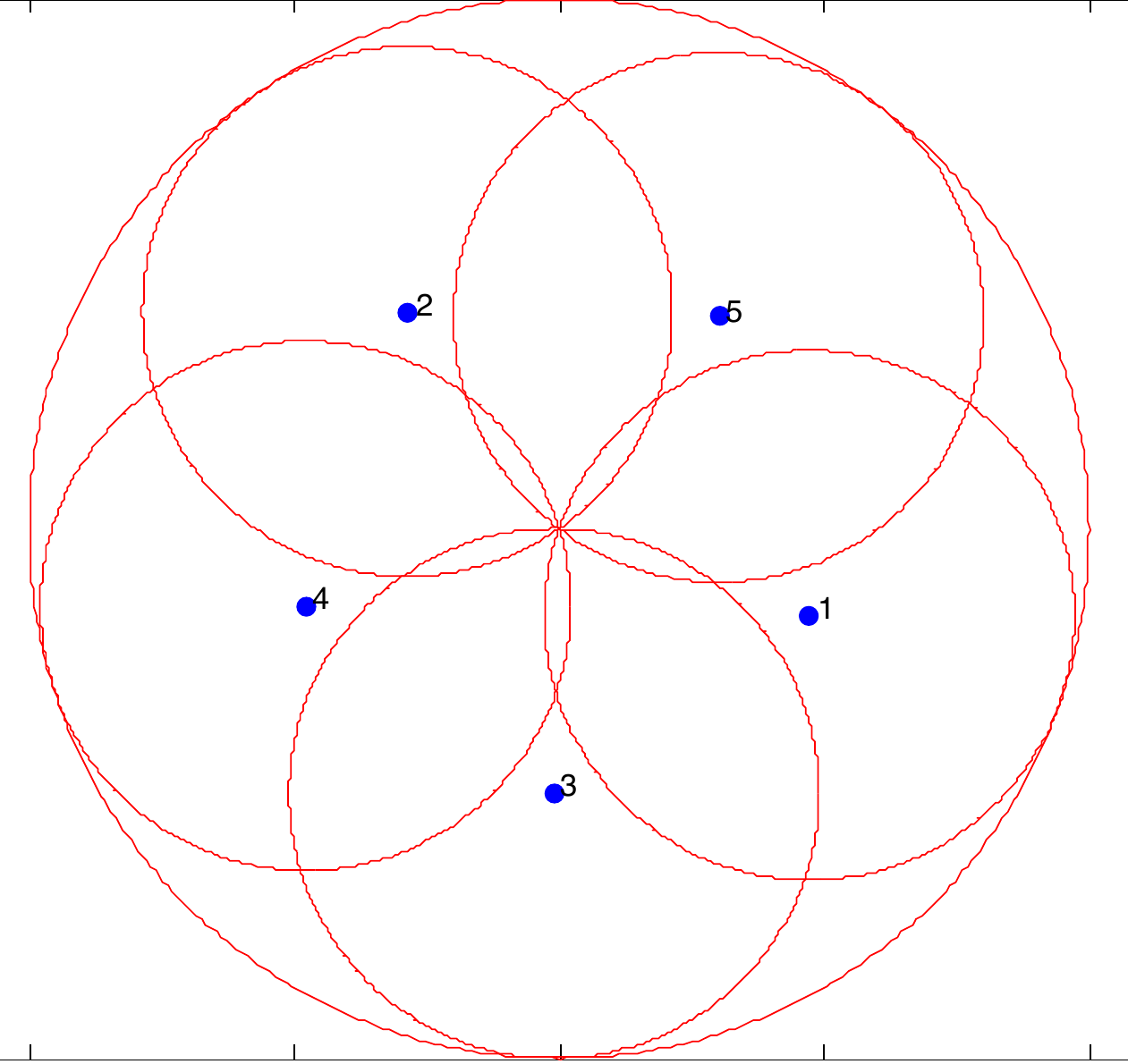}\label{fig:circles5:0}} \;
\subfigure[Local Solution: $o=.5$]{\includegraphics[scale=0.32]{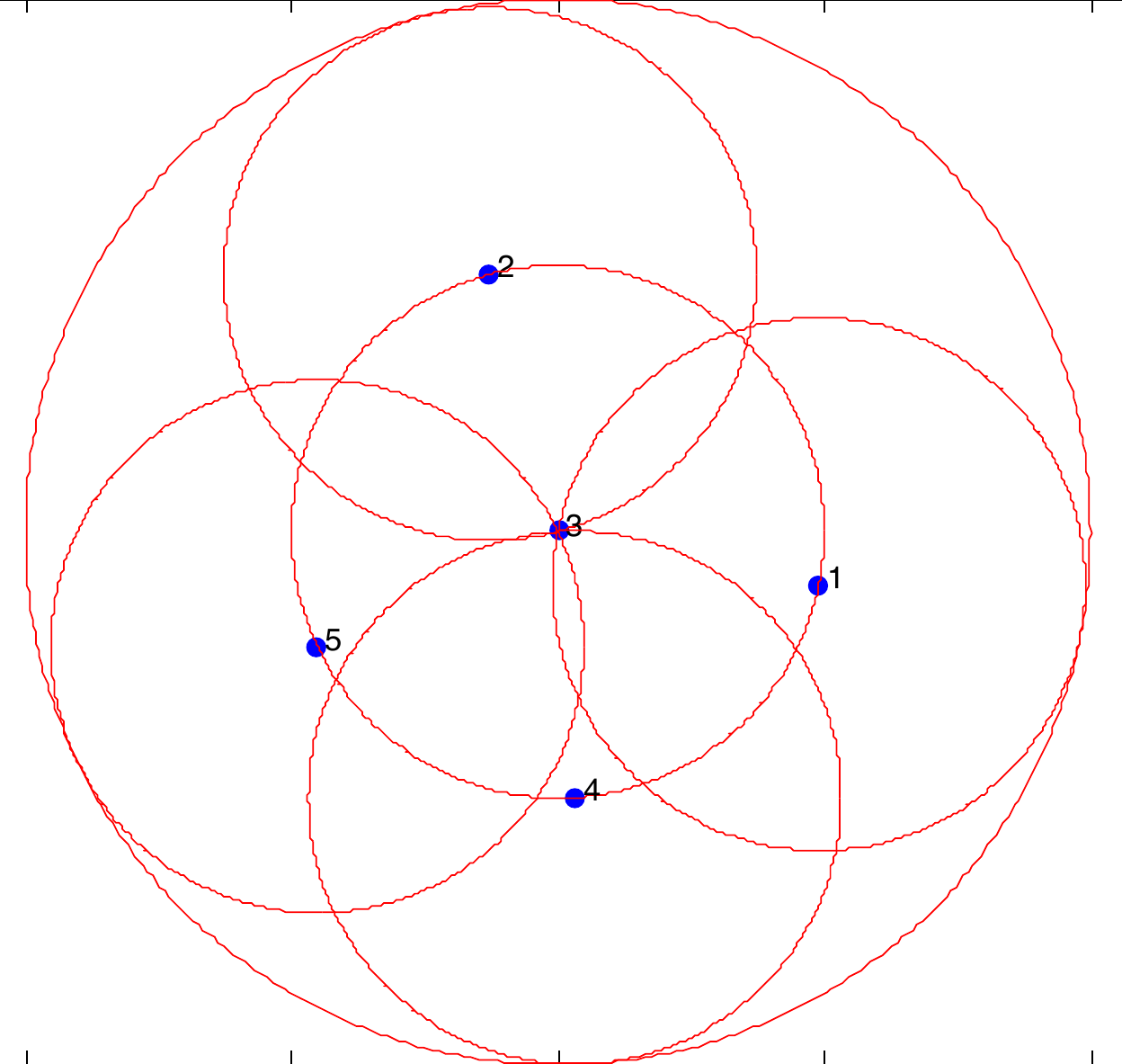}\label{fig:circles5:1}} \;
\subfigure[Local Solution: $o=.5$]{\includegraphics[scale=0.32]{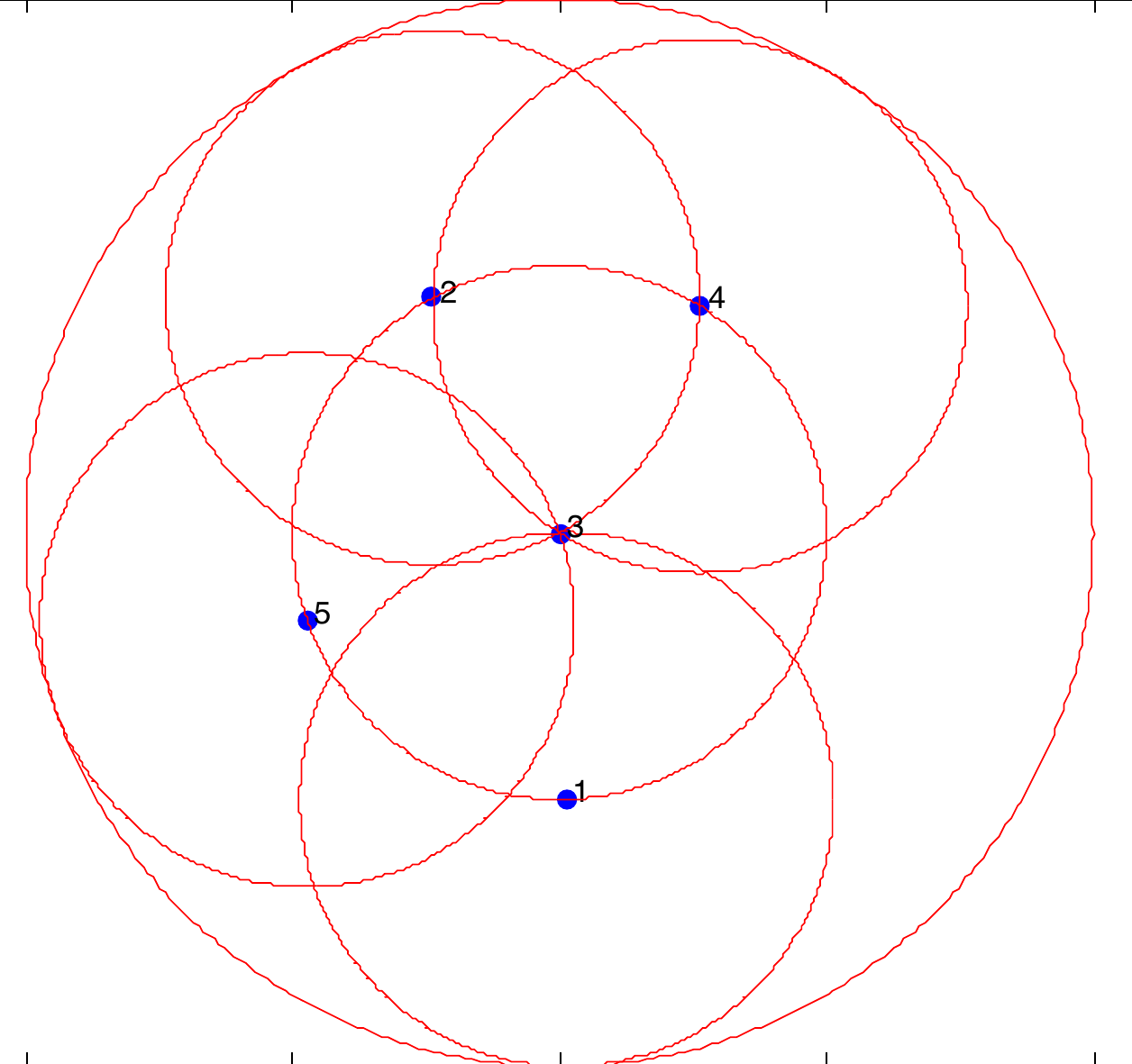}\label{fig:circles5:2}}
\label{fig:circles5}
\caption{Solutions obtained by Algorithm~\ref{alg:circ} for packing
  circles of radius $.5$ into a circle of radius $1$, showing final
  overlap measures for each.}
\end{figure}

\begin{Exa}[Five Circles]
  Consider the problem of packing five circles of radius .5 into an
  enclosing circle of radius 1. Results obtained with
  Algorithm~\ref{alg:circ}, with objective $H(\xi) =
  \|\xi\|_{\infty}$, from random starting points reveal an apparent
  global solution (Figure~\ref{fig:circles5:0}) and a family of local
  solutions (Figures~\ref{fig:circles5:1} and
  \ref{fig:circles5:2}). The local solutions are characterized by one
  of the packed circles having its center at the center of the
  enclosing circle; this circle thus has an overlap of .5 with all
  four of the outer circles. The outer circles in this local solution
  need only be arranged so that their maximum pairwise overlap is no
  greater than .5. Algorithm~\ref{alg:circ} required only a few
  iterations for each of these examples.
\end{Exa}

\begin{figure}[!t]
\centering
\includegraphics[scale=0.43]{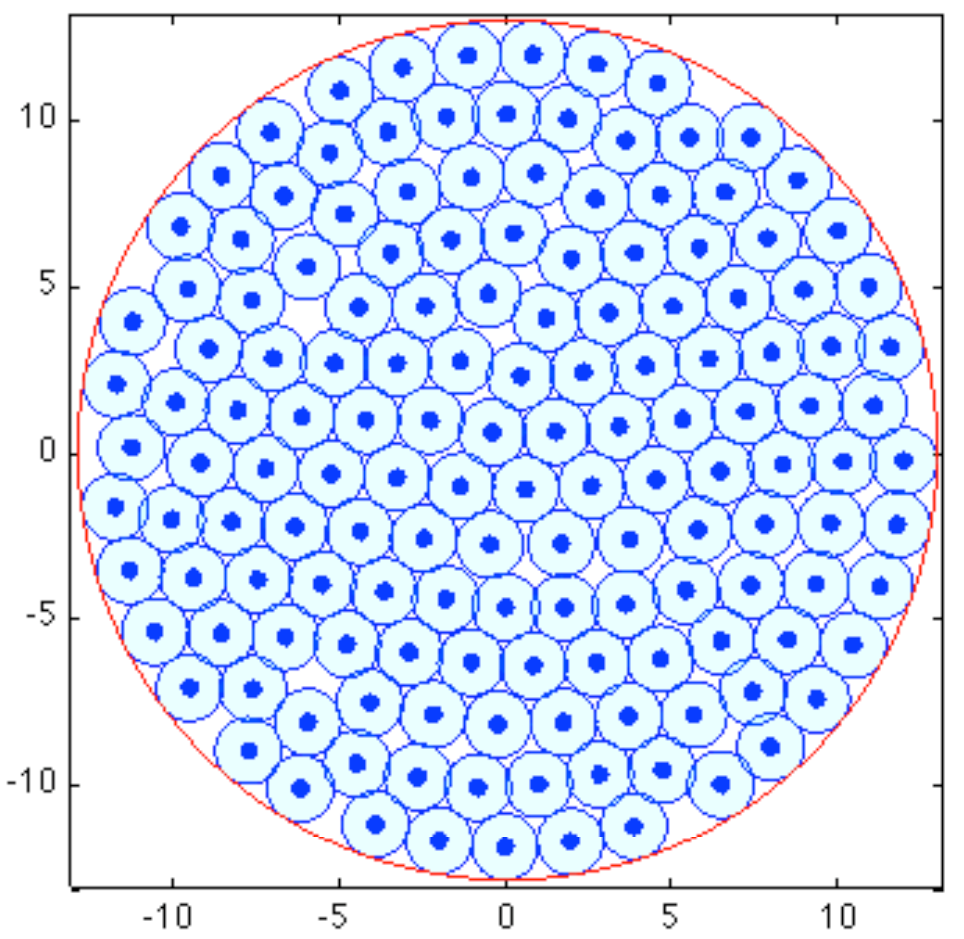} 
\caption{Circle packing in a circular enclosure. A nearly hexagonal
  arrangement is seen in the interior.}
\label{fig:circle_packing_circle}
\end{figure}

As noted in Section~\ref{sec:intro} optimal sphere packings
(configurations with no overlap) have been obtained in two and three
dimensions, for spaces of infinite extent.
%  the densest circle packing
% is achieved by the hexagonal lattice and has density
% $\pi/\sqrt{12}$. In three dimensions, the densest sphere packing in
% infinite Euclidean space is achieved by the FCC lattice, with density
% $\pi/\sqrt{18}$. 
Our algorithm can only solve problems with finite enclosing shapes,
but we can use large enclosures to investigate how similar the local
solutions attained by our algorithm are to the known optimal packings
in $\R^2$ (hexagonal lattice with density $\pi/\sqrt{12}$) and $\R^3$
(FCC lattice with density $\pi/\sqrt{18}$).

\begin{figure}[!b]
\centering
\subfigure[$o=.1192514295$]{\includegraphics[scale=0.29]{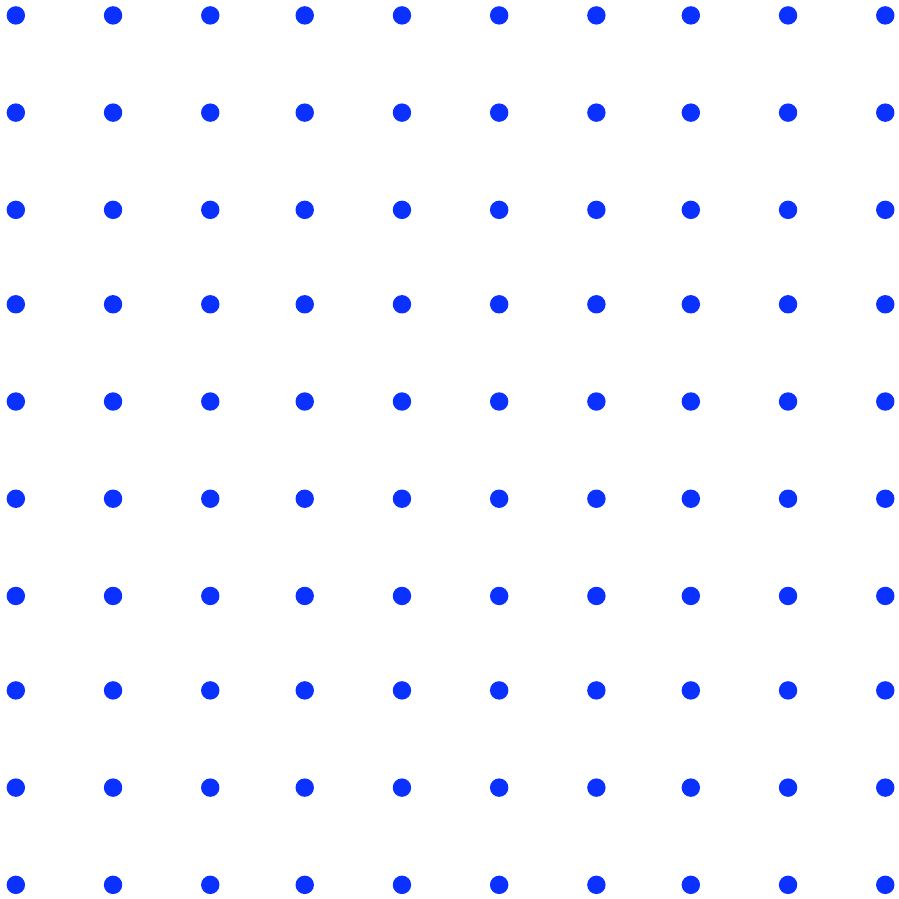}\label{fig:sphere_packing_in_square:0}} \;
\subfigure[$o=.1188906843$]{\includegraphics[scale=0.30]{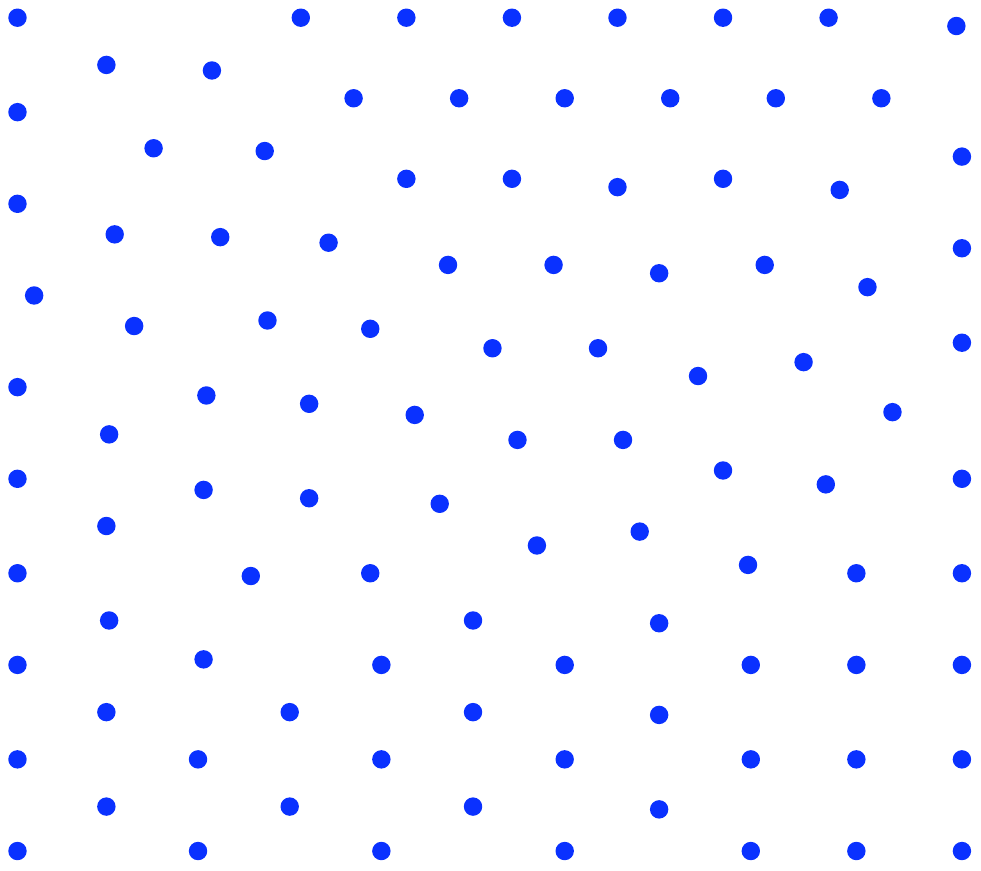}\label{fig:sphere_packing_in_square:1}} \;
\subfigure[$o=.1181440939$]{\includegraphics[scale=0.30]{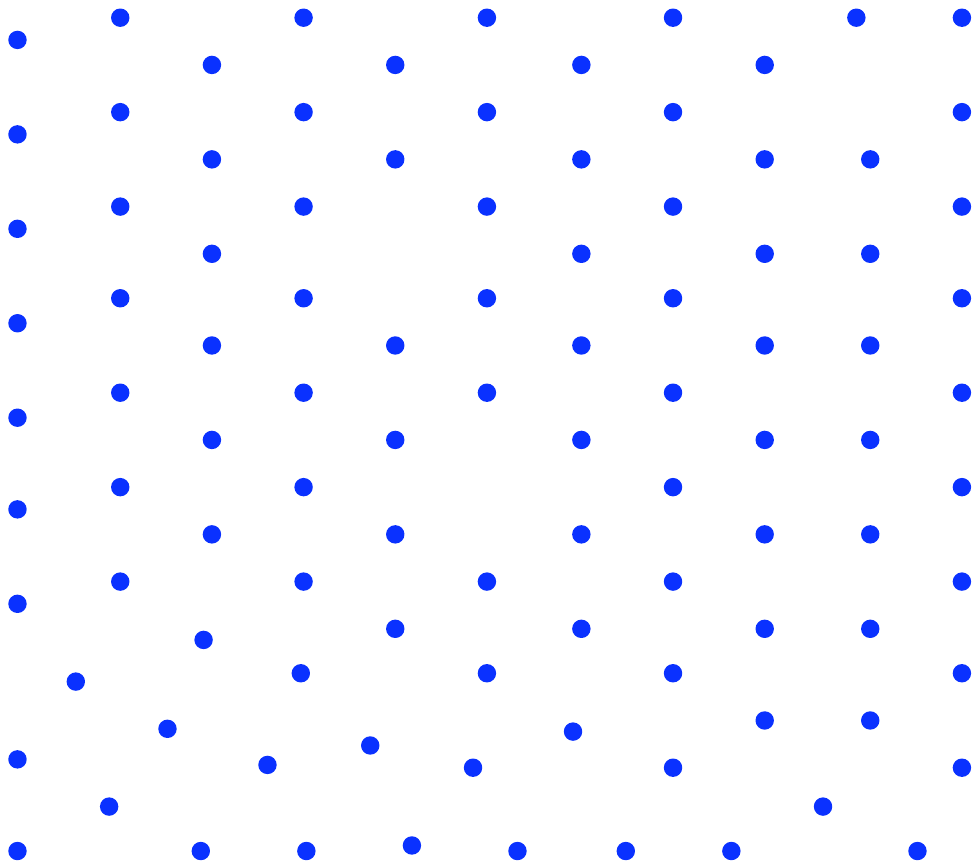}\label{fig:sphere_packing_in_square:2}} \;
\subfigure[$o=.1179656050$]{\includegraphics[scale=0.30]{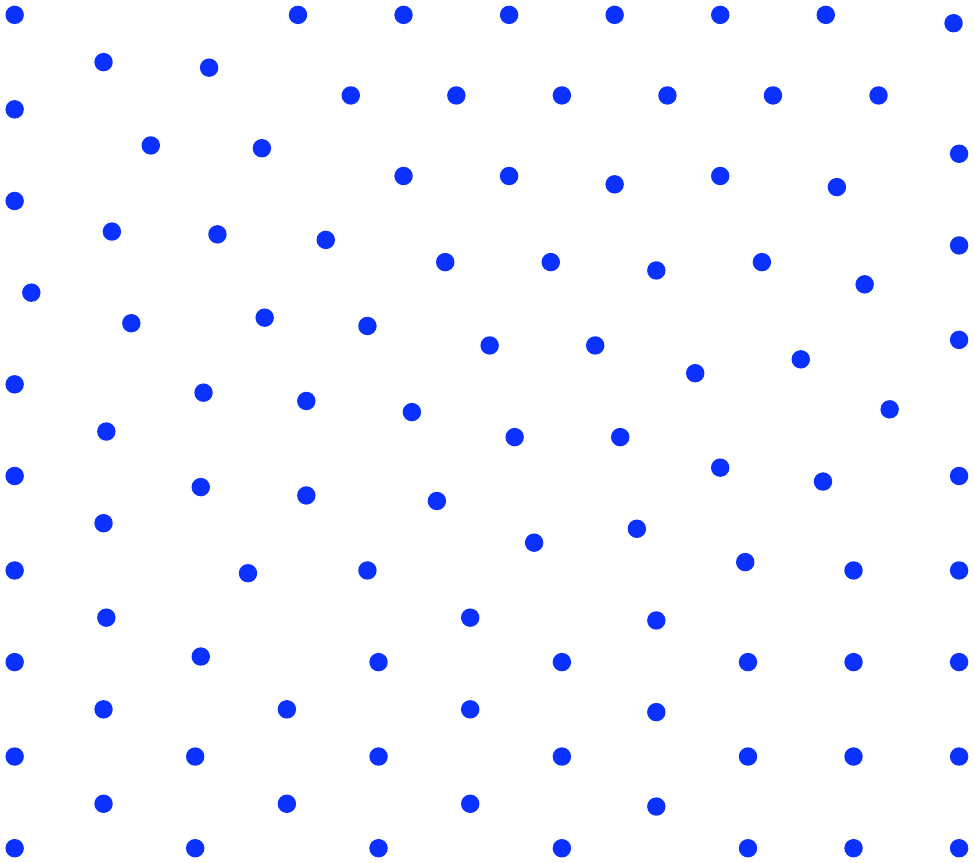}\label{fig:sphere_packing_in_square:3}}

\caption{Local minima obtained by Algorithm~\ref{alg:circ} for
    packing circles into a square, showing final overlap measures for each.\label{fig:sphere_packing_in_square}}
\end{figure}

% \sjwcomment{Should we replace the square-grid solution at left of
%   Figure~\ref{fig:sphere_packing_in_square} with the perfect hexagon?}

\begin{Exa}[Uniform Circles in $\R^2$] \label{ex:circles}
  We ran Algorithm~\ref{alg:circ} with $N=150$ circles, each of area
  $\pi$, and a circular container of size $150\sqrt{12}$. This results
  in a total circle area-to-container area ratio which is equal to the
  optimal packing density. The resulting circle configuration is shown
  in Figure~\ref{fig:circle_packing_circle}. The hexagonal arrangement of the
  circles is clearly visible in the interior of the container.

  We also ran tests in which 100 circles are packed into a square
  container. (Rectangular feasible sets $\Omega_i$ are easily
  incorporated into the formulation by defining bound constraints on
  the centers $c_i$.) We generate starting points by arranging the
  centers in a $10 \times 10$ square lattice. We may then add a random
  perturbation to each center.  Results are shown in
  Figure~\ref{fig:sphere_packing_in_square}.  (For clarity, we show
  only the centers in this figure, omitting the circles.) When no
  perturbations are added to the starting configuration, the algorithm
  does not move from the initial square configuration shown in
  Figure~\ref{fig:sphere_packing_in_square:0}. When random initial
  perturbations are applied (large enough that the original square
  grid structure is not recognizable in the initial point), many
  different local minima are obtained. Three of these are shown in
  Figures~\ref{fig:sphere_packing_in_square:1},
  \ref{fig:sphere_packing_in_square:2}, and
  \ref{fig:sphere_packing_in_square:3}.  Note that all of these have a
  maximum overlap less than the square configuration, and that
  hexagonal structure is recognizable in large parts of the domain,
  with square structure and disorder in intermediate regions.
\end{Exa}

\begin{figure}[!b]
  \centering \subfigure[Distribution of neighbor
  counts.]{\includegraphics[scale=0.40]{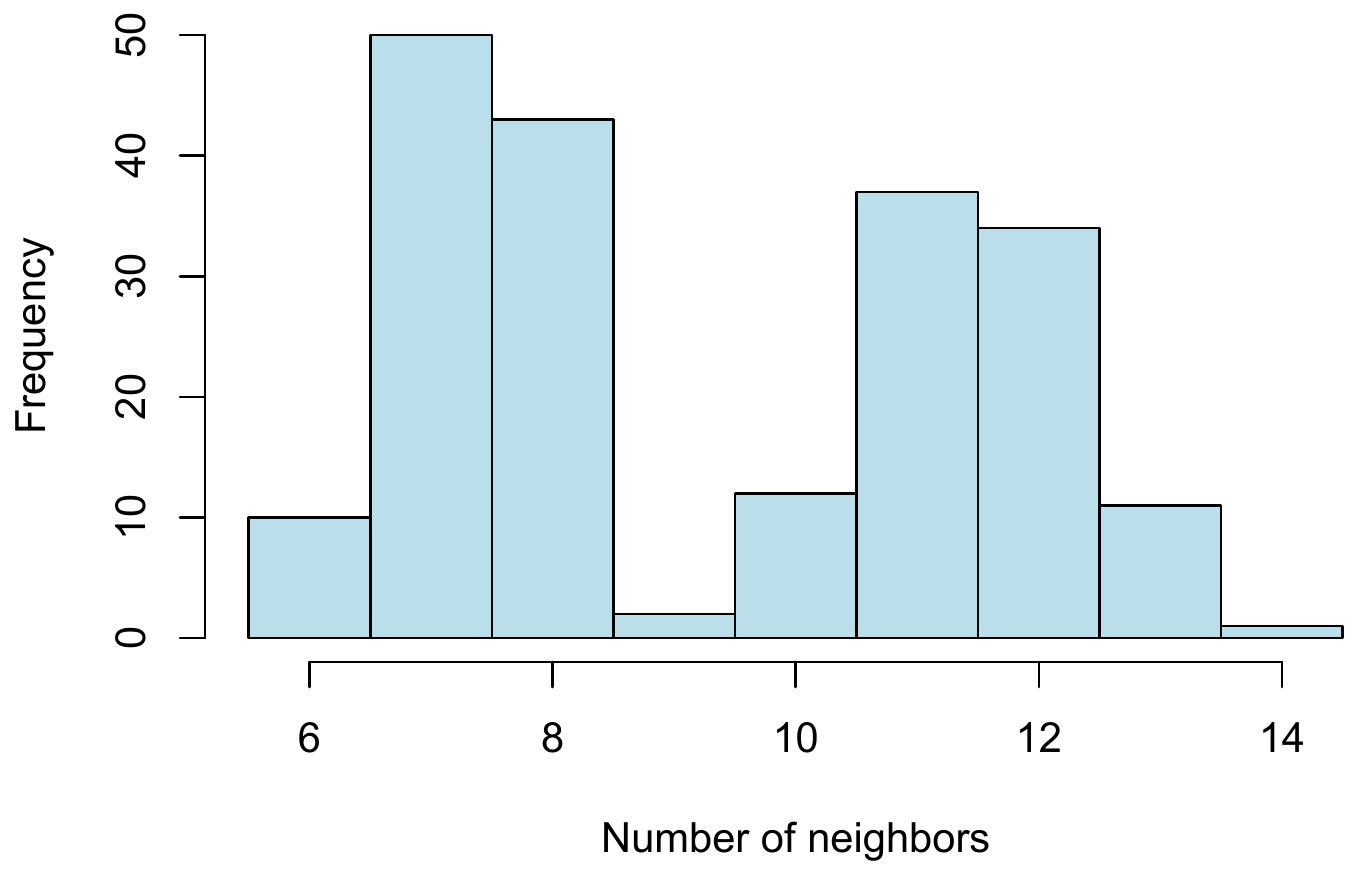}\label{fig:sphere_packing:1}}
  \;
\subfigure[Distribution of neighbor counts,
after spheres at the periphery have been
removed.]{\includegraphics[scale=0.40]{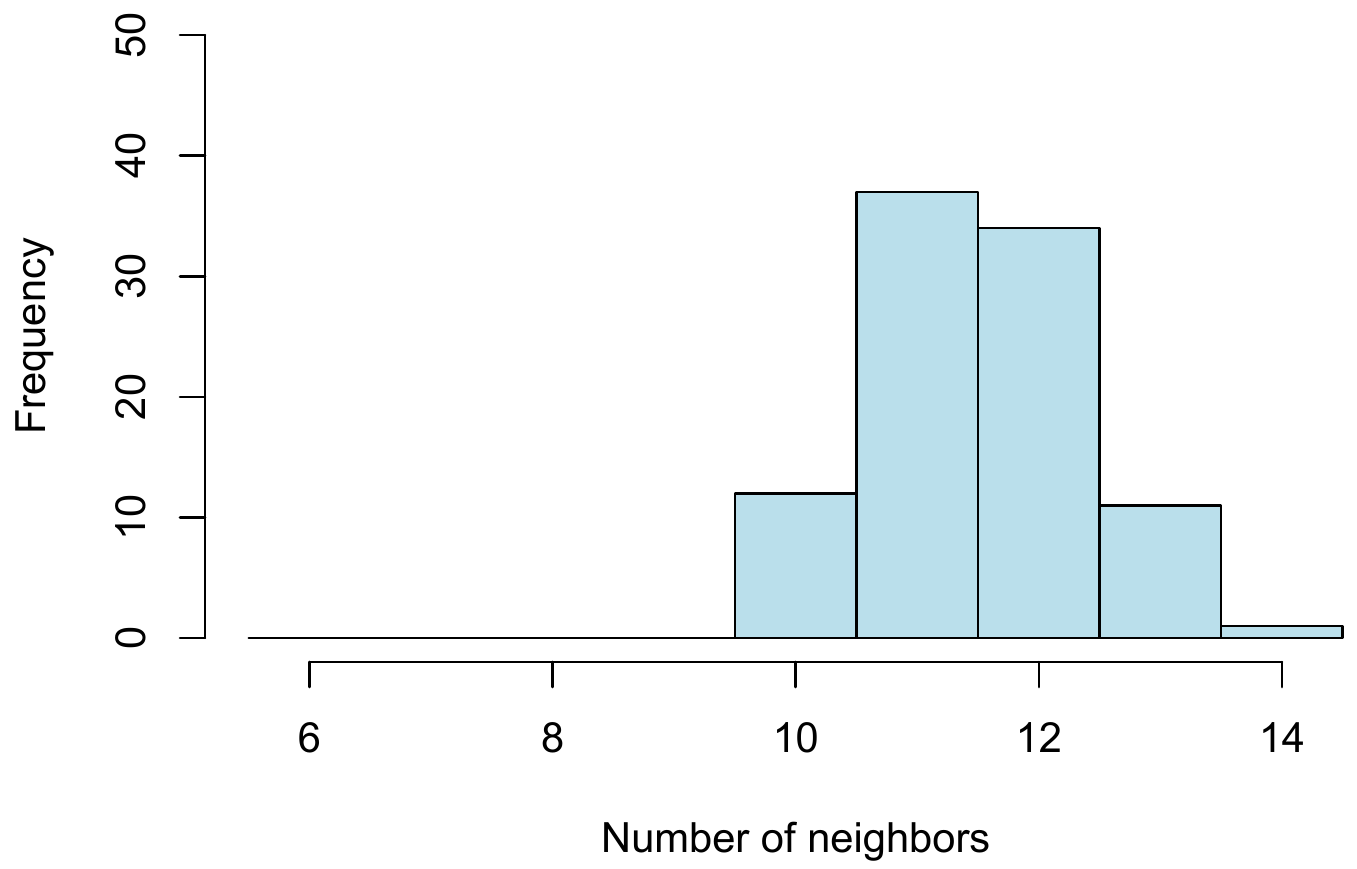}\label{fig:sphere_packing:2}}
\caption{Neighbor counts for packing of 100 three-dimensional spheres in  in a spherical container.}
\label{fig:sphere_packing}
\end{figure}

\begin{Exa}[Uniform spheres in $\R^3$] \label{ex:spheres}
%
% In two dimensions (in infinite Euclidean space) the densest circle
% packing is achieved by the hexagonal lattice and has density
% $\pi/\sqrt{12}$. We ran Algorithm~\ref{alg:circ} with $N=150$ circles,
% each of area $\pi$, and a round container of size $150\sqrt{12}$. This
% results in a total circle area to container area ratio which is equal
% to the optimal packing density. We ran Algorithm~\ref{alg:circ} for 50
% iterations. The resulting circle configuration is shown in Figure
% \ref{fig:sphere_packing} (left). The hexagonal arrangement of the
% circles is clearly visible in the interior of the container.
%
We performed a similar test to Example~\ref{ex:circles} in three
dimensions. 
% The densest sphere packing in Euclidean space $\R^3$ of infinite
% extent is achieved by the FCC lattice and has density
% $\pi/\sqrt{18}$.
We checked to see whether Algorithm~\ref{alg:circ} converges to a
solution like the FCC lattice
%  (known to be optimal in space of infinite extent) 
in a finite minimum-overlap arrangement with 200 spheres enclosed in a
larger sphere. We chose the small spheres to have volume $\pi$ and the
containing sphere to have volume $200\sqrt{18}$, giving a density of
$\pi/\sqrt{18}$, identical to the FCC lattice, which is optimal in
infinite space. At the solution obtained by Algorithm~\ref{alg:circ},
we counted the number of spheres that touch or intersect each
sphere. This statistic provides an indication of the type of packing
attained, since the FCC lattice has 12 neighbors per sphere, while the
BCC lattice has only 8 neighbors per sphere. The histogram for the
number of neighboring spheres is shown in
Figure~\ref{fig:sphere_packing:1}. A more instructive diagram is
obtained by removing from consideration those spheres that touch the
enclosing sphere. After doing so, we obtain the histogram in Figure
\ref{fig:sphere_packing:2}. This figure suggests strongly that the
calculated solution is close to the FCC lattice over most of the
interior region of the domain.
\end{Exa}

Finally, we report on solutions obtained by Algorithm~\ref{alg:circ}
on packings of discs in a circle, for which the optimal packing is
known in only a few cases. In particular, we analyze packings with
equally sized discs, where the number of discs is given by a hexagonal
number, that is, 
\begin{equation} \label{eq:hex}
h(k)=3k(k+1)+1, \qquad k \geq 1.
\end{equation}

\begin{figure}
\centering
\subfigure[$37$ discs]{
\includegraphics[scale=0.22]{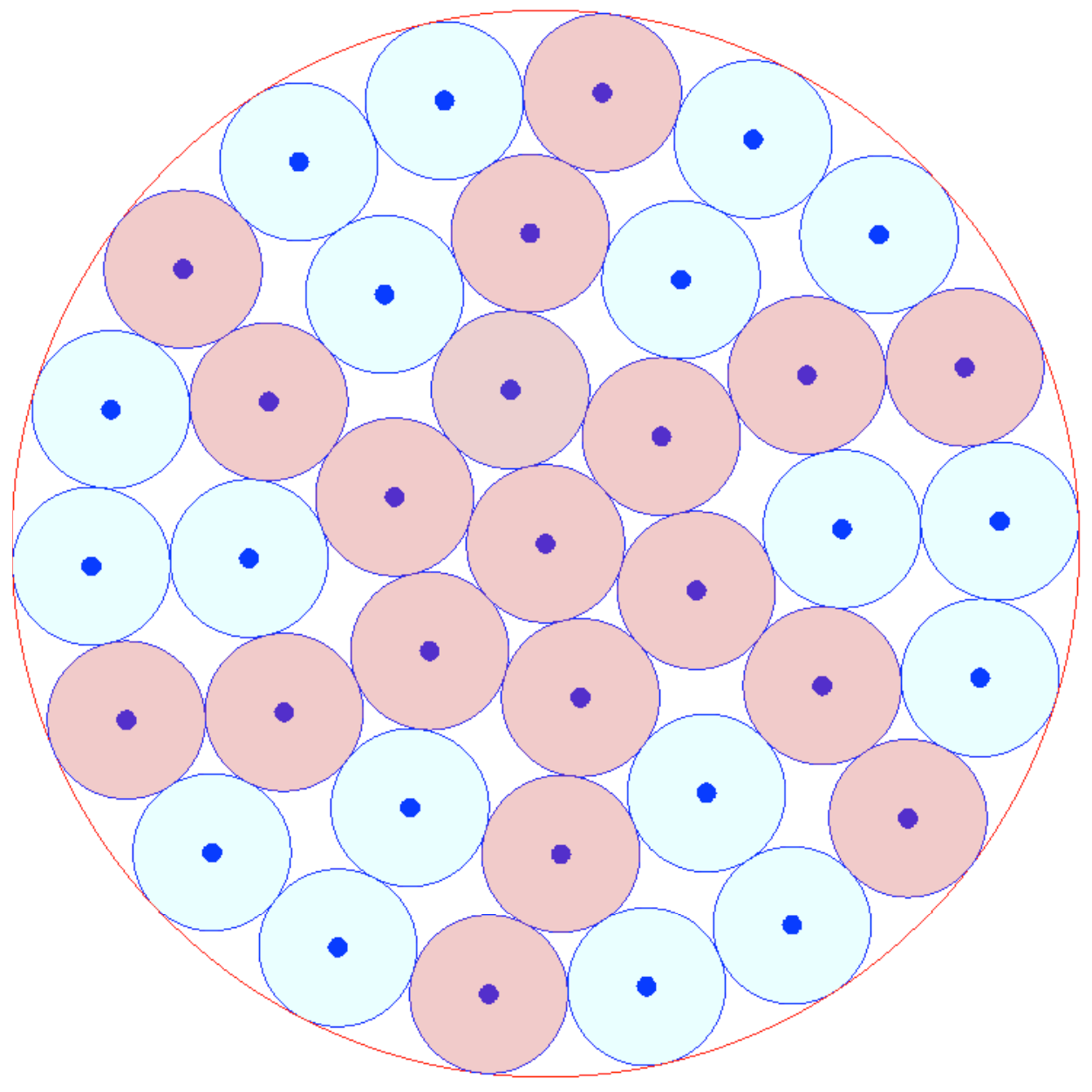}\label{fig:curved:1}} \;
\subfigure[$61$ discs]{
\includegraphics[scale=0.22]{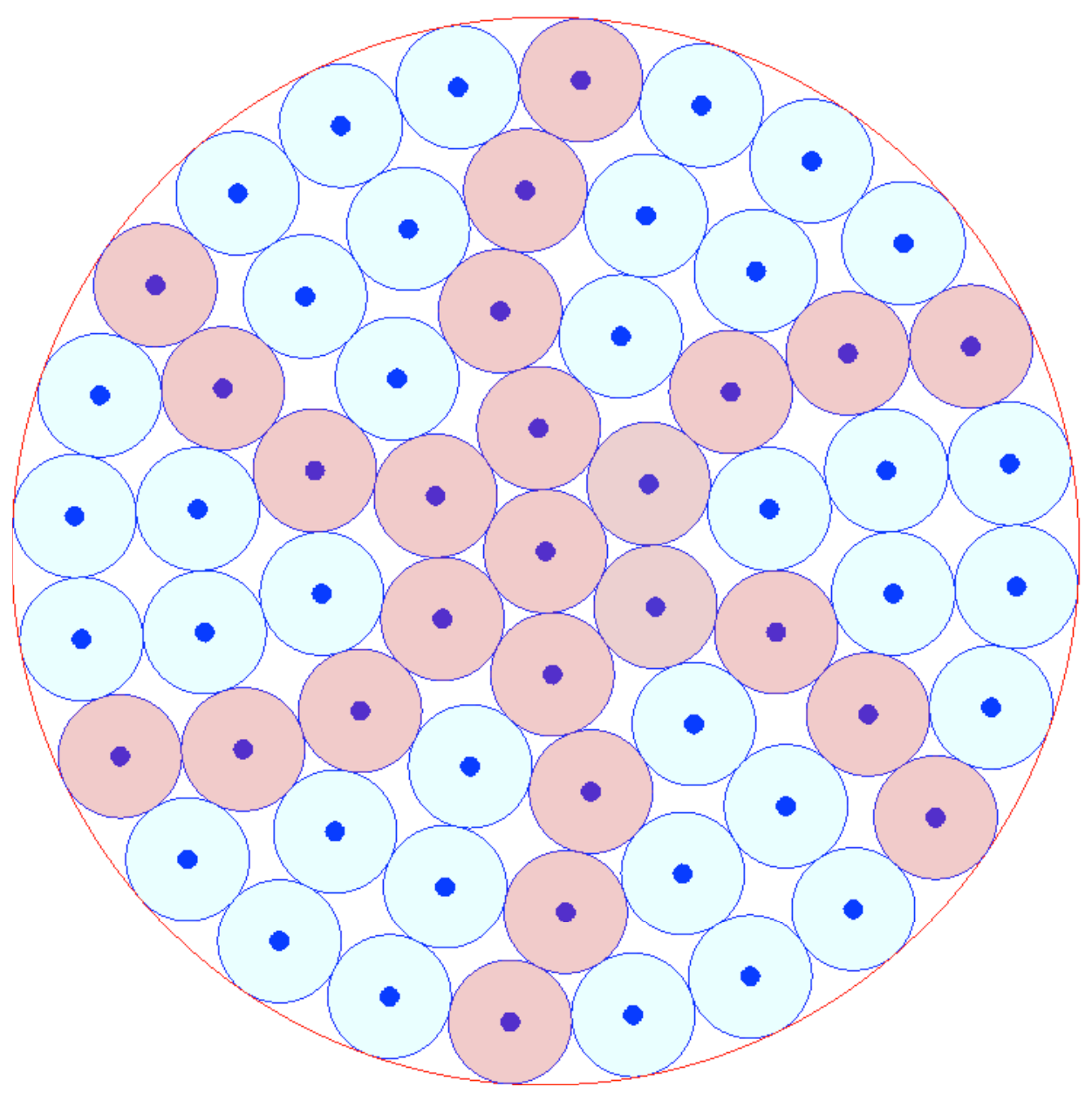}\label{fig:curved:2}} \;
\subfigure[$91$ discs]{
\includegraphics[scale=0.22]{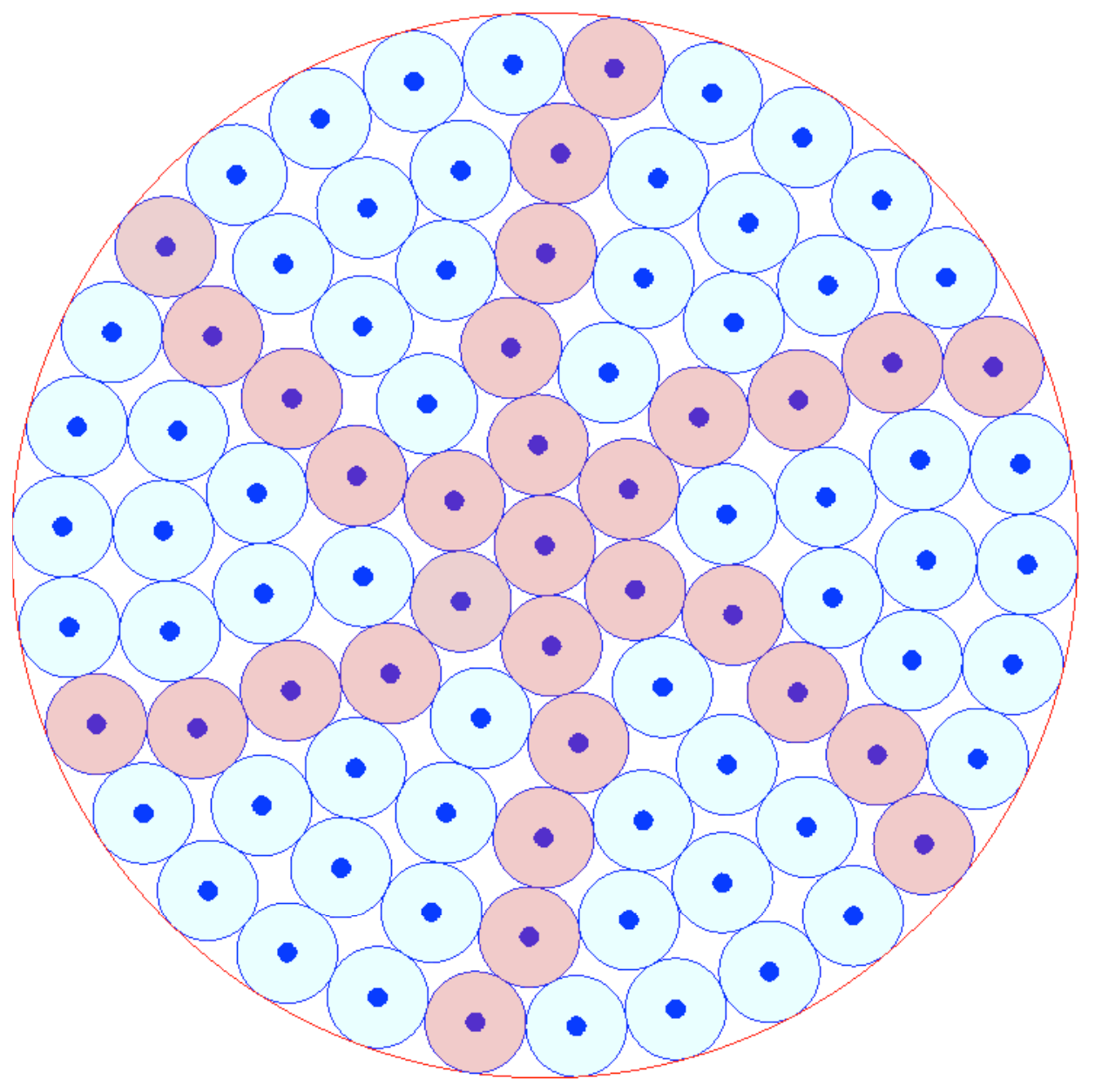}\label{fig:curved:3}}\;
\subfigure[$37$ discs, larger radii]{
\includegraphics[scale=0.22]{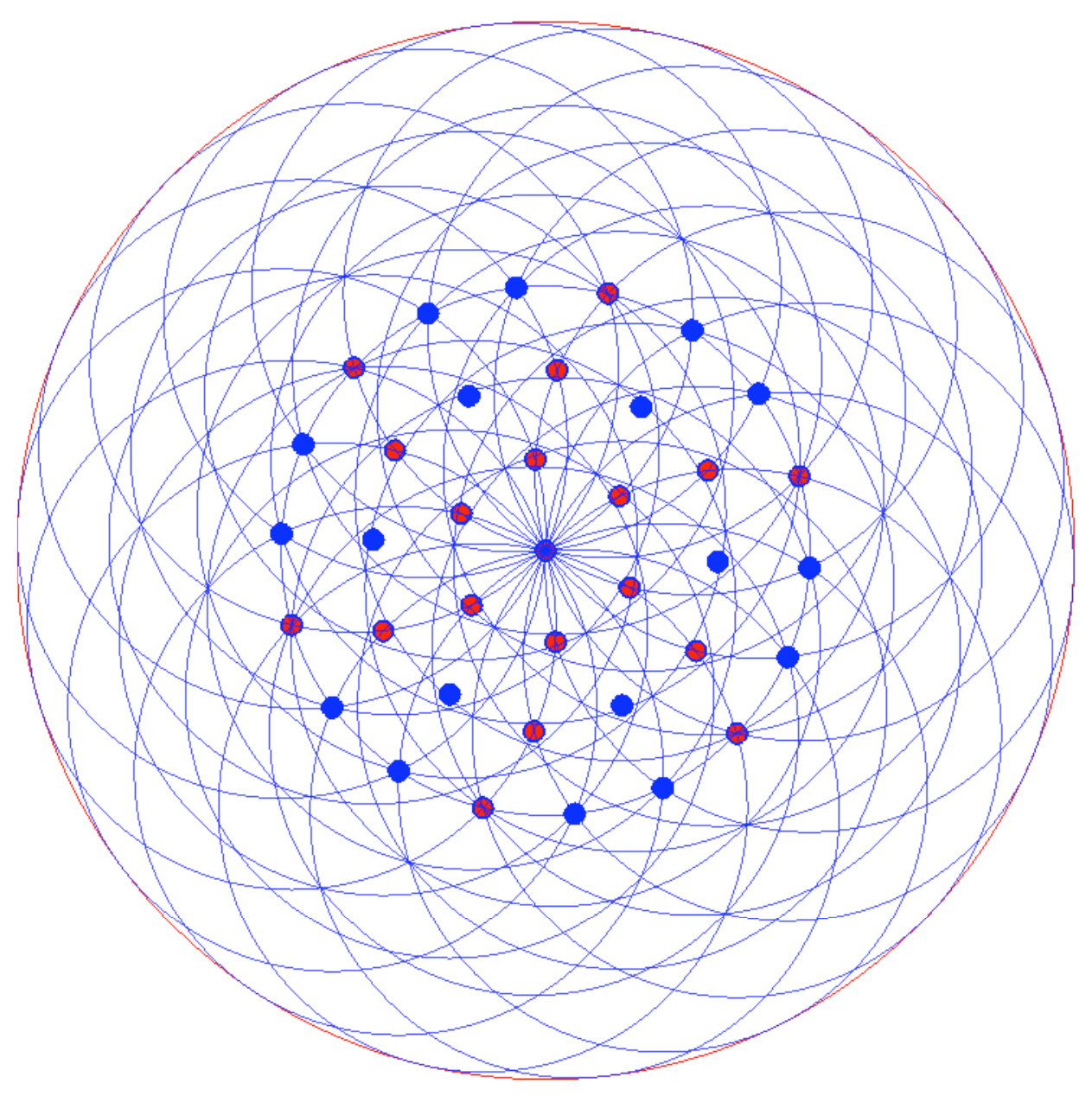}\label{fig:curved:4}}
\caption{Optimal configurations found using
    Algorithm~\ref{alg:circ} for hexagonal numbers of unit discs in an enclosing circle.}
\label{fig:curved}
\end{figure}

\begin{Exa} \label{ex:curved} Lubachevsky and
  Graham~\cite{LubG04a} introduce \emph{curved hexagonal
    packings}, a new family of packings for configurations with a
  hexagonal number of discs. This family contains the best packings
  found so far for $h(k)$ defined by \eqnok{eq:hex}, for $k \leq
  5$. We ran Algorithm~\ref{alg:circ} with the optimal densities found
  in \cite{LubG04a} for $k=3, 4, 5$. The best local optima we
  found are shown in Figure~\ref{fig:curved}; they are identical to
  the configurations found in \cite{LubG04a}. (We highlight some
  of the circles to emphasize the ``curved'' feature of the packing,
  which distinguishes it from a standard hexagonal arrangement, which
  has slightly lower density when restricted to a finite
  circle.)
  % \sjwcomment{Did we actually run the no-overlap problem for these
  %   three examples? Maybe with an earlier version of the code
  %   e.g. with the sum-of-squares objective?}

  When we ran Algorithm~\ref{alg:circ} on the problem of 37 uniform
  discs in a larger disc, where the radii were too large to allow
  packing without overlap, the algorithm with $H(\xi) = \|
  \xi\|_{\infty}$ produced the same arrangement of centers as in
  Figure~\ref{fig:curved:1} (see Figure~\ref{fig:curved:4}) when
  initialized at a sufficiently close initial point.  It is a well
  known property of minimization of the norm $\| \cdot \|_{\infty}$
  that many elements of the argument vector tend to achieve the
  maximum value. In our application, this means that the maximal
  overlap is attained by many pairs of circles. We can obtain
  non-overlapping configurations by simply reducing the radii of all
  discs uniformly, by an amount equal to half the maximal
  overlap. This will yield a solution in which each pair of circles
  that formerly overlapped maximally now just touches.
\end{Exa}

\section{Ellipsoid Packing} \label{sec:ellipsoid}

Here we discuss a bilevel optimization procedure for packing
ellipsoids into an ellipsoidal container in a way that minimizes the
maximum overlap of any pair of ellipsoids. It is not as obvious how to
measure the overlap between two ellipsoids as between two spheres,
since it depends on the orientation of the ellipsoids as well as the
location of their centers.  We measure the overlap by the {\em sum of
  principal semi-axes of the largest ellipsoid that can be inscribed
  in the intersection of the two ellipsoids}. This overlap measure can
be calculated by solving a small semidefinite optimization problem,
constructed according to the S-procedure (see
Subsection~\ref{sec:overlap}). These are the lower-level problems in
our bilevel optimization formulation. The upper-level problem is to
position and orient the ellipsoids so as to minimize the maximum
overlap (see Subsection~\ref{sec:positioning}), while keeping all
ellipsoids inside the enclosing shape. We refer to this problem as
``min-max-overlap.'' Dual information from the lower-level problems
provides a measure of sensitivity of the overlaps to the ellipsoid
parameters, allowing us to develop a successive approximation
approach, with trust regions, whose accumulation points are stationary
for the min-max-overlap problem. Technical results regarding the
trust-region approach and the proof of convergence are given in
Subsection~\ref{sec:tech}.

% An alternative measure of overlap would be the {\em volume} of the
% maximum inscribed ellipsoid, which is proportional to the product
% (rather than the sum) of its semi-axis lengths. The subproblem to be
% solved to calculate this objective is similar to the semidefinite
% programming formulation described below, but its objective is no
% longer linear.

\subsection{Measuring Overlap}
\label{sec:overlap}

% The overlap between two ellipsoids can be measured in different
% ways. We will generalize the approach taken in Section
% \ref{sec:sphere} and approximate the overlap by inscribing an
% ellipsoid into the intersection.
Boyd and Vandenberghe \cite[Section~8.4.2]{BoyV03} consider the
problem of finding the ellipsoid of largest volume inscribed in an
intersection of ellipsoids.  The volume of an ellipsoid $\mathcal{E} =
\{c + S u \mid \norm{u}_2 \leq 1 \}$ is proportional to
$\det(S)$. Although this problem is convex, it is not a semidefinite
program (SDP), because the objective is nonlinear. We thus consider an
alternative in which $\trace (S)$ is used as the objective. The trace
is the sum of lengths of the semi-axes of the ellipsoid, which is a
good proxy for the volume in problems of the type we consider. 
% We have found that use of the trace produces similar computational
% results to $\log \det (S)$, with faster computation time.
Trace maximization admits an SDP formulation of the lower-level
problems, which facilitates theoretical development and analysis of
our min-max-overlap problem.

Recalling from \eqnok{eq:defE} that we define the ellipsoid
$\mathcal{E}_i$ by
\beq \label{eq:Ei}
\mathcal{E}_i := \{ c_i + S_i u \mid \|u\|_2 \le 1 \},
\eeq
we introduce the notation $\Sigma_i = S_i^2$.
% \beq \label{eq:defSig}
% \Sigma_i = S_i^2.
% \eeq
%
Parametrizing the inscribed ellipsoid similarly by $\mathcal{E}_{ij}
:= \{ c_{ij} + S_{ij} u \mid \| u \|_2 \le 1 \}$, and using
\eqnok{eq:LMI} to formulate the fact that the inscribed ellipsoid is
contained in both $\mathcal{E}_i$ and $\mathcal{E}_j$, we formulate
the problem of measuring overlap as follows:
\begin{subequations}
\label{prob:lin_max_volume_ell}
\begin{align}
\hat{O} (c_i,c_j,\Sigma_i,\Sigma_j) := 
 \max_{S_{ij} \succeq 0, c_{ij}, \lambda_{ij1}, \lambda_{ij2}} \;
&\trace (S_{ij}) \nonumber\\
\mbox{subject to} \quad
\label{eq:sdps.1}&\begin{pmatrix} - \lambda_{ij1} I &0&S_{ij}\\0& \lambda_{ij1}-1&(c_{ij}-c_i)^T\\S_{ij} &c_{ij}-c_i&-\Sigma_i\end{pmatrix}\preceq 0, \\
\label{eq:sdps.2}& \begin{pmatrix} - \lambda_{ij2} I &0&S_{ij}\\0& \lambda_{ij2}-1&(c_{ij}-c_j)^T\\S_{ij} &c_{ij}-c_j&-\Sigma_j\end{pmatrix}\preceq 0.
\end{align}
\end{subequations}
The Lagrangian can be written as
\begin{align}
{\cal L} (c, S_{ij},\lambda_{ij1}, \lambda_{ij2}, T_{ij}, M_{ij1}, M_{ij2})  := &
 \langle I, S_{ij}\rangle + \langle T_{ij}, S_{ij} \rangle\nonumber\\ 
& - \langle M_{ij1},
\begin{pmatrix} - \lambda_{ij1} I &0&S_{ij}\\0& \lambda_{ij1}-1&(c_{ij}-c_i)^T\\S_{ij}&c_{ij}-c_i&-\Sigma_i\end{pmatrix}
\rangle \nonumber\\
& - 
\langle M_{ij2},
\begin{pmatrix} - \lambda_{ij2} I &0&S_{ij}\\0& \lambda_{ij2}-1&(c_{ij}-c_j)^T\\S_{ij}&c_{ij}-c_j&-\Sigma_j\end{pmatrix}
\rangle,
\nonumber\end{align}
with the dual problem being derived from
\[
\min_{M_{ij1} \succeq 0, M_{ij2} \succeq 0, T_{ij} \succeq 0} \, \left\{
\max_{S_{ij} \succeq 0, c_{ij}, \lambda_{ij1},\lambda_{ij2}} \,
{\cal L} (c_{ij}, S_{ij},\lambda_{ij1}, \lambda_{ij2}, T_{ij}, M_{ij1}, M_{ij2}) \right\}.
\]
Introducing the following notation for $M_{ij1}$ and $M_{ij2}$:
\begin{equation} \label{def:M}
M_{ij1}:=\begin{pmatrix} R_{ij1} & r_{ij1} & P_{ij1} \\ r_{ij1}^T & p_{ij1} & q_{ij1}^T\\ P_{ij1} & q_{ij1} & Q_{ij1}\end{pmatrix}, \qquad 
M_{ij2}:=\begin{pmatrix} R_{ij2} & r_{ij2} & P_{ij2} \\ r_{ij2}^T & p_{ij2} & q_{ij2}^T\\ P_{ij2} & q_{ij2} & Q_{ij2}\end{pmatrix},
\end{equation}
we can write the dual explicitly as follows:
\begin{align}
\label{prob:dual_lin_max_volume_ell}
\hat{O} (c_i,c_j,\Sigma_i,\Sigma_j) := 
\nonumber
\min_{M_{ij1} \succeq 0, M_{ij2} \succeq 0, T_{ij}\succeq 0} \qquad&
p_{ij1}+p_{ij2}+2q_{ij1}^Tc_i+2q_{ij2}^Tc_j \\
\nonumber
& \qquad +\langle Q_{ij1},\Sigma_{i}\rangle+\langle Q_{ij2},\Sigma_{j}\rangle \nonumber\\ &\nonumber \\
\mbox{subject to} \qquad &0=I+T_{ij}-2P_{ij1}- 2P_{ij2}\\
&0=\trace(R_{ij1})-p_{ij1}\nonumber\\ &0=\trace(R_{ij2})-p_{ij2}\nonumber\\  &0=q_{ij1}+q_{ij2}. \nonumber
\end{align}
(We have assumed without loss of generality that $P_{ij1}$ and
$P_{ij2}$ are in $S \R^{n \times n}$; this follows from $S_{ij} \in S
\R^{n \times n}$.)

When the ellipsoids ${\cal E}_i$ and ${\cal E}_j$ overlap, strong
duality holds for this primal-dual pair of semidefinite programs
since, as we now verify, both problems have a strictly feasible point.
For \eqnok{prob:lin_max_volume_ell} we know that there exists an
ellipsoid with positive volume that is {\em strictly} inscribed in the
intersection. By setting $c_{ij}$ and $S_{ij}$ to be the parameters of
this ellipsoid (with $S_{ij} \succ 0$), the S-procedure for strict
inequalities can be applied to show that (strict) definiteness holds
in \eqnok{eq:sdps.1} and \eqnok{eq:sdps.2}. This fact establishes
strict feasibility of \eqnok{prob:lin_max_volume_ell}. For the dual
(\ref{prob:dual_lin_max_volume_ell}), we can set $T_{ij}=I$ and define
\[
M_{ij1} = M_{ij2} = \begin{pmatrix}  I &0& \frac{1}{2}I\\0& n &0\\ \frac{1}{2}I&0&I\end{pmatrix}.
\]
It is easy to verify that these choices satisfy the constraints in
(\ref{prob:dual_lin_max_volume_ell}), along with the (strict)
interiority conditions $M_{ij1} \succ 0$, $M_{ij2} \succ 0$, $T_{ij}
\succ 0$.  This observation of strong duality justifies our use of the
notation $\hat{O} (c_i,c_j,\Sigma_i,\Sigma_j)$ to denote the optimal
objectives of both primal and dual.

%If ${\cal E}_i$ and ${\cal E}_j$ touch but do not overlap, the
%solution to \eqnok{eq:sdps} is evidently obtained by setting $c_{ij}$
%to be the point of contact, and $B_{ij}=0$, giving an objective of
%zero. \sjwcomment{Can we construct a solution of the dual that also
% achieves a zero objective?}  When ${\cal E}_i$ and ${\cal E}_j$ do
%not touch or overlap, \eqnok{eq:sdps} is infeasible. \sjwcomment{Can
 % we thus infer unboundedness of \eqnok{eq:ov3}?}

\subsection{Choosing Ellipse Positions and Orientations}
\label{sec:positioning}

The main variables in the min-max-overlap problem are the parameters
defining the ellipses $\mathcal{E}_i$ for $i=1,2,\dotsc,N$: the
centers $c_i$ and the orientations defined by $S_i$ (and thus
$\Sigma_i = S_i^2$).  For $n=3$ (which we assume in this section and
subsequently), we would like to fix the lengths of the axes of each
ellipsoid at the values $r_{i1}$, $r_{i2}$, and $r_{i3}$ (assuming
that $r_{i1}\geq r_{i2}\geq r_{i3}$). This is equivalent to fixing the
eigenvalues of $\Sigma_i$ at $r_{i1}^2$, $r_{i2}^2$, and $r_{i3}^2$,
or to fixing the eigenvalues of $S_i$ to $r_{i1}$, $r_{i2}$, and
$r_{i3}$.

Using the notation $\hat{O}$ defined in
\eqnok{prob:lin_max_volume_ell} and
\eqnok{prob:dual_lin_max_volume_ell}, we can formulate the
min-max-overlap problem as the following bilevel optimization problem:
\begin{subequations}
\label{prob:non_convex_ell}
\begin{align}
\min_{\xi, (c_i,S_i,\Sigma_i), i=1,2,\dotsc,N} \;&
\xi&\\
\mbox{subject to} \quad&\xi \geq \hat{O}(c_i, c_j, \Sigma_i, \Sigma_j),& 1\leq i<j\leq N, \label{prob:non_convex_ell1}\\ 
\label{prob:non_convex_ell2} &\mathcal{E}_i\subset \mathcal{E}, & i=1,2,\dots,N, \\ 
\label{prob:non_convex_ell4} &\Sigma_i=S_i^2, \\
\label{prob:non_convex_ell3} &\mbox{semi-axes of }\mathcal{E}_i \mbox{ have lengths } r_{i1}, r_{i2}, r_{i3},  & i=1,2,\dots,N.
\end{align}
\end{subequations}
This problem is nonconvex for three reasons. First, each pairwise
overlap objective $\hat{O}(c_i,c_j,\Sigma_i,\Sigma_j)$ is a nonconvex
function of its arguments. 
% Second, the constraint
% \eqnok{prob:non_convex_ell3} is a nonconvex function of $\Sigma_i$.
This issue is intrinsic; as in the sphere-packing problem, we expect
there to be many local solutions and we can only expect our algorithm
to find one of them. As we see below (in \eqnok{prob:ellipsoids_upper}
and Algorithm~\ref{alg:xxx}), our algorithm iteratively solves
subproblems in which each $\hat{O}$ is replaced by a linearized
approximation that makes use of the optimal dual variables $M_{ij1}$ and
$M_{ij2}$ from the formulation \eqnok{prob:dual_lin_max_volume_ell}.
These subproblems will be convex if we can overcome the other two
sources of nonconvexity in the formulation
\eqnok{prob:non_convex_ell}, which we address now.

The second nonconvexity issue is in the constraint
\eqnok{prob:non_convex_ell3} on the eigenvalues of $S_i$,
$i=1,2,\dotsc,N$. We consider instead the following convex relaxation:
\begin{equation}
\label{constraints_ellipsoid}
S_i-r_{i1} I\preceq 0, \quad S_i-r_{i3} I\succeq 0, \quad 
\trace(S_i)=r_{i1}+r_{i2}+r_{i3}.
\end{equation}
Note that this is indeed a relaxation: If $\mathcal{E}_i$ has the
desired dimensions, then the eigenvalues of $S_i$ are $r_{i1}$,
$r_{i2}$, and $r_{i3}$, and the conditions
\eqnok{constraints_ellipsoid} are satisfied.  Because the overall
goal is to minimize maximal overlap, and because minimum-volume
ellipsoids are those that are most eccentric, the relaxation
\eqnok{constraints_ellipsoid} is usually ``tight'' with respect to
\eqnok{prob:non_convex_ell} in many interesting cases.
% That is, at the optimum of
% \eqnok{prob:non_convex_ell_relax}, the matrices $S_i$ have eigenvalues
% $r_{i1}$, $r_{i2}$, and $r_{i3}$, for most $i$. 
% Indeed, such is observed to be the case in our computational
% experiments, although 
Intermediate iterates are often observed to have ellipsoids less
eccentric than desired.

The third source of nonconvexity --- the constraint
\eqnok{prob:non_convex_ell4} --- is relatively easy to deal with. We
replace it with the following pair of restrictions:
\beq \label{eq:SSig} 
\left[ \begin{matrix} \Sigma_i & S_i \\ S_i &
    I \end{matrix} \right] \succeq 0, \qquad S_i \succeq 0, \qquad
i=1,2,\dotsc,N.  
\eeq
The first of these conditions ensures only that $\Sigma_i \succeq
S_i^2$. However,
% note that the objective in the problem
% \eqnok{prob:non_convex_ell} aims to minimize overlaps between pairs of
% ellipses, and that (from definition
% \eqnok{prob:dual_lin_max_volume_ell}), 
the overlap $\hat{O}(c_i,c_j,\Sigma_i,\Sigma_j)$ will grow if
$\Sigma_i$ is replaced by any matrix $\tilde{\Sigma}_i \succeq
\Sigma_i$. Hence, because of the min-max-overlap objective in
\eqnok{prob:non_convex_ell}, the matrices $\Sigma_i$ will be set to
the ``smallest possible values'' that satisfy the conditions
\eqnok{eq:SSig}, that is, $\Sigma_i = S_i^2$.

Finally, defining the containing ellipse to be $\mathcal{E} := \{ x
\mid (x-c)^T \Sigma^{-1} (x-c) \le 1 \}$, we can use \eqnok{eq:LMI} to
formulate the condition \eqnok{prob:non_convex_ell2} as follows:
\beq \label{eq:EiE} 
\left[ \begin{matrix} - \lambda_i I &0&S_i\\
0& \lambda_i-1&(c_i-c)^T\\ 
S_i &c_i-c&-\Sigma \end{matrix} \right] \preceq 0,
\eeq

Given all these considerations, we define the relaxed version of
\eqnok{prob:non_convex_ell} to be addressed in this section as
follows:
\begin{subequations}
\label{prob:non_convex_ell_relax}
\begin{align}
\min_{\xi, (\lambda_i,c_i,S_i \Sigma_i), i=1,2,\dotsc,N} \;\;\;&
\xi\\
\mbox{subject to} \qquad &\xi\geq \hat{O}(c_i, c_j, \Sigma_i, \Sigma_j),& 1\leq i<j\leq N,\label{prob:non_convex_ell_relax1}\\ 
\label{prob:non_convex_ell_relax2} 
&\left[ \begin{matrix} - \lambda_i I &0&S_i\\
0& \lambda_i-1&(c_i-c)^T\\ 
S_i &c_i-c&-\Sigma \end{matrix} \right] \preceq 0, & i=1,2,\dotsc,N, \\
\label{prob:non_convex_ell_relax2a}
&\left[ \begin{matrix} \Sigma_i & S_i \\ S_i & I \end{matrix} \right] \succeq 0, & i=1,2,\dotsc,N, \\
\label{prob:non_convex_ell_relax3}
&S_i-r_{i1} I\preceq 0, \quad S_i-r_{i3} I\succeq 0,  & i=1,2,\dots,N, \\
\label{prob:non_convex_ell_relax4}
& \trace(S_i)=r_{i1}+r_{i2}+r_{i3}, & i=1,2,\dots,N.
\end{align}
\end{subequations}
Note that when the ellipse $\mathcal{E}_i$ is actually a circle, that
is, $r_{i1}=r_{i2}=r_{i3}$, we can fix $S_i = r_{i1} I$ and $\Sigma_i
= r_{i1}^2 I$ in \eqnok{prob:ellipsoids_upper}, and eliminate these
variables from that problem. Hence, we can assume without loss of
generality that $r_{i1}>r_{i3}$.

In the remainder of this subsection, we describe our algorithm for
solving the bilevel optimization problem
\eqnok{prob:non_convex_ell_relax}, and prove convergence
properties. Our development and analysis takes place in a general
setting that encompasses \eqnok{prob:non_convex_ell_relax} but uses
simpler notation. A key step in the algorithm is the solution of a
subproblem for \eqnok{prob:non_convex_ell_relax} in which the
objective is linearized using the optimal values from the dual overlap
formulation \eqnok{prob:dual_lin_max_volume_ell}. The other
constraints in \eqnok{prob:non_convex_ell_relax} remain the same, and
a trust region is added to restrict the amount by which the ellipsoid
parameters can change. This subproblem can be stated as follows:
\begin{subequations}
\label{prob:ellipsoids_upper}
\begin{align}
\label{prob:ellipsoids_upper.1}
\min_{\xi, (\lambda_i, c_i, S_i, \Sigma_i), i=1,2,\dotsc,N} \;\;\;
&\xi\\
\nonumber
\mbox{\rm subject to} \qquad & \xi\geq p_{ij1}+p_{ij2}+2q_{ij1}^Tc_i+2q_{ij2}^Tc_j\nonumber\\ 
\label{prob:ellipsoids_upper.2}
&\qquad\,\,+\langle Q_{ij1},\Sigma_{i}\rangle+\langle Q_{ij2},\Sigma_{j}\rangle, & \mbox{for }  (i, j) \in {\cal I},  \\
\label{prob:ellipsoids_upper.3}
&\left[ \begin{matrix} - \lambda_i I &0&S_i\\
0& \lambda_i-1&(c_i-c)^T\\ 
S_i &c_i-c&-\Sigma \end{matrix} \right] \preceq 0, & i=1,2,\dotsc,N, \\
\label{prob:ellipsoids_upper.3a}
&\left[ \begin{matrix} \Sigma_i & S_i \\ S_i & I \end{matrix} \right] \succeq 0,  & i=1,2,\dotsc,N, \\
\label{prob:ellipsoids_upper.4}
&S_i-r_{i1} I \preceq 0, \quad S_i-r_{i3} I\succeq 0, &  i=1,2,\dots,N,\\
\label{prob:ellipsoids_upper.5}
 &\trace(S_i)=r_{i1}+r_{i2}+r_{i3}, & i=1,2,\dots,N,\\
\label{prob:ellipsoids_upper.6}
&\norm{c_i-c_i^-}_2^2 \leq \Delta_c^2, &  i=1,2,\dots,N, \\
\label{prob:ellipsoids_upper.7}
&\norm{S_i-S_i^-} \leq \Delta_S, &  i=1,2,\dots,N, \\
\label{prob:ellipsoids_upper.8}
&|\lambda_i-\lambda_i^-| \leq \Delta_{\lambda}, &  i=1,2,\dots,N.
\end{align}
\end{subequations}
Here, $(\lambda_i^-,c_i^-,S_i^-,\Sigma_i^-)$ are the values of the
variables at the current iteration, and $\Delta_c$, $\Delta_S$, and
$\Delta_{\lambda}$ are trust-region radii. The quantities $p_{ij1}$,
$p_{ij2}$, $q_{ij1}$, $q_{ij2}$, $Q_{ij1}$, and $Q_{ij2}$ are
extracted from the dual solutions $M_{ij1}$ and $M_{ij2}$ of
\eqnok{prob:dual_lin_max_volume_ell} according to the structure
\eqnok{def:M}.  The set ${\cal I}$ represents a {\em subset} of all
possible pairs $(i,j)$ for $1 \le i < j \le N$, representing some
selection of ellipses which currently have a (strict) overlap. Further
details on the choice of ${\cal I}$ are given in
Subsection~\ref{sec:alg}.

We claim first that, if it is possible to fit each ellipsoid
$\mathcal{E}_i$ {\em strictly} inside the enclosing ellipsoid
$\mathcal{E}$, the subproblem \eqnok{prob:ellipsoids_upper} satisfies a
Slater condition. That is, there exists a point that satisfies the
linear equality constraints and {\em strictly} satisfies the
inequality constraints in this problem. To justify this claim, we
first show that it is possible to find a point
$(\bar{\lambda_i},\bar{c}_i, \bar{S}_i, \bar{\Sigma}_i)$ that
satisfies the following conditions:
\begin{subequations}
\label{prob:slater}
\begin{align}
\label{prob:slater.3}
\left[ \begin{matrix} - \bar{\lambda}_i I &0&\bar{S}_i\\
0& \bar{\lambda}_i-1&(\bar{c}_i-c)^T\\ 
\bar{S}_i &\bar{c}_i-c&-\Sigma \end{matrix} \right] &\prec 0, && i=1,2,\dotsc,N, \\
\label{prob:slater.3a}
\left[ \begin{matrix} \bar{\Sigma}_i & \bar{S}_i \\ \bar{S}_i & I \end{matrix} \right] &\succ 0,  && i=1,2,\dotsc,N, \\
\label{prob:slater.4}
\bar{S}_i-r_{i1} I \prec 0, \quad \bar{S}_i-r_{i3} I&\succ 0, && i=1,2,\dots,N,\\
\label{prob:slater.5}
\trace(S_i)=r_{i1}+r_{i2}+r_{i3},& &&  i=1,2,\dots,N.
\end{align}
\end{subequations}
First, choosing $\bar{c}_i=c$ in \eqnok{prob:slater.3}, and orienting
ellipsoid $\mathcal{E}_i$ appropriately, we can find
$\bar{\lambda}_i>0$ such that \eqnok{prob:slater.3} is satisfied. This
remains true if we perturb $\bar{S}_i$ slightly so that its spectrum
lies in the {\em open} interval $(r_{i3},r_{i1})$ while still
satisfying the trace condition \eqnok{prob:slater.5}. This perturbed
$\bar{S}_i$ thus satisfies the conditions
\eqnok{prob:slater.4}. Second, we can simply define
$\bar{\Sigma}_i = \sigma_i I$ for large enough $\sigma_i>0$ to ensure
that \eqnok{prob:slater.3a} is satisfied strictly.

Next, note that from the current iteration, we have a point
$(\lambda_i^-,c_i^-,S_i^-,\Sigma_i^-)$ that satisfies the constraints of
\eqnok{prob:non_convex_ell_relax}, that is,
\begin{subequations}
\label{prob:previt}
\begin{align}
\label{prob:previt.2} 
\left[ \begin{matrix} - \lambda_i^- I &0&S_i^-\\
0& \lambda_i^--1&(c_i^--c)^T\\ 
S_i^- &c_i^--c&-\Sigma \end{matrix} \right] &\preceq 0, && i=1,2,\dotsc,N, \\
\label{prob:previt.2a}
\left[ \begin{matrix} \Sigma_i^- & S_i^- \\ S_i^- & I \end{matrix} \right] &\succeq 0, 
&& i=1,2,\dotsc,N, \\
\label{prob:previt.3}
S_i^--r_{i1} I\preceq 0, \quad S_i^--r_{i3} I&\succeq 0,  && i=1,2,\dots,N, \\
\label{prob:previt.4}
\trace(S_i^-)=r_{i1}+r_{i2}+r_{i3},& && i=1,2,\dots,N.
\end{align}
\end{subequations}
It is now easy to check that for sufficiently small $\epsilon>0$, the point
\[
(\lambda_i(\epsilon),c_i(\epsilon), S_i(\epsilon),\Sigma_i(\epsilon)) :=
(1-\epsilon)   (\lambda_i^-,c_i^-,S_i^-,\Sigma_i^-) +
\epsilon  (\bar{\lambda_i},\bar{c}_i, \bar{S}_i, \bar{\Sigma}_i)
\]
satisfies the inequalities \eqnok{prob:ellipsoids_upper.3},
\eqnok{prob:ellipsoids_upper.3a}, and \eqnok{prob:ellipsoids_upper.4}
strictly, satisfies the linear constraint
\eqnok{prob:ellipsoids_upper.5}, and satisfies the trust-region
constraints \eqnok{prob:ellipsoids_upper.6},
\eqnok{prob:ellipsoids_upper.7}, and \eqnok{prob:ellipsoids_upper.8}
strictly. Since we can choose $\xi$ arbitrarily large to strictly
satisfy \eqnok{prob:ellipsoids_upper.2}, we conclude that there exists
a Slater point for \eqnok{prob:ellipsoids_upper}.

\subsection{Technical Results} \label{sec:tech}

We prove here some technical results that provide the basis for
convergence of the trust-region strategy. To simplify the notation, we
note that each dual overlap problem
\eqnok{prob:dual_lin_max_volume_ell} has the general form
\begin{subequations} 
\label{eq:50}
\begin{align}
P(l,C): \qquad
t^*_l(C) := &\min_{M_l}   \;\langle C,M_l \rangle \\
 &\,\,\mbox{s.t.} \,\,\langle A_{l,i},M_l \rangle =b_{l,i}, \;\; i=1,2,\dotsc,p_l, 
\;\; M_l \succeq 0.
\end{align}
\end{subequations}
Here $C$ captures the parameters that describe all the ellipses, and
$M_l$ is the dual variable for the overlap problem. We assume that $C$
lies in a set $\Omega$ of the following form:
\beq \label{eq:COm}
\Omega := \bar{\Omega} \cap \{ C \, : \, \langle B_k,C \rangle = b_k, \ k=1,2,\dotsc,p \},
\eeq
where $\bar{\Omega}$ is closed, convex, bounded, with nonempty
interior. We now verify formally that the ellipse parameters
satisfying the constraints in \eqnok{prob:non_convex_ell_relax} can be
expressed in the form \eqnok{eq:COm}. We define $C$ to be a
block-diagonal matrix with $N$ blocks of the form:
\beq \label{eq:Cblock}
\left[ \begin{matrix} \Sigma_i & c_i \\ c_i^T & 1 \end{matrix} \right], \quad
i=1,2,\dotsc,N,
\eeq
and define $\bar{\Omega}$ to be the set of all symmetric matrices of
this form for which there exist $\lambda_i$ and $S_i$ such that each
tuple $(\lambda_i,c_i,S_i, \Sigma_i)$ satisfies the conditions
\eqnok{prob:non_convex_ell_relax2},
\eqnok{prob:non_convex_ell_relax2a}, and
\eqnok{prob:non_convex_ell_relax3}.  Boundedness of  $c_i$
is obvious from the containment condition $\mathcal{E}_i
\subset \mathcal{E}$; boundedness of $S_i$ follows from
\eqnok{prob:non_convex_ell_relax3}; while
\eqnok{prob:non_convex_ell_relax2} implies that $\lambda_i \in
[0,1]$. Boundedness of $\Sigma_i$ is not guaranteed by the constraints
in \eqnok{prob:non_convex_ell_relax}. We could, however, add the
constraint $\Sigma_i \preceq r_{i1}^2 I$ without changing the solution
of the problem, thus completing the verification of boundedness of
$\bar{\Omega}$. (For simplicity, however, we do not put this explicit
bound on $\Sigma_i$ in our discussion below.) Closedness and convexity
are immediate consequence of the form of the constraints
\eqnok{prob:non_convex_ell_relax2},
\eqnok{prob:non_convex_ell_relax2a}, and
\eqnok{prob:non_convex_ell_relax3}. To verify nonemptiness of the
interior of $\bar{\Omega}$, recall the discussion following
\eqnok{prob:non_convex_ell_relax}, where we noted that variable $S_i$
can be eliminated from the formulation if ellipsoid $\mathcal{E}_i$ is
in fact a circle. Thus, we can assume without loss of generality that
$r_{i1}>r_{i3}$ for all $i$, and hence, from the discussion
surrounding \eqnok{prob:slater}, we conclude that the set of tuples
$(\lambda_i,c_i,S_i,\Sigma_i)$ allowed by constraints
\eqnok{prob:non_convex_ell_relax2},
\eqnok{prob:non_convex_ell_relax2a}, and
\eqnok{prob:non_convex_ell_relax3} has nonempty interior.  The
structural features of $C$ (the diagonal element $1$ in \eqnok{eq:Cblock} and
the off-diagonal zeros) can in principle be enforced by affine
constraints of the form given in \eqnok{eq:COm}. The constraints
\eqnok{prob:non_convex_ell_relax4} can also be enforced in this way.

Following the notation of \eqnok{prob:slater}, we denote by $\bar{C}$
the point that satisfies
\beq \label{eq:Cbar} 
\bar{C} \in \inT \bar{\Omega} \quad \mbox{and}
\quad \langle B_k,\bar{C} \rangle = b_k, \ k=1,2,\dotsc,p.  
\eeq
We denote by $M_l(C)$ an optimal value of $M_l$ in \eqnok{eq:50} (not
necessarily unique).

The primal form \eqnok{prob:lin_max_volume_ell} of the overlap problem
\eqnok{eq:50} has the form
\begin{equation} \label{eq:50a}
\max_{\zeta_l=(\zeta_{l,1},\zeta_{l,2}, \dotsc, \zeta_{l,p_l})} \, b_l^T\zeta_l \quad \mbox{s.t.} \quad 
C-\sum_{i=1}^{p_l} \zeta_{l,i} A_{l,i}  \succeq 0.
\end{equation}
As discussed in Subsection~\ref{sec:overlap}, both \eqnok{eq:50} and
\eqnok{eq:50a} have strictly feasible points when there is positive
overlap between two ellipsoids. Therefore,  strong duality holds, so
the following optimality conditions relating the solutions of
\eqnok{eq:50} and \eqnok{eq:50a} are satisfied:
\begin{subequations} \label{eq:50kkt}
\begin{align}
\label{eq:50kkt.1}
0 \preceq M_l & \perp C-\sum_{i=1}^{p_l} \zeta_{l,i} A_{l,i}  \succeq 0 \\
\langle A_{l,i},M_l \rangle  & =b_{l,i}, \;\; i=1,2,\dotsc,p_l.
\end{align}
\end{subequations}
By convention, we set $t^*_l(C)=-\infty$ if \eqnok{eq:50} is
infeasible, that is, if there is no overlap between the two ellipses
corresponding to index $l$.  By the nature of the problem, we know
that $t^*_l(C)>0$ if these two ellipses have positive overlap.  It is
easy to see that $t^*_l(C)$ is a concave, extended-valued function of
$C \in \Omega$, and as a consequence that $t^*_l(C)$ is continuous on
the relative interior of its domain.  Further useful facts about
$t^*_l(C)$ are given in Lemma~\ref{lem:subdiff}. These include
Lipschitz continuity in a neighborhood of a point $C$ at which
\eqnok{eq:50a} has a strictly interior point (which, as noted in
Subsection~\ref{sec:overlap}, occurs when the two ellipsoids have
positive overlap), and the fact that any solution $M_l(C)$ of
\eqnok{eq:50} belongs to the Clarke subdifferential of $t^*_l(C)$.

Using the notation of \eqnok{eq:50} and \eqnok{eq:50a} to capture the
min-max-overlap problem \eqnok{prob:non_convex_ell}, we can state this
problem as follows:
\beq \label{eq:51}
\min_{C \in \Omega} \;\; t^*(C) := 
\max_{l=1,2,\dotsc,m} \, t_l^*(C).
\eeq
Here each element in $\{1,2,\dotsc,m\}$ represents the overlap problem
for a given pair of ellipsoids. Note that $t^*(C)=-\infty$ if no pair
of ellipsoids overlaps or touches.

We now define the subproblems to be solved at each iteration of the
algorithm, which depend on just a subset $\cF \subset
\{1,2,\dotsc,m\}$ of the individual overlap problems.  The key
quantity is defined as follows
\beq \label{eq:52} 
t^*_{\cF} (C) := \max_{l \in  \cF} \, t^*_l(C),
\eeq
where $\cF$ is a subset of the {\em strictly overlapping} ellipsoid
pairs, that is,
\[
\cF \subset \{ l=1,2,\dotsc,m \, : \, t^*_l(C) > 0 \}.
\]
(We will be more specific about the definition of $\cF$ later.)
%
% \begin{subequations}
% \label{eq:52}
% \begin{alignat}{2}
% P(\cF,C): \qquad t^*_{\cF}(C) := \min_{t,M_l, l \in \cF} & \, t &&\\
% \mbox{s.t.} \;\; t & \ge \langle C,M_l \rangle, \quad && l \in \cF, \\
% \langle A_{l,i} , M_l \rangle &= b_{l,i}, && l \in \cF, \; i=1,2,\dotsc,p_l, \\
% M_l & \succeq 0, && l \in \cF.
% \end{alignat}
% \end{subequations}
% Note that \eqnok{eq:52} can be solved by solving \eqnok{eq:50} for
% each $l \in \cF$ individually, then simply finding the maximum of
% $t^*_l(C)$ over those $l \in \cF$ for which $t^*_l(C)>-\infty$. Hence,
% the matrices $M_l(C)$ that are optimal in \eqnok{eq:50} also provide
% the optima in \eqnok{eq:52}.
%
% Optimality conditions for the convex problem \eqnok{eq:52} are as
% follows: There exist $\lambda_l$ for $l \in \cF$ and $\mu_{l,i}$ for
% $l \in \cF$ and $i=1,2,\dotsc,p_l$ such that 
% \begin{subequations} 
% \label{eq:54}
% \begin{align}
%  1-\sum_{l \in \cF} \lambda_l &= 0, \\
% 0 \le \lambda_l & \perp t^*_{\cF}(C) - \langle C,M_l \rangle \ge 0, 
% \;\; l \in \cF,\\
% 0 \preceq M_l & \perp \lambda_l C - \sum_{i=1}^{p_l} 
% \mu_{l,i} A_{l,i} \succeq 0,  \;\; l \in \cF, \\
% \langle A_{l,i} , M_l \rangle &= b_{l,i}, \;\; l \in \cF, \; i=1,2,\dotsc,p_l.
% \end{align}
% \end{subequations}
%
In the algorithm, the solutions $M_l(C)$ of \eqnok{eq:50} for $l \in
\cF$ are used to construct a linearized subproblem whose solution is a
step $\DC$ in the parameter $C$, assuming that the current $C$ is
feasible.  The subproblem is as follows:
\begin{subequations}  \label{eq:53}
\begin{align}
L(\cF,C,M_{\cF}(C),\rho): \qquad &r(\cF,C,M_{\cF}(C),\rho) 
:= \min_{r, \DC} \, r \\
\mbox{s.t.} \qquad &r \ge t^*_l(C) + \langle \DC , M_l(C) \rangle, \;\;
l \in \cF, \\
&C + \DC  \in \Omega, \;\;\; \|\DC\| \le \rho.
\end{align}
\end{subequations}
Here $\rho>0$ is a trust-region radius, and $M_{\cF}(C)$ denotes the
set of matrices $\{ M_l(C) \, : \, l \in \cF \}$.  The problem
\eqnok{eq:53} is convex, and its feasible set is bounded, so it has an
optimal value which we denote by $\DC(\rho)$.  Further, the KKT
conditions are satisfied at this point. This claim follows from the
fact that, given the point $\bar{C}$ satisfying \eqnok{eq:Cbar}, and
defining $\DC = \epsilon(\bar{C}-C)$ for some small positive
$\epsilon>0$, we have that
\[
C + \DC = (1-\epsilon)C + \epsilon \bar{C} \in \inT \Omega,
\]
while the trust-region condition is strictly satisfied ($\| \epsilon
(\bar{C}-C) \| < \rho$), and the remaining  constraints in \eqnok{eq:53}
are affine. Hence, the conditions of \cite[Theorem~28.2]{Roc70} are
satisfied, and we can apply \cite[Corollary~28.3.1]{Roc70} to deduce
that there exist
% \sjwcomment{Gah! Do we actually
%   need strong duality? $C$ denotes ellipsoid positions /
%   orientations. The set $\Omega$ could have an empty interior,
%   e.g. when the ellipsoid is actually a sphere. Is it enough for
%   strong duality to say that $\Omega$ has nonempty {\em relative}
%   interior? Maybe all is OK because we can treat \eqnok{eq:53} as a
%   {\em linear} program in the components of $\DC$ with an additional
%   abstract (but closed convex) set membership constraints $C + \DC \in
%   \Omega$ and $\| \DC \| \le \rho$.} 
$\mu_l$, $l \in \cF$ and $\tau \ge 0$
such that
\begin{subequations} \label{eq:55}
\begin{align}
&1-\sum_{l \in \cF} \mu_l = 0, \\
&0 \le \mu_l \perp r(\cF,C,M_{\cF}(C),\rho) - t^*_l(C) - \langle
\DC, M_l(C) \rangle \ge 0, \;\; l \in \cF, \\
&- \sum_{l \in \cF} \mu_l M_l(C) - \tau u \in N_{\Omega}(C+\DC) \;\;
\mbox{for some $u \in \partial \|\DC\|$,} \\
&C + \DC \in \Omega, \\
&0 \le \tau \perp  \rho - \| \DC \| \ge 0.
\end{align}
\end{subequations}
Here $N_{\Omega}(C)$ denotes the normal cone to the closed convex set
$\Omega$ at the point $C$ (see \eqnok{def.normal}) and
$\partial$ denotes a subdifferential. (As noted in
Section~\ref{sec:intro}, since $\|\cdot\|$ is convex and Lipschitz
continuous, the Clarke subdifferential coincides with the
subdifferential from convex analysis.) Note that the set $\{ \tau v \,
: \, \tau \ge 0, \; v \in \partial \|\DC\|\}$ is a closed convex cone
and that it is an outer semicontinuous set-valued function of $\DC$.

It is sometimes convenient to rewrite $L(\cF,C,M_{\cF}(C),\rho)$ by
defining the function 
\begin{equation} \label{eq:defG}
G_{\cF}(\DC;C,M_{\cF}(C)) := \max_{l \in \cF} \, \langle C + \DC , M_l(C) \rangle,
\end{equation}
and writing
\begin{equation} \label{eq:L.alt}
L(\cF,C,M_{\cF}(C),\rho): \;\; \min_{\DC} \, G_{\cF}(\DC;C,M_{\cF}(C)) 
\;\; \mbox{s.t.} \;\; C + \DC  \in \Omega, \;\; \|\DC\| \le \rho.
\end{equation}
Note that $G_{\cF}(\cdot;C,M_{\cF}(C))$ is convex, in fact piecewise
linear.

Next, we define the reference problem $P(\cF)$ that minimizes
$t^*_{\cF}(C)$ defined in \eqnok{eq:52} over $C \in \Omega$:
\beq \label{eq:56}
P(\cF): \qquad t^*_{\cF} := \min_{C \in \Omega} \,  t^*_{\cF}(C) =
\min_{C \in \Omega} \, \max_{l \in \cF} \, t^*_l(C).
% \;\; \makebox{s.t. $C \in \Omega$}.
\eeq
Nonsmooth analysis provides the following result that characterizes
solutions of \eqnok{eq:56}.
% Here and subsequently, we denote by $\partial$ the Clarke
% subdifferential of a function.
\begin{lemma} \label{lem:pfstat} Suppose that for a given
set $\cF \subset \{1,2,\dotsc,m\}$, $\bar{C}$ is a local
  solution of \eqnok{eq:56} at which 
  \eqnok{eq:50a} has a strictly interior point, for all $l \in
  \cF$. Define $\bar{\cF}(\bar{C})$ to be the set of indices achieving
  the maximum in \eqnok{eq:56}, that is, $\bar{\cF}(\bar{C}) = \{ l
  \in \cF \, : \, t^*_l(\bar{C}) = t^*_{\cF} \}$. Then there exist
  $\bar{M}_l \in \partial t^*_l(\bar{C})$ and scalars $\mu_l$, for
  all $l \in \bar{\cF}(\bar{C})$, such that
\begin{equation} \label{eq:56opt}
-\sum_{l \in \bar{\cF}(\bar{C})} \mu_l \bar{M}_l \in N_{\Omega}(\bar{C}), 
\quad
\sum_{l \in \bar{\cF}(\bar{C})} \mu_l = 1,
\quad
\mu_l \ge 0, \; l \in  \bar{\cF}(\bar{C}), 
\quad
\bar{C} \in \Omega.
\end{equation}
\end{lemma}
\begin{proof}
  We appeal to results about the Clarke subdifferential applied to
  max-functions and sums of functions.  First, note that the strict
  interiority assumption means that we can apply
  Lemma~\ref{lem:subdiff} (iv) to deduce that each $t^*_l$ is
  Lipschitz continuous in a neighborhood of $\bar{C}$.  Hence,
  applying \cite[Theorem~6.1.5]{BorL00}, we have that
\begin{equation} \label{eq:dTf}
  \partial \, t^*_{\cF}(C)  \subset
  \mbox{conv} \{ \partial t^*_l(\bar{C}) \, : \, l \in \bar{\cF} (\bar{C}) \},
\end{equation}
where $\mbox{conv}(\cdot)$ denotes the convex hull.  The Corollary in
\cite[p.~52]{Cla83} can be used to show that when $\bar{C}$ is a
solution of \eqnok{eq:56}, we have
\[
0 \in \partial \, t^*_{\cF}(\bar{C}) + N_{\Omega}(\bar{C}).
\]
The result follows by combining this expression with \eqnok{eq:dTf}.
\end{proof}

By introducing multipliers for the indices $l \in \cF \setminus
\bar{\cF} (\bar{C})$, we can restate the conditions \eqnok{eq:56opt}
as follows:
\begin{subequations}
\label{eq:56opta}
\begin{align} 
&0 \le \mu_l \perp  t^*_{\cF} - t^*_l(\bar{C}) \ge 0, 
\qquad \mbox{for all $l \in \cF$,} \\
&\sum_{l \in \cF} \mu_l =1, \\
&-\sum_{l \in \cF} \mu_l \bar{M}_l \in  N_{\Omega}(\bar{C}), \\
&\bar{C} \in \Omega.
\end{align}
\end{subequations}
We say that a point $\bar{C}$ at which these conditions are satisfied
is {\em Clarke-stationary} for $P(\cF)$ defined in \eqnok{eq:56}.

%
% By reintroducing the variables $M_l$ from \eqnok{eq:50}, we can state
% $P(\cF)$ equivalently as follows:
% \begin{subequations}  \label{eq:57}
% \begin{alignat}{2}
% P(\cF): \qquad t^*_{\cF} := \min_{C, M_l} \, & t && \\
% \mbox{s.t.} \;\; t & \ge \langle C,M_l \rangle, \;\; && l \in \cF, \\
% \langle A_{l,i},M_l \rangle  &= b_{l,i}, \;\; && l \in \cF, \; i=1,2,\dotsc,p_l, \\
% M_l & \succeq 0, \;\; &&l \in \cF, \\
% C & \in \Omega. &&
% \end{alignat}
% \end{subequations}
% Stationarity conditions for this problem (which is not convex in
% general because of the presence of the bilinear term $\langle C,M_l
% \rangle$) are as follows: There exist multipliers $\mu_l$, $l \in
% \cF$ and $\mu_{l,i}$, $l \in \cF$ and $i=1,2,\dotsc,p_l$, such that
% \begin{subequations} \label{eq:58}
% \begin{align}
% 1-\sum_{l \in \cF} \mu_l &=0, \\
% 0 \le \mu_l & \perp t^*_{\cF} - \langle C,M_l \rangle \ge 0, \;\;
% l \in \cF, \\
% 0 \preceq M_l & \perp  \mu_l C - \sum_{i=1}^{p_l} 
% \mu_{l,i} A_{l,i}  \succeq 0, \;\; l \in \cF, \\
% - \sum_{l \in \cF} \mu_l M_l & \in N_{\Omega}(C), \\
% \langle A_{l,i},M_l \rangle  &= b_{l,i}, \;\; l \in \cF, \; i=1,2,\dotsc,p_l, \\
% C & \in \Omega.
% \end{align}
% \end{subequations}

For purposes of our main technical lemma, we define the ``predicted
decrease'' from subproblem $L(\cF,C,M_{\cF}(C),\rho)$ as follows:
\beq \label{eq:Lam}
\Lambda(\cF,C,M_{\cF}(C),\rho) := t^*_{\cF}(C) - r(\cF,C,M_{\cF}(C),\rho).
\eeq
Note that since $\DC=0$ is feasible for 
\eqnok{eq:53}, we have $\Lambda(\cF,C,M_{\cF}(C),\rho) \ge 0$.

\begin{lemma} \label{lem:tech}
Let $\cF \subset \{1,2,\dotsc,m\}$ be given.
\begin{itemize}
\item[(i)] Suppose that $\bar{C}$ is such that  \eqnok{eq:50a}
  has a strictly feasible point for all $l \in \cF$.
  If $\Lambda(\cF,\bar{C},M_{\cF}(\bar{C}),\rho)=0$ for some $\rho>0$
  and some set of solutions $M_l(\bar{C})$ to \eqnok{eq:50} for $l \in
  \cF$, then $\bar{C}$ is Clarke-stationary for $P(\cF)$.
\item[(ii)] $\Lambda(\cF,C,M_{\cF}(C),\rho)$ is an increasing function
  of $\rho>0$.
\item[(iii)] $\Lambda(\cF,C,M_{\cF}(C),\rho) / \rho$ is a decreasing
  function of $\rho>0$.
\item[(iv)] For all $C$ in some neighorhood of $\bar{C}$ defined in
  (i), we have that $t^*_{\cF}(C + \DC(\rho)) \le
  r(\cF,C,M_{\cF}(C),\rho)$ for any $\rho>0$.
\end{itemize}

\end{lemma}

{\em Proof}.
  \noindent (i) If $r(\cF,\bar{C},M_{\cF}(\bar{C}),\rho)=t^*_{\cF}$
  for some $\rho>0$, then $\DC=0$ achieves the minimum in
  \eqnok{eq:53}, for $C=\bar{C}$. Hence, there exist $\mu_l$, $l \in
  \cF$ such that the optimality conditions \eqnok{eq:55} are satisfied
  with $\DC=0$ and $\tau=0$.  Thus, conditions \eqnok{eq:56opta} are
  satisfied with $\bar{M}_l=\bar{M}_l(\bar{C})$ and the same values of
  $\mu_l$, $l \in \cF$.

\medskip\noindent (ii) Trivial, as the size of the feasible region for
$L(\cF,C,M_{\cF}(C),\rho)$ increases with $\rho$.

\medskip\noindent (iii) We need to show that for $\rho_1$ and $\rho_2$
with $0 < \rho_1 < \rho_2$, we have
\[
\frac{\Lambda(\cF,C,M_{\cF}(C),\rho_1)}{\rho_1} \ge 
\frac{\Lambda(\cF,C,M_{\cF}(C),\rho_2)}{\rho_2}.
\]
Using the formulation \eqnok{eq:L.alt} of $L(\cF,C,M_{\cF}(C),\rho)$,
and in particular the convex function $G_{\cF}(\cdot;C,M_{\cF}(C))$
defined in \eqnok{eq:defG}, we have that
\[
G_{\cF}(0;C,M_{\cF}(C)) = \max_{l \in \cF} \, \langle C , M_l(C) \rangle = t^*_{\cF}(C).
\]
Given the solution $\DC(\rho_2)$ of $L(\cF,C,M_{\cF}(C),\rho_2)$, note
that $\frac{\rho_1}{\rho_2} \DC(\rho_2)$ is feasible for
$L(\cF,C,M_{\cF}(C),\rho_1)$. Since $\DC(\rho_1)$ is optimal for
$L(\cF,C,M_{\cF}(C),\rho_1)$, we have
\begin{align*}
G_{\cF}(\DC(\rho_1);C,M_{\cF}(C)) 
& \le G_{\cF}\!\left( \frac{\rho_1}{\rho_2} \DC(\rho_2); C,M_{\cF}(C) \right) \\
& \le \left( 1-\frac{\rho_1}{\rho_2} \right)\!G_{\cF} (0;C,M_{\cF}(C)) + 
\frac{\rho_1}{\rho_2} G_{\cF}(\DC(\rho_2);C,M_{\cF}(C)).
\end{align*}
The result follows by rearrangement of this expression, since 
\begin{align*}
\Lambda(\cF,C,M_{\cF}(C),\rho_1) &= G_{\cF}(0;C,M_{\cF}(C)) 
- G_{\cF}(\DC(\rho_1);C,M_{\cF}(C)), \\
\Lambda(\cF,C,M_{\cF}(C),\rho_2) &= G_{\cF}(0;C,M_{\cF}(C)) - 
G_{\cF}(\DC(\rho_2);C,M_{\cF}(C)).
\end{align*}

% \medskip\noindent (iv) We have
% \begin{align*}
% t^*_l(C +\DC) &= \min_{M_l \succeq 0, \, \langle A_{l,i},M_l \rangle = b_{l,i}, \, i=1,2,\dotsc,p_l} \, \langle C + \DC, M_l \rangle \\
%  & \le \langle C + \DC, M_l(C) \rangle \\
% &= t^*_l(C) + \langle \DC, M_l(C) \rangle.
% \end{align*}

\medskip\noindent (iv) The result follows immediately from
Lemma~\ref{lem:subdiff} (iv), when we use the definition \eqnok{eq:52} and
the fact that
\[
r(\cF,C,M_{\cF}(C),\rho) = \max_{l \in \cF} \, t^*_l(C) + \langle \DC(\rho), M_l(C) \rangle. \qquad\endproof
\]

We show now that all accumulation points of a sequence $\{ C_k \}$ for
which 
\[
\Lambda(\cF,C_k,M_{\cF}(C_k),1) \to 0
\]
are Clarke-stationary for $P(\cF)$.
\begin{theorem} \label{th:stat} Suppose that for a given set $\cF
  \subset \{1,2,\dotsc,m\}$, $\{ C_k \}$ is a sequence of matrices in
  $\Omega$ converging to a limit $\bar{C}$ such that \eqnok{eq:50a}
  has a strictly feasible point for each $l \in \cF$. Suppose further
  that $\Lambda (\cF,C_k,M_{\cF}(C_k),1) \to 0$. Then $\bar{C}$ is
  Clarke-stationary for $P(\cF)$.
\end{theorem}
\begin{proof}
  We invoke Theorem~\ref{th:limitsol} to deduce that for all $l \in
  \cF$, the solution sets of $P(l,C_k)$ in \eqnok{eq:50} are uniformly
  bounded for all $k$ sufficiently large. Hence, we can identify
  subsequences of $\{ M_l(C_k) \}$ for each $l \in \cF$ that approach
  some limit $\bar{M}_l$, where by Theorem~\ref{th:limitsol} (ii),
  $\bar{M}_l$ is a solution of $P(l,\bar{C})$ for each $l \in \cF$. We
  can thus write $M_l(\bar{C}) = \bar{M}_l$ for each $l \in \cF$. By
  defining $M_{\cF}(C_k)$ and $M_{\cF}(\bar{C})$ in obvious ways, and
  taking a subsequence, we have that $M_{\cF}(C_k) \to
  M_{\cF}(\bar{C})$.

  We show next, by contradiction, that
  $\Lambda(\cF,\bar{C},M_{\cF}(\bar{C}),1)=0$.  If this claim is not
  true, there exists $\DC'$ such that
\[
\| \DC' \| \le 1, \quad
\bar{C}+\DC' \in \Omega, \quad
G_{\cF}(\DC';\bar{C},M_{\cF}(\bar{C})) \le t^*_{\cF}(\bar{C}) - \epsilon,
\]
for some $\epsilon>0$. Defining
\[
\DC_k' := \bar{C} - C_k + \DC',
\]
we have from $C_k \to \bar{C}$, $\| \DC_k \| \le 1$, and $C_k + \DC_k'
= \bar{C} + \DC' \in \Omega$ that $\DC_k'$ is feasible for
$L(\cF,C_k,M_{\cF}(C_k),2)$. Hence, invoking Lemma~\ref{lem:tech}
(iii), we have
\begin{align*}
\lim_k \, G_{\cF}(\DC_k';C_k,M_{\cF}(C_k)) & \ge
\lim_k \, t^*_{\cF}(C_k) - \Lambda(\cF,C_k,M_{\cF}(C_k),2) \\
& \ge 
\lim_k t^*_{\cF}(C_k) - 2 \Lambda(\cF,C_k,M_{\cF}(C_k),1)  \\
& = t^*_{\cF}(\bar{C}).
\end{align*}
The final limit above follows from the definition of $t^*_{\cF}$,
Lemma~\ref{lem:subdiff} (iv), and the assumption that
$\Lambda(\cF,C_k,M_{\cF}(C_k),1) \to 0$.  On the other hand, we have
from $C_k + \DC_k' = \bar{C} + \DC'$, the definition of $G_{\cF}$
\eqnok{eq:defG}, and the limit $M_{\cF}(C_k) \to M_{\cF}(\bar{C})$
that
\begin{align*}
\lim_k \, G_{\cF}(\DC_k';C_k,M_{\cF}(C_k)) 
& = \lim_k \, \max_{l \in \cF} \, \langle C_k+\DC_k',M_l(C_k) \rangle \\
&= \lim_k \, \max_{l \in \cF} \, \langle \bar{C}+\DC',M_l(C_k) \rangle \\
&=  \max_{l \in \cF} \, \langle \bar{C}+\DC',M_l(\bar{C}) \rangle \\
& \le
t^*_{\cF}(\bar{C})-\epsilon,
\end{align*}
where $\epsilon>0$ is defined above.  This yields the contradiction,
so we conclude that $\Lambda(\cF,\bar{C},M_{\cF}(\bar{C},1)=0$, as
claimed. Clarke stationarity of $\bar{C}$ for $P(\cF)$ now follows
from Lemma~\ref{lem:tech} (i).
\end{proof}

\subsection{Algorithm} \label{sec:alg}

We now define the algorithm for solving the problem $P$ defined by
\eqnok{eq:51}.
%  \beq \label{eq:P} P: \qquad \min_C \,
%\max_{l=1,2,\dotsc,m} \, t^*_l(C) \;\; \makebox{s.t. $C \in \Omega$.}
%\eeq
%
%We use the notation of the previous subsection, for simplicity.  
Note that in this general setting, $t^*_l(C)$ defined by
\eqnok{eq:50} is continuous on the set
\[
\Psi_l := \{ C \, : \, t^*_l(C)>-\infty \},
\]
which is closed and convex.
We make additional assumptions about the nature of the solutions
to the parametrized primal-dual pair \eqnok{eq:50}, \eqnok{eq:50a},
that do not hold in general, but which are satisfied for the
application we consider here.

\begin{assumption} \label{ass:sdpo} 
\begin{itemize}
\item[(i)] $t^*_l(C) > -\infty \;\; \Rightarrow \;\; t^*_l(C) \ge 0$.
\item[(ii)] If $t^*_l(C)>0$, then the dual \eqnok{eq:50a} has a strict
  feasible point.
\end{itemize}
\end{assumption}

It is an immediate consequence of Assumption~\ref{ass:sdpo} and
Lemma~\ref{lem:subdiff} that all points $C$ on the boundary of
$\Psi_l$ have $t^*_l(C)=0$.  We also have the following uniform
continuity result.
\begin{lemma} \label{lem:tlc} Suppose that $t^*_l(C)$ is defined by
  \eqnok{eq:50}, that Assumption~\ref{ass:sdpo} holds, and that
$\Omega$ has the form \eqnok{eq:COm}. Let
  $\bar{t}>0$ be given. Then for any $\epsilon>0$, there is $\delta>0$
  such that for all $\bar{C}\in \Omega$, all $C \in \Omega$ with $\|
  C-\bar{C} \| \le \delta$, and all $l=1,2,\dotsc,m$, the following
  conditions hold:
\begin{itemize}
\item[(i)] If $t^*_l(\bar{C}) \ge \bar{t}$, then $t^*_l(C) \ge
t^*_l(\bar{C}) - \epsilon$;
\item[(ii)] $t^*_l(C) \le \max(0,t^*_l(\bar{C})) + \epsilon$.
\end{itemize}
\end{lemma}
\begin{proof}
  Note first that since $\Omega$ is compact, the set 
  $\Psi_l(\bar{t}) := \{ C \in \Omega \mid t^*_l(C) \ge \bar{t} \}$ is
  also compact, for any $l\in \{1,2,\dotsc,m\}$ and any $\bar{t} > 0$.
Since $t^*_l(\cdot)$ is continuous at every point of this set, under
the stated assumptions, it is uniformly continuous on this set. Thus
for any $\epsilon>0$, there is a value $\delta = \delta_l(\epsilon)>0$
such that (i) holds. Thus it is sufficient for (i) to define $\delta$
to be $\min_{l=1,2,\dotsc,m} \delta_l(\epsilon)$. 

For (ii), we suppose for contradiction that for some $\epsilon>0$,
there is no $\delta>0$ with the property claimed. Thus, for any
sequence $\{ \delta_r \}$ with $\delta_r \downarrow 0$, we can find
$\bar{C}_r \in \Omega$, $C_r \in \Omega$ with $\| C_r - \bar{C}_r \|
\le \delta_r$, and $l \in \{1,2,\dotsc,m\}$ such that
\begin{equation} \label{eq:tlc.1}
t^*_l(C_r) > \max(0, t^*_l(\bar{C}_r)) + \epsilon.
\end{equation}
By taking a subsequence if necessary, we can assume that this
inequality holds for some fixed $l \in \{1,2,\dotsc,m\}$. Since all
$C_r$ belong to the compact set $\Omega \cap \{ C \mid t^*_l(C) \ge
\epsilon \}$, we can assume (by taking another subsequence if
necessary) that $C_r \to \hat{C}$, for some $\hat{C}$ with
$t^*_l(\hat{C}) \ge \epsilon$. It follows that $\bar{C}_r \to \hat{C}$
also, so using continuity of $t^*_l$ and taking limits in both sides
of \eqnok{eq:tlc.1}, we obtain
\[
t^*_l(\hat{C}) = \lim_{r \to \infty} \, t^*_l(C_r) \ge \lim_{r \to \infty}
\, \max(0, t^*_l(\bar{C}_r)) + \epsilon  = \max(0,t^*_l(\hat{C})) + \epsilon
\ge t^*_l(\hat{C}) + \epsilon,
\]
a contradiction.
\end{proof}

A key issue in implementing the algorithm is to decide which subset
$\cF$ of the overlapping pairs to use in calculating the step $\DC$ in
\eqnok{eq:53}. Clearly, $\cF$ should include the indices $l$ for which
the overlaps between the corresponding ellipsoid pairs are at or near
the maximum. It could also include other indices with positive (but
smaller) overlap. Clearly, it cannot contain any non-overlapping
ellipsoids, as the problem $P(l,C)$ \eqnok{eq:50} has no solution in
this case, so $M_l(C)$ is not defined. We settle on the following
requirement, which depends on parameters $\eta_1, \eta_2 \in (0,1)$
with $0 < \eta_1 < \eta_2 <1$: Given $C_k$ for which $t^*(C_k)>0$ (see
definition \eqnok{eq:51}), we choose $\cF_k$ to satisfy:
\beq \label{eq:Fk}
\{ l \, : \, t^*_l(C_k)  \ge \eta_2 t^*(C_k) \} \subset \cF_k \subset
\{ l \, : \, t^*_l(C_k) \ge \eta_1 t^*(C_k) \}.
\eeq

Algorithm~\ref{alg:xxx} describes our method. It follows a standard
trust-region framework, though its analysis is a little
non-standard. At each iteration, we calculate a candidate step $\DC_k$
by solving the linearized subproblem \eqnok{eq:53} with trust-region
radius $\rho_k$, and calculate the predicted reduction
$\Lambda(\cF_k,C_k, M_{\cF_k}(C_k), \rho_k)$ \eqnok{eq:Lam} expected
from this step. If the actual objective achieves at least a fraction
$c_1$ of this decrease (for $c_1 \in (0,1)$), we accept the step. If
in fact the improvement is at least a larger fraction $c_2$ of the
expected decrease, we may increase the trust-region radius for the
next iteration.  Otherwise, we do not take the step, but rather shrink
the trust-region radius and proceed to the next iteration.

\begin{algorithm} \label{alg:xxx}
\caption{Packing Ellipsoids by Minimizing Overlap}
\begin{algorithmic}
  \STATE Given $\Omega \subset S \R^{n \times n}$ compact; $\eta \in
  (0,1)$; $c_1$ and $c_2$ with $0 < c_1 < c_2 < 1$; $\phi_1$ and
  $\phi_2$ with $0 < \phi_1 < 1 < \phi_2$; and
  $\rhomax >0$;
\STATE Choose $C_0 \in \Omega$, $\rho_0 \in (0,\rhomax]$;
\FOR{$k=0,1,2,\dotsc$}
\STATE Define $\cF_k$  as in \eqnok{eq:Fk}; 
\STATE Solve $L(\cF_k,C_k,M_{\cF_k}(C_k),\rho_k)$ \eqnok{eq:53} to obtain $\DC_k$; 
\STATE Compute predicted decrease  $\Lambda(\cF_k,C_k,M_{\cF_k}(C_k),\rho_k)$ from
\eqnok{eq:Lam};
\IF{$\DC_k=0$ or $t^*(C_k+\DC_k)=0$}
\STATE {\bf stop};
\ENDIF
\IF{$t^*(C_k+\DC_k) \le t^*(C_k) - c_1 \Lambda(\cF_k,C_k,M_{\cF_k}(C_k),\rho_k)$}
\STATE $ C_{k+1} \leftarrow C_k + \DC_k$; \\
\IF{$t^*(C_k+\DC_k) \le t^*(C_k) - c_2 \Lambda(\cF_k,C_k,M_{\cF_k}(C_k),\rho_k)$}
\STATE $\rho_{k+1}  \leftarrow \min( \phi_2 \rho_k, \rhomax)$; 
\ENDIF
\ELSE
\STATE $C_{k+1} \leftarrow C_k$;
\STATE $\rho_{k+1} \leftarrow \phi_1 \rho_k$; 
\ENDIF
\ENDFOR
\end{algorithmic}
\end{algorithm}

We now show that when the values $t^*(C_k)$ are bounded away from
zero, there is a positive threshold such that any step $\DC_k$ with
norm smaller than this threshold will be accepted.

\begin{lemma} \label{lem:rhobar} Suppose that
  Assumption~\ref{ass:sdpo} holds and let $\bar{t}>0$ be given. Let
  $C_k$ be any iterate with $t^*(C_k) \ge \bar{t}$ such that $C_k$ is
  not Clarke-stationary for the problem $P$ defined in \eqnok{eq:51},
  and suppose that $\cF_k$ satisfies \eqnok{eq:Fk}. Then there exists
  a threshold value $\bar{\rho}_{\bar{t}}>0$ (independent of $C_k$)
  such that
  \begin{equation} \label{eq:trstep} 
t^*(C_k+\DC(\rho)) \le t^*(C_k) -
    \Lambda(\cF_k,C_k,M_{\cF_k}(C_k),\rho) < t^*(C_k)
\end{equation}
whenever $\|\DC(\rho)\|$ is a solution of $L(\cF_k,C_k,M_{\cF_k}(C_k),\rho)$ with $\rho
\in (0,\bar{\rho}_{\bar{t}}]$.
\end{lemma}
\begin{proof}
  Note first that if $C_k$ were Clarke-stationary for $P(\cF_k)$,
  given that $\cF_k$ contains all the indices $l$ for which
  $t^*_l(C_k)$ attains the maximum $t^*(C_k)$, we would have that
  $C_k$ is also Clarke-stationary for $P$, which we have assumed is
  not the case.  From Assumption~\ref{ass:sdpo} and Lemma~\ref{lem:tech} (i) we have therefore that
  $\Lambda(\cF_k,C_k,M_{\cF_k}(C_k),\rho)>0$ for all $\rho>0$ and all
  solutions $M_l(C_k)$ to \eqnok{eq:50} with $C=C_k$ and $l \in
  \cF_k$.

  Now define $\epsilon = \bar{t}(1-\eta_2)/2$, and let
  $\bar{\rho}_{\bar{t}}$ be the corresponding (positive) value of
  $\delta$ from Lemma~\ref{lem:tlc}. Consider any $\DC$ such that $\|
  \DC\| \le \bar{\rho}_{\bar{t}}$. For indices $l$ such that
  $t^*_l(C_k) = t^*(C_k) \ge \bar{t}$, we have from
  Lemma~\ref{lem:tlc} (i) that
\[
t^*_l(C_k+\DC) \ge t^*(C_k) - \bar{t}(1-\eta_2)/2 \ge t^*(C_k)
(1+\eta_2)/2.
\]
For indices $l \notin \cF_k$, we have $t^*_l(C_k) \le \eta_2
t^*(C_k)$, and so from Lemma~\ref{lem:tlc} it follows
that
\[
t^*_l(C_k+\DC) \le \max(0, t^*_l(C_k)) + \bar{t}(1-\eta_2)/2 \le
\eta_2 t^*(C_k) + \bar{t}(1-\eta_2)/2 \le t^*(C_k)
(1+\eta_2)/2.
\]
Hence, for $\| \DC \| \le \bar{\rho}_{\bar{t}}$, the index $l$ for
which $t^*(C_k + \DC) = t^*_l(C_k+\DC)$ comes from the set $\cF_k$,
that is,
\[
t^*_{\cF_k} (C_k+\DC) = t^*(C_k+\DC).
\]
So choosing $\rho \in (0,\bar{\rho}_{\bar{t}}]$ and setting
$\DC(\rho)$ to be the solution of $L(\cF_k,C_k,M_{\cF_k}(C_k),\rho)$, we have
 from Lemma~\ref{lem:tech} (iv)  that 
\begin{align*}
t^*(C_k+\DC(\rho)) & = t^*_{\cF_k} (C_k+\DC(\rho))  \\
& \le t^*_{\cF_k}(C_k) - 
\Lambda (\cF_k,C_k,M_{\cF_k}(C_k),\rho) \\
&  < t^*_{\cF_k}(C_k)  \\
& = t^*(C_k),
\end{align*}
as claimed.
\end{proof}

The inequality \eqnok{eq:trstep} satisfies the step acceptance
conditions in Algorithm~\ref{alg:xxx}, since $0<c_1<c_2<1$. It follows
immediately that for any $C_k$ with $t^*(C_k)>0$, the algorithm cannot
``get stuck'' by performing infinitely many unsuccessful iterations
--- eventually it will decrease $\rho$ to the point where the step
acceptance condition holds.

We now prove the main convergence result.
\begin{theorem} \label{th:conv} Suppose that Assumption~\ref{ass:sdpo}
  holds. Then either Algorithm~\ref{alg:xxx} terminates finitely, or
  else it generates an infinite sequence of iterates $\{ C_k \}$ for
  which accumulation points exist, and all accumulation points are
  Clarke-stationary points of $P$. When $t^*(C_k) \downarrow 0$, all
  accumulation points are in fact global solutions of $P$.
\end{theorem}
\begin{proof}
  The finite termination cases are obvious, so we focus on the case of
  an infinite sequence $\{ C_k \}$. Since all iterates are confined to
  the compact set $\Omega$, accumulation points of sequence $\{ C_k\}$
  exist. Note that the sequence of function values $\{ t^*(C_k)\}$ is
  decreasing. The final statement of the theorem is self-evident, as
  this case indicates convergence to points at which there are no
  overlaps between ellipsoids. Hence, we focus on the case
  in which there exists $\bar{t}>0$ such that $t^*(C_k) \ge \bar{t}$
  for all $k$.  From Lemma~\ref{lem:rhobar}, we see that at each
  iteration $k$, the trust-region radius $\rho_k$ will generate a
  successful step $\DC_k$ whenever it falls below
  $\bar{\rho}_{\bar{t}}$. Hence, since the algorithm decreases $\rho$
  by a factor of $\phi_1$ after each unsuccessful step, we have that
\beq \label{eq:phi1rho}
  \rho_k \ge \min (\rho_0, \phi_1 \bar{\rho}_{\bar{t}}), \quad
\mbox{for all $k$.}
\eeq

  In considering accumulation points of the sequence $\{ C_k \}$ we
  can remove all repeated entries from the sequence.  These repeats
  arise from unsuccessful steps (for which the acceptance condition
  was not satisfied), and the accumulation points of the sequence are
  the same whether the repeated entries are present or not. Note that
  there must be infinitely many successful steps since, as we note in
  the comment after Lemma~\ref{lem:rhobar}, the algorithm must
  eventually move away from any non-stationary point $C_k$ with
  $t^*(C_k)>0$. We denote the subsequence of successful iterates by
  $\cS$.

  At a successful iteration $k \in \cS$, we have
\begin{align*}
t^*(C_{k+1})  &= t^*(C_k + \DC_k) \\
& \le t^*(C_k) - c_1 \Lambda (\cF_k, C_k, M_{\cF_k}(C_k), \rho_k) \\
& \le t^*(C_k) - c_1 \min(\rho_k,1) \Lambda(\cF_k,C_k, M_{\cF_k}(C_k), 1) \\
& \le t^*(C_k) - c_1 \min(\rho_0,\phi_1 \bar{\rho}_{\bar{t}},1) \Lambda(\cF_k,C_k, M_{\cF_k}(C_k), 1),
\end{align*}
where we used Lemma~\ref{lem:tech} (ii) and (iii) to derive the second
inequality, and the last
inequality comes from \eqnok{eq:phi1rho}.
Since the sequence $\{ t^*(C_k) \}$ is decreasing and
bounded below (by $\bar{t}$), we have $0 < t^*(C_k)-t^*(C_{k+1})
\downarrow 0$, so by rearranging  and using the
fact that $\min(\rho_0,\phi_1 \bar{\rho}_{\bar{t}},1)>0$, we have
\begin{equation} \label{eq:limlam}
\lim_{k \in \cS} \, \Lambda(\cF_k,C_k, M_{\cF_k}(C_k), 1) =0.
\end{equation}

Now suppose that $\bar{C} \in \Omega$ is an accumulation point of the
full sequence $\{ C_k \}$. As noted above, it must also be an
accumulation point of the ``successful iterate'' sequence $\{C_k \}_{k
  \in \cS}$. So by taking a further subsequence $\cS' \subset \cS$, we
have $\lim_{k \in \cS'} \, C_k = \bar{C}$. Since there is only a finite
number of possibilities for the set $\cF_k$, we can take another
subsequence $\cS'' \subset \cS'$ such that, in addition, $\cF_k \equiv
\cF$ for all $k \in \cS''$. By the definition \eqnok{eq:Fk}, we have
that $t^*_l(C_k) \ge \eta_1 t^*(C_k) \ge \eta_1 \bar{t}$ for all $l \in
\cF_k$ and all $k \in \cS''$. Thus, using continuity of $t^*_l$, we have that 
\[
t^*_l(\bar{C}) \ge \eta_1 \bar{t}>0, \qquad \mbox{for all $l \in \cF$,}
\]
implying that $P(l,\bar{C})$ has a strictly feasible point for
all $l \in \cF$. Clarke stationarity of $\bar{C}$ now follows from
\eqnok{eq:limlam}, using the fact that $\cS'' \subset \cS$ and
applying Theorem~\ref{th:stat}.
\end{proof}

\section{Application: Chromosomal Arrangement in Human Cell
  Nuclei}\label{sec:biology}

We return to the application introduced in Section~\ref{sec:intro},
that is, finding arrangements of chromosome territories in a cell
nucleus on the basis of simple geometric principles (namely, low
overlap and discouragement of proximity for homologous pairs) and
seeing how closely the resulting arrangements match the experimental
observations that have been made to date.

% Our goal is to develop a model for the chromosome arrangement in cell
% nuclei. For this purpose, we apply the bilevel optimization procedure
% to find ellipsoidal packings with minimal overlap. We analyze what
% characteristics of the chromosome arrangements can be explained purely
% by the geometric constraints.

During most of the cell cycle the chromosomes of higher eukaryotes are
organized into distinct compartments known as chromosome territories
(CTs). These domains have a roughly ellipsoidal shape and can overlap
each other. This overlap is believed to have an important biological
purpose, since it allows for the interaction and co-regulation of
different genes. Additionally, the CTs tend to exploit the space
available inside the cell nucleus, to allow for internal DNA-free
channels, the \emph{interchromatin compartments}. These compartments
allow CTs deep inside the cell nucleus to be accessible for regulatory
factors.

As noted earlier, the arrangement of CTs is known to be
non-random. Arrangements are known to be broadly conserved during
evolution and are similar among cell types with similar developmental
pathways. CT arrangements can also change during processes such as cancer
development or cell differentiation. See \cite{Cremer} for more
details and an overview about what is known about CT arrangements.

There is strong evidence that chromosomes have a preferred radial
position inside the nucleus. These preferences appear to be
different for nuclei of different shapes, spherical and
ellipsoidal. In ellipsoidal nuclei, CT size seems to drive the radial preferences, with the smaller CTs tending to lie nearer to the center. In spherical nuclei the situation is less clear, with the more gene-dense chromosomes seeming to lie nearer to the center of the nucleus. 

There is also evidence for neighbor preferences, which may play an
important role in causing co-regulated genes in different chromosomes
to be closer together. In particular, it has been observed recently
that CTs tend to favor neighborhoods of heterologous chromosomes. This results in the
two chromosomes in a homologous pair tending to be well
separated. (In human cells there are 22 homologous pairs, each consisting
of one chromosome from the mother and a similar one from the
father.)

We model a CT arrangement as a packing of overlapping ellipsoids of
various sizes inside an ellipsoidal container, which represents the
cell nucleus. Minimizing maximum overlap mimics the fact that the CTs
exploit the space available in the nucleus, to allow for the presence
of contiguous DNA-free interchromatin channels. These
channels extend from the nuclear pores into the interior of the
nucleus, making even the deepest CTs accessible from outside and
connecting most chromosome territories.

In this section we analyze whether purely geometric considerations,
enforcing the simple principles of minimal overlap and
well-separatedness of homologous pairs, can explain the observed
arrangements of CTs in cell nuclei of different sizes and shapes.

\subsection{Human Cell Nucleus}

The human cell nucleus has a volume of between $500$ $\mu m^3$ and
$1600$ $\mu m^3$, depending on the cell size and stage of
differentiation.  The shape also differs according to cell type. Human
fibroblasts, for example, have flat ellipsoidal nuclei, whereas
lymphocytes have spherical cell nuclei. In this study we analyze three
different nucleus sizes: small ($500$ $\mu m^3$), medium ($1000$ $\mu
m^3$) and large ($1600$ $\mu m^3$). For all three sizes we consider
two shapes: spherical nuclei and flat ellipsoidal nuclei (with axis
lengths in the ratio 1:2:4).

We estimate the volume of each CT to be proportional to the number of
base pairs in the chromosome, with the constant of proportionality
determined by the average chromatin packing density.  The number of base
pairs for human cells ranges from $47$ Mbp (chromosome 21) to $247$
Mbp (chromosome 1), while the human chromatin packing density in living
cells has been estimated to be $0.15$ $\mu m^3$/Mbp
(\cite{Mueller_chr_packing}). By multiplying the total number of base
pairs by this average density, we arrive at a total volume of about
$461$ $\mu m^3$ over all CTs. The individual volumes for each
chromosome territory are given in
Table~\ref{table_chr_volume}.

\begin{table}[!h]
 \centering
\caption{Volume of each chromosome territory based on chromatin packing density of $0.15$ $\mu m^3$/Mbp.}
\label{table_chr_volume}
\begin{tabular}{l | c c c c c c c c}
CT & 1 & 2 & 3 & 4 & 5 & 6 & 7 & 8 \\ \hline
volume & 37.05 & 36.45 & 29.85 & 28.65 & 27.15 & 25.65 & 23.85 & 21.90 \end{tabular}

\begin{tabular}{l | c c c c c c c c}
CT & 9 & 10 & 11 & 12 & 13 & 14 & 15 & 16 \\ \hline
volume &  21.00 & 20.25 & 20.10 & 19.80 & 17.10 & 15.90 & 15.00 & 13.35
\end{tabular}

\begin{tabular}{l | c c c c c c c c}
CT & 17 & 18 & 19 & 20 & 21 & 22 & X & Y\\ \hline
volume & 11.85 & 11.40 & 9.45 & 9.30 & 7.05 & 7.50 & 23.25 & 8.70
\end{tabular}\;\,\quad\quad\quad
\end{table}

\subsection{Implementation} \label{sec:implementation}

The algorithm was implemented in Matlab and CVX~\cite{GraB12}. Both,
the pairwise-overlap problems and the master problem at each iterate
of the algorithm were formulated in CVX. The algorithm is terminated
when one of the following conditions holds.
\begin{itemize}
\item[(i)] The ratio of predicted decrease to trust-region radius
  falls below a specified tolerance. Using the notation of
  Algorithm~\ref{alg:xxx}, we state this condition as
\[
\Lambda(\cF_k,C_k,M_{\cF_k}(C_k),\rho_k) / \rho_k \le \mbox{\tt Tol1},
\]
where we set $\mbox{\tt Tol1} = .005$ in our experiments.
\item[(ii)] The maximum overlap falls below a small fraction
  $\mbox{\tt Tol2}$ of the volume of the enclosing ellipsoid. We used
  $\mbox{\tt Tol2}=.0001$ in our experiments.
\item[(iii)] The algorithm runs for $100$ iterations.
\end{itemize}

Many instances of the problem, including problems of different sizes
and shapes, with different random starting points, were executed on
servers running various versions of Linux.
% at the University of California-Berkeley and the Wisconsin
% Institutes for Discovery at the University of Wisconsin-Madison.

\subsection{Radial Preferences} \label{sec:radial}

In the first set of experiments, we use Algorithm~\ref{alg:xxx} to
arrange the CTs so as to minimize the maximal pairwise overlap, with
overlap measured as in Section~\ref{sec:overlap}. (We enforce no
constraints on homologous pairs in this first set.) We set up numerous
trials with the data varied as follows.
\begin{itemize}
\item[(i)] CT volumes are obtained by sampling from a normal
  distribution, with mean taken from Table~\ref{table_chr_volume} and
  the standard deviation set to $.02$ of the mean.
\item[(ii)] The relative axis lengths are varied around the intercepts
  found in \cite{Khalil_heterologues} for mouse chromosomes, namely,
  1:2.9:4.4. The second and third ratios are sampled from a Gaussian
  distribution with mean values $2.9$ and $4.4$, respectively, and
  standard deviations of $.1$ times the mean. (The absolute axis
  lengths are then adjusted to match the volume chosen in (i).)
\end{itemize}
We analyzed the radial preferences for two different nucleus shapes
--- spherical and flat ellipsoidal with axis ratios of 1:2:4 --- and
for the small, medium, and large nuclei with sizes described above.
% (small: $500\mu m^3$, medium: $1000\mu m^3$, large: $1600\mu m^3$).

For each of these six different scenarios we ran 100-200 trials, each
with data perturbed as described above and each from a different
random starting point. We applied a screening step to remove those
trials that have a final objective value greater than
\[
f^* + \max  ( 0.5, \min (0.2*f^*, 2.0)),
\]
where $f^*$ is the lowest objective value obtained over all trials
for this scenario. Table~\ref{ta:objectives:0} shows statistics on the
final objective value for each of the six scenarios. Only a few trials
were removed in the screening step, mostly for the large spherical nucleus in which the no-overlap solution was not quite attained. After
screening, the final objective values were similar for all
trials on a given scenario.

Recall we use Algorithm~\ref{alg:xxx} to solve the convex relaxation
\eqnok{prob:non_convex_ell_relax} of the original formulation
\eqnok{prob:non_convex_ell}, in which the prescribed half-axis lengths
$(r_{i1},r_{i2},r_{i3})$ are replaced by the constraints
\eqnok{constraints_ellipsoid}. We found that a number of the
ellipsoids were ``more rounded'' at the solution than our prescription
would require, but that the deformation typically affected only a
subset of the ellipsoids and was not too severe. By taking relative
difference in the $\ell_2$ norm between the vector of actual half-axis
lengths and the prescribed values, we found that on small nuclei, an
average of 8 of the 46 CTs experienced a relative change of
greater than $10\%$. For medium spherical nuclei, about 7 out of 46
changed by more than $10\%$ while the corresponding number for large
spherical nuclei is 11 out of 46. The statistics for ellipsoidal
nuclei are slightly smaller, about 5 for small and large nuclei and 3
for medium nuclei.

%(Algorithm~\ref{alg:xxx} was terminated
%for the large-nucleus trials when the maximum overlap dropped below
%$10^{-4}$ of the nuclear volume of 1600 $\mu m^3$, or $.16$.)

\begin{table}
\centering
\caption{Statistics for final objective values attained in the six scenarios, showing number of trials, means, and standard deviations, both before and after the screening step.\label{ta:objectives:0}}
\begin{tabular}{cc|ccc|ccc}
&& \multicolumn{3}{c|}{Before Screening} &
\multicolumn{3}{c}{After Screening} \\ 
shape & vol ($\mu m^3$) & trials & mean & sd & trials 
& mean & sd \\ \hline
spherical &    500 &  100 &  3.0889 & 0.0533 &  100 &  3.0889 & 0.0533 \\
ellipsoidal &    500 &  100 &  3.2769 & 0.0660 &  100 &  3.2769 & 0.0660 \\
spherical &   1000 &  200 &  1.8927 & 0.4409 &  195 &  1.8233 & 0.0657 \\
ellipsoidal &   1000 &  200 &  1.9723 & 0.0714 &  200 &  1.9723 & 0.0714 \\
spherical &   1600 &  100 &  0.6342 & 1.4325 &   89 &  0.1349 & 0.0362 \\
ellipsoidal &   1600 &  100 &  0.1338 & 0.0291 &  100 &  0.1338 & 0.0291 \\
\end{tabular}
\end{table}

We analyzed the solutions generated in the trials remaining after the
screening step to find the distances of each ellipsoid from the center
of the nucleus. Figure~\ref{boxplot_medium_ellipsoids} contains scatter plots that
show the mean volume of each CT (on the horizontal axis) plotted
against the distance between the center of that CT and the nuclear
center (on the vertical axis), for a medium-sized nucleus (volume of
$1000$ $\mu m^3$) and for both spherical and flat ellipsoidal shapes.
% (in Figures~\ref{boxplot_medium_ellipsoids} (a) and
% \ref{boxplot_medium_ellipsoids} (b), respectively).
(The scatter plots for the large and small volume nuclei are similar,
so we do not show them here.)  A least-squares regression line is also
shown. In both graphs, a negative trend is detectable, meaning that
the larger ellipsoids tend to lie closer to the nuclear center, while
the smaller ones prefer peripheral positions. This is the opposite
trend to the one observed in nature, suggesting that the
minimum-overlap criterion alone is insufficient to explain the
experimental results.

\begin{figure}[!t]
\centering
\subfigure[Spherical Nucluei]{\includegraphics[scale=0.24]{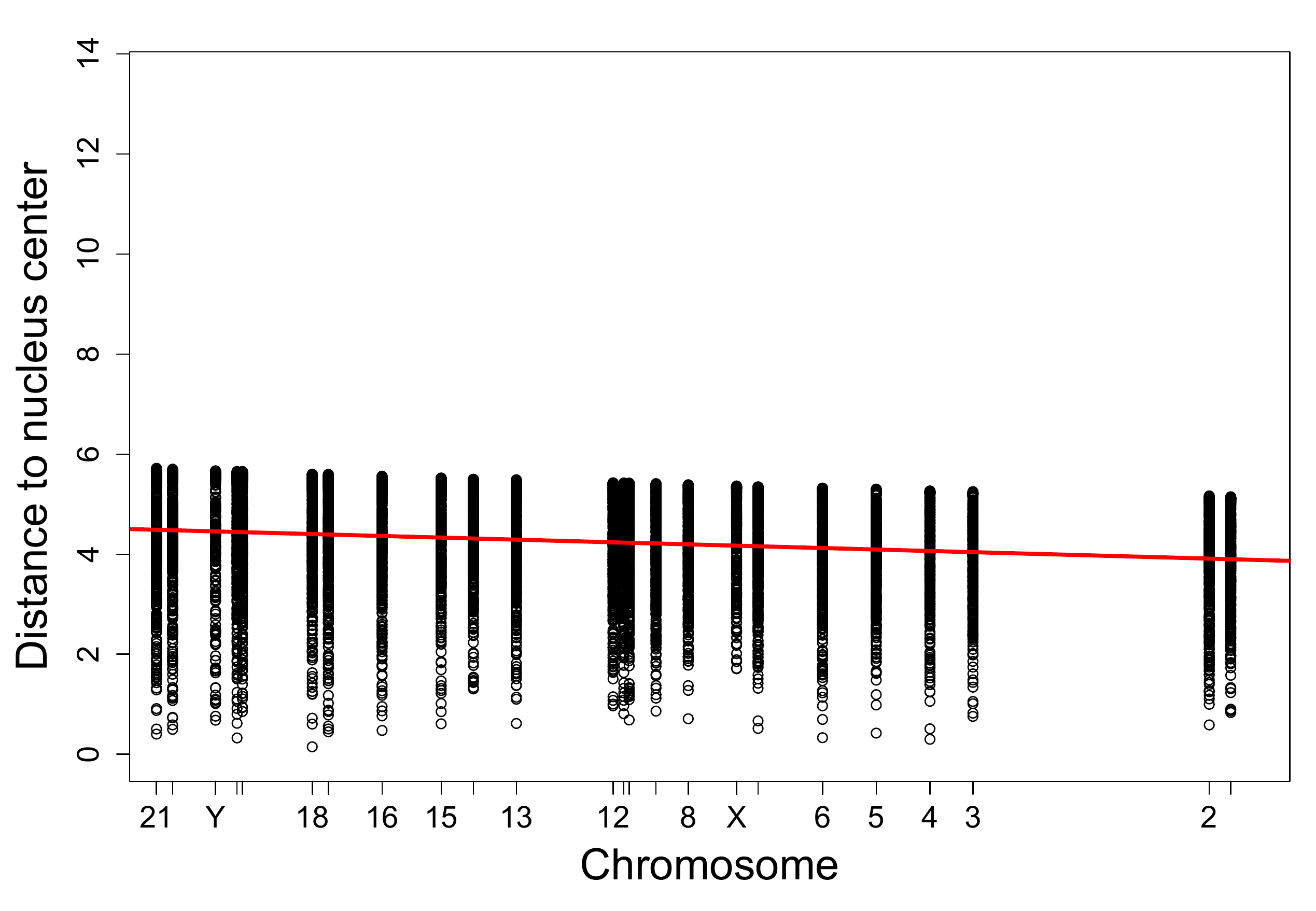}\label{boxplot_medium_ellipsoids.sphere}} \;
\subfigure[Ellipsoidal Nuclei]{\includegraphics[scale=0.24]{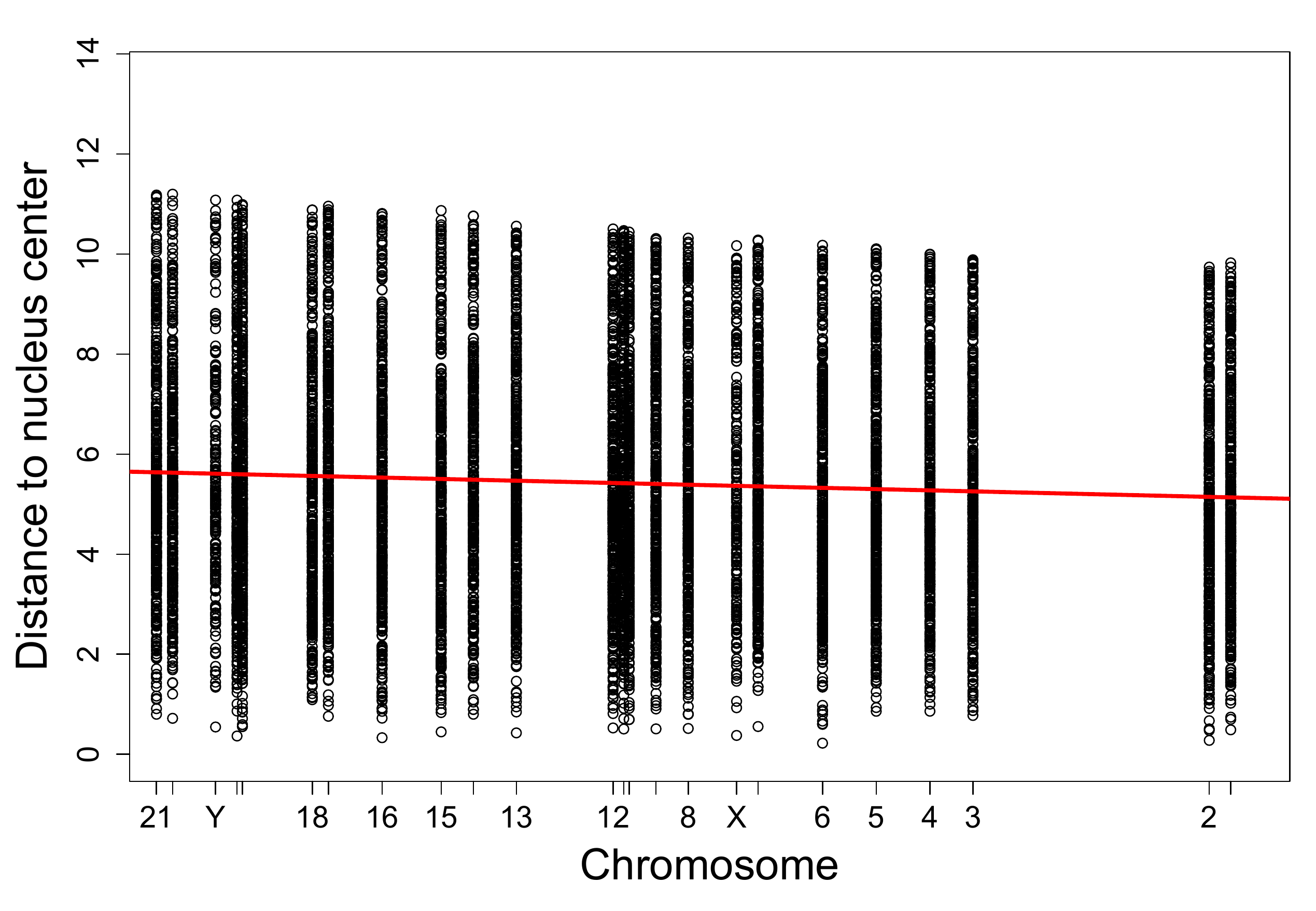}\label{boxplot_medium_ellipsoids.ellipse}} 
\caption{Scatter plots and regression lines for distances of
  ellipsoidal chromosome territories to nucleus center, for
  medium-sized nuclei. Horizontal axis is CT volume, vertical axis is
  distance to nucleus center.}
% The scatter plot and corresponding regression
%  line for medium sized spherical nuclei is shown in the left figure
%  and the scatter plot and regression line for medium sized flat
%  ellipsoidal nuclei in the right figure.}
\label{boxplot_medium_ellipsoids}
\end{figure}

In Table~\ref{table_slope_0.00} we report the slopes of the regression
line for all six scenarios. Interestingly, the negative trend is
consistently weaker for flat ellipsoidal nuclei compared to spherical
nuclei. Experimentalists report a preference for larger CTs to be on
the periphery for ellipsoidal nuclei, while for spherical nuclei, the
radial preferences are believed to be correlated with gene density.

\begin{table}[!h]
\centering
\caption{Slope of regression line for all six scenarios.}
\label{table_slope_0.00}
\begin{tabular}{l | c c c}
 & small & medium & large \\ \hline
 spherical & -0.0050562 & -0.0029405 & -0.0020672\\
 ellipsoidal & -0.0047345 & -0.0025045 & -0.0018849 \end{tabular}
 \end{table}

\subsection{Radial preferences assuming heterologous CT groupings}
\label{sec:homologs}

Khalil et al. \cite{Khalil_heterologues} showed that CTs tend to
assemble in heterologous neighborhoods, causing the distances between
homologous chromosome pairs to be larger in general than heterologous
inter-CT distances. They discuss a number of possible explanations for
this phenomenon, such as that heterologous neighborhoods act as a
buffer zone in preventing inter-homologue recombination and protect
against the loss of heterozygosity. The authors also analyze whether
the radial preferences discussed in the previous subsection could
explain the preference for arrangements with larger homologous
inter-CT distances. Using simulations, they give a negative answer to
this question.

In the following analysis, we invert the question, asking instead
whether the preference for heterologous neighborhoods can explain the
observed radial preferences. To investigate this question, we add
penalties to our model to discourage the CTs in a homologous pair from
being too close to each other. We solve the resulting formulation
using a combination of Algorithm~\ref{alg:circ} for sphere packing
with Algorithm~\ref{alg:xxx} for ellipsoid packing.

We denote the set of index pairs $(i,j)$ corresponding to homologous
chromosome pairs by $H$ and we introduce a new variable $\eta$ to
capture the proximity of CTs in a homologous pair.  Specifically, we
define for each ellipsoid $i$ an enclosing sphere that is concentric
with the ellipsoid $i$, with radius $\lambda$ times the maximum
semi-axis length $r_i$ of the CT, where $\lambda \ge 1$ is a
user-defined parameter. We define $\eta$ to be the maximal overlap of
these enclosing spheres, over all homologous pairs, by adding
constraints whose form is similar to \eqnok{prob:min_overlap.2}. We
then add a penalty term $c \eta$ to the objective (where $c\geq 0$ is some
penalty parameter), to obtain the following extension of formulation
\eqnok{prob:non_convex_ell}.
\begin{subequations}
\label{prob:non_convex_ell_heterologous}
\begin{align}
\min_{\xi, \eta, (c_i,S_i,\Sigma_i), i=1,2,\dotsc,N} \;&
\xi + c \eta&\\
\mbox{subject to} \quad&\xi \geq \hat{O}(c_i, c_j, \Sigma_i, \Sigma_j),& 1\leq i<j\leq N, \label{prob:non_conv_ell1}\\ 
\label{prob:non_conv_ell5} & \lambda (r_i - r_j) - \norm{c_i-c_j}_2 \leq \eta, & (i,j)\in H, \\ 
\label{prob:non_conv_ell2} &\mathcal{E}_i\subset \mathcal{E}, & i=1,2,\dots,N, \\ 
\label{prob:non_conv_ell4} &\Sigma_i=S_i^2, \\
\label{prob:non_conv_ell3} &\mbox{semi-axes of }\mathcal{E}_i \mbox{ are } r_{i1}, r_{i2}, r_{i3},  & i=1,2,\dots,N.
\end{align}
\end{subequations}
We can relax this to obtain an extended formulation of
\eqnok{prob:non_convex_ell_relax}. To solve, we extend
Algorithm~\ref{alg:xxx} by adding linearizations of the constraints
\eqnok{prob:non_conv_ell1} to each subproblem, in the manner of
\eqnok{prob:min_overlap_lin.2}.

% By linearizing the non-convex constraints as described in Section
% \ref{sec:sphere} and Section \ref{sec:ellipsoid} we can formulate an
% iterative bilevel optimization procedure, for which the convergence
% results described in Theorem \ref{th:conv} hold. With the resulting
% algorithm we can model chromosomal arrangement as a minimal overlap
% ellipsoidal packing with heterologous neighborhoods inside an
% ellipsoidal container.

For our simulations, we choose $c = 100$ and $\lambda = 1.25$. As in
Subsection~\ref{sec:radial}, we generated about 100-200 trials by
perturbing CT volumes and dimensions randomly around given mean values
and using different random starting points. The screening procedure
described in the previous subsection was applied to remove those
trials with less competitive final objective values. Statistics for
the final objectives are shown in Table~\ref{ta:objectives:1.25}. The
large objective values in the first line of the table indicates that
for small spherical nuclei, it was not possible to find solutions in
which the homolog separation was enforced adequately. (The only trial
that survived screening was one that violated these conditions
significantly less than most others.) Among the other scenarios, only
the medium spherical nucleus saw significant numbers of trials removed
by screening. Here, most of the trials attained final objectives quite
close to 1.90, while the others had significantly higher values. In
the other four scenarios --- small ellipsoidal, medium ellipsoidal,
large spherical, and large ellipsoidal --- proximity penalties for
homologous pairs were not incurred, and final objective values were
tightly clustered.

\begin{table}[!b]
\centering
\caption{Statistics for final objective values attained in the six scenarios in which homolog proximity is penalized, showing number of trials, means, and standard deviations, both before and after the screening step.\label{ta:objectives:1.25}}
\begin{tabular}{cc|ccc|ccc}
&& \multicolumn{3}{c|}{Before Screening} &
\multicolumn{3}{c}{After Screening} \\ 
shape & vol ($\mu m^3$) & trials & mean & sd & trials 
& mean & sd \\ \hline
spherical &    500 &  100 & 294.0617 & 46.8185 &    1 & 184.0292 &  0.0000 \\
ellipsoidal &    500 &  100 &   3.6556 &  0.2691 &   95 &   3.5987 &  0.0879 \\
spherical &   1000 &  200 &  15.0885 & 20.7260 &  114 &   1.8993 &  0.0833 \\
ellipsoidal &   1000 &  200 &   2.0088 &  0.5118 &  192 &   1.9060 &  0.0691 \\
spherical &   1600 &  100 &   0.4424 &  1.2293 &   94 &   0.1369 &  0.0213 \\
ellipsoidal &   1600 &  100 &   0.2752 &  0.8097 &   97 &   0.1343 &  0.0251 
\end{tabular}
\end{table}

The convex relaxation of our problem that encourages separation of
homologous CT pairs does less well in preserving the dimensions of the
ellipsoids than the formulation considered in
Section~\ref{sec:radial}.  For spherical nuclei 26 out of the 46 CTs
for small nuclei experienced a relative change in the half-axes
lengths of more than 10\%. For the medium spherical nuclei it was in
average 11 out of 46 and for the large spherical nuclei 17 out of
46. The statistics for the ellipsoidal nuclei were somewhat smaller:
12 for the small nuclei, 4 for the medium nuclei, and 9 for the large
nuclei.  On the small nuclei, the distortions can be explained by the
tightness of space, while on large nuclei, the fact that all CTs can
be fit without any overlap reduces the need for them to adopt their
lowest-volume dimensions (which would achieve the prescribed semi-axis
lengths).

Figure~\ref{boxplot_medium_ellipsoids_heterologous} contains scatter
plots showing the mean volume of each CT (on the horizontal axis)
plotted against the distance between the center of that CT and the
nuclear center (on the vertical axis), for a medium-sized nucleus
(volume of $1000$ $\mu m^3$) and for both spherical and flat
ellipsoidal shapes.  (As in Subsection~\ref{sec:radial}, scatter plots
for the large and small volume nuclei are similar, so we do not show
them here.) Here the regression line shows a significant positive
trend, meaning that the smaller ellipsoids tend to lie in the interior
of the nucleus, while the larger ones prefer peripheral positions.
Hence, by adding penalties on nearness of homologous pairs to the
formulation, we are able to match the radial preferences observed in
nature.

Another interesting observation, more evident in
Figure~\ref{boxplot_medium_ellipsoids_heterologous}(a), is that the X
and Y chromosomes both lie closer to the nucleus center than their size
would suggest.  This makes sense, as these are the only two
chromosomes not subject to the homologous-pair separation penalties.

%so they are not forced by these penalties to be further from the
%nuclear center.

\begin{figure}[!t]
\centering
\subfigure[Spherical Nuclei]{\includegraphics[scale=0.24]{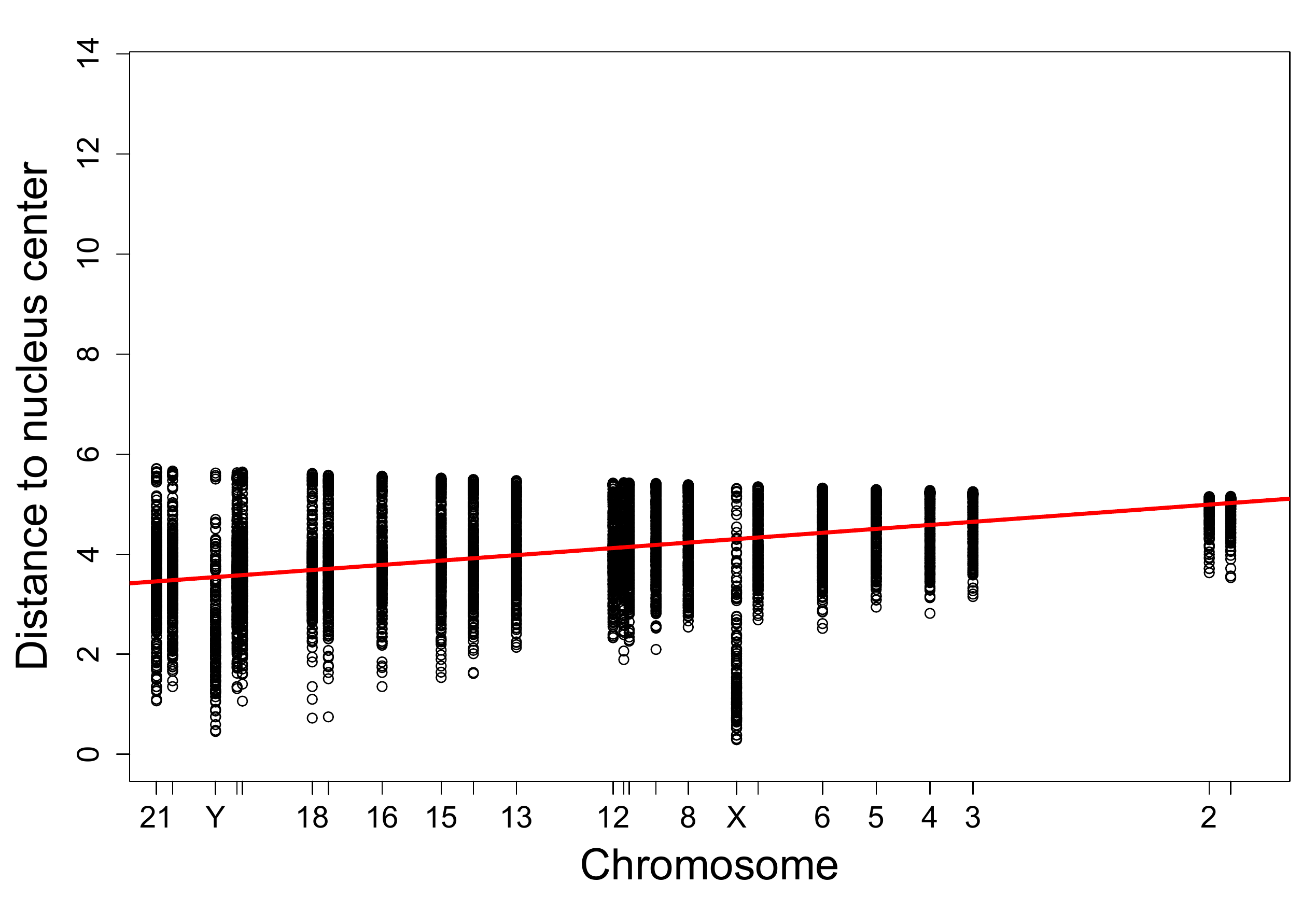}\label{boxplot_medium_ellipsoids_heterologous.sphere}} \;
\subfigure[Ellipsoidal Nuclei]{\includegraphics[scale=0.24]{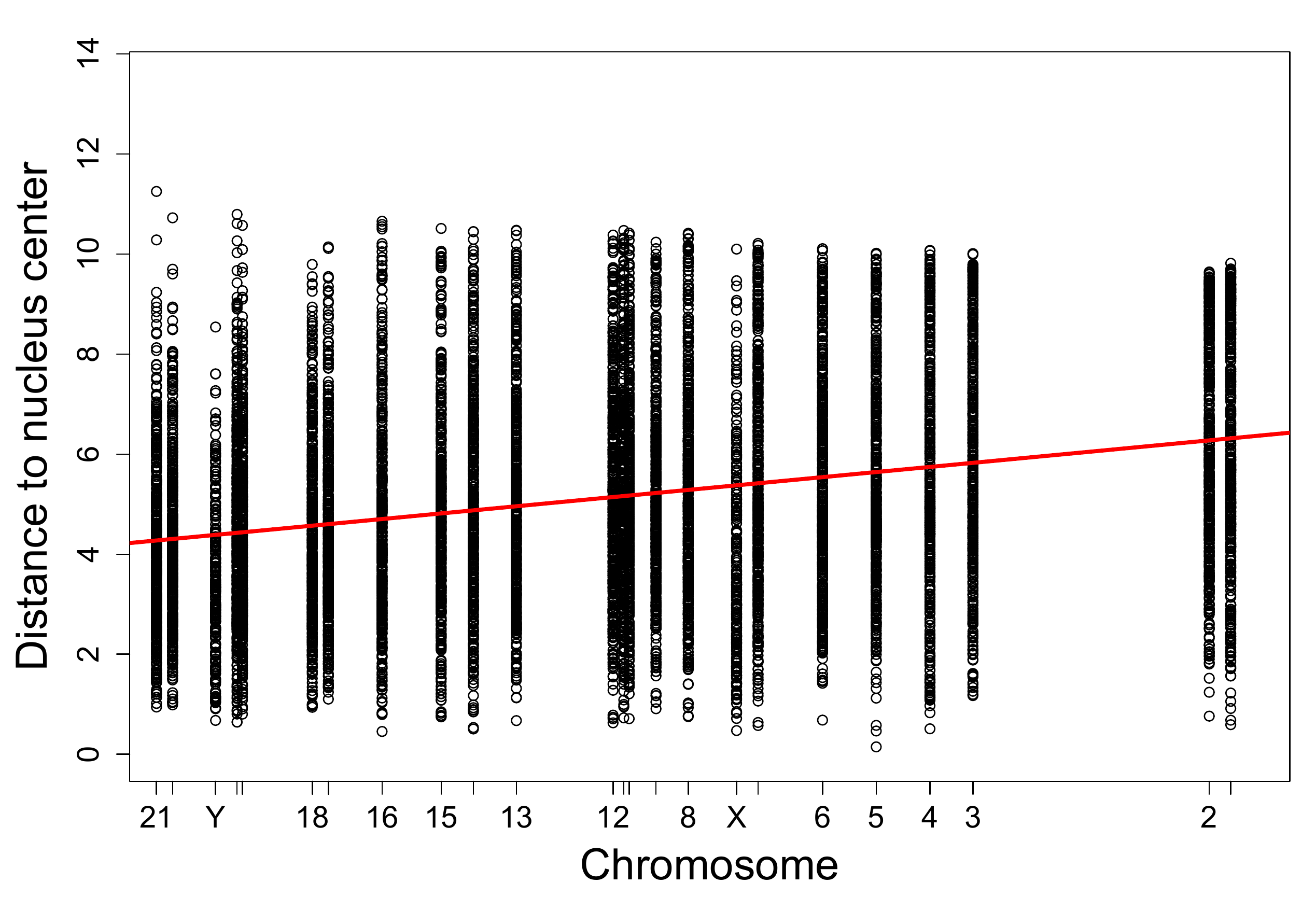}\label{boxplot_medium_ellipsoids_heterologous.ellipsoid}} 
\caption{Scatter plots and regression lines for ellipsoidal chromosome
  territory distances to nucleus center, where penalties to enforce
  heterologous groupings are present in the objective.  Horizontal
  axis is CT volume, vertical axis is distance to nucleus center.}
  % The scatter plot and corresponding regression line for medium
  % sized spherical nuclei is shown in the left figure and the scatter
  % plot and regression line for medium sized flat ellipsoidal nuclei
  % in the right figure.}
\label{boxplot_medium_ellipsoids_heterologous}
\end{figure}

In Table~\ref{table_slope_1.25} we report the slopes of the regression
line for all six size / shape scenarios considered in this
section. These results highlight a significant difference between
spherical and flat-ellipsoidal nuclei. The radial preference is
consistently weaker for spherical nuclei than in flat-ellipsoidal
nuclei.

\begin{table}[!h]
\centering
\caption{Slope of regression line for all six scenarios assuming heterologous CT groupings.}
\label{table_slope_1.25}
\begin{tabular}{l | c c c}
 & small & medium & large \\ \hline
 spherical & 0.0031803 & 0.0078299 & 0.0072028 \\
 ellipsoidal & 0.0081768 & 0.0102088 & 0.0145001 \end{tabular}
 \end{table}

\section{Discussion} \label{sec:discussion}

We have described a bilevel optimization procedure for finding local
solutions of the problem of packing spheres and ellipsoids in finite
volumes, and used these procedures to model and analyze chromosome
arrangement in cell nuclei. Semidefinite programming duality is used
to obtain the sensitivity information needed to construct the
approximation to the upper-level problem that is solved at each
iteration of the trust-region procedure.  Our convergence analysis
takes place in a general setting in which the lower-level problems are
semidefinite programs parametrized by their objective coefficient
matrix; it is not confined to the specific form of the semidefinite
programs arising from the S-procedure for overlapping ellipsoids. Thus
it may be adaptable to other design problems involving parametrized
systems that can be modeled by semidefinite programs.

In the CT packing application discussed in Section~\ref{sec:biology},
we initially found that the arrangements observed experimentally could
not be explained by the simple geometric principle of minimizing the
maximum overlap. However, when we enhanced the model to capture the
recently observed phenomenon of heterologous neighborhoods /
homologous pair separation, the radial preferences observed in nature
(in which larger CTs tended to lie further from the nuclear center)
were recovered in our simulations.  The homologous-pair-separation
aspects of our model are governed by two positive parameters $c$ and
$\lambda$; we reported results in Subsection~\ref{sec:homologs} only
for the values $c=100$ and $\lambda=1.25$. From an examination of
Tables~\ref{table_slope_0.00} and \ref{table_slope_1.25}, we speculate
that it would be possible to choose these parameters in such a way
that the slope of the regression line for spherical nuclei would be
approximately zero, while the corresponding slope for ellipsoidal
nuclei would be positive. Such a result would be consistent with
experimental observations that identify no clear radial preference for
spherical nuclei, but a pronounced radial preference for ellipsoidal
nuclei.

% This might explain why the results regarding the radial preferences
% for spherical nuclei do not all agree and sometimes find a
% correlation with chromosome size, sometimes with gene density, or
% both.

We obtained results on a limited but representative range of nuclei
dimensions. In future work, we will explore CT configurations for a
wider range of ellipsoidal shapes and sizes, corresponding to known
dimensions of nuclei in different cell types. We will also enhance the
model as further biological results are obtained, aiming to find
biologically plausible, elementary principles that explain
experimental observations (in the spirit of Occam's Razor).

\appendix

\section{Technical Results for Parametrized Semidefinite Programs}

In this section we consider the following primal-dual pair of
semidefinite programs that are parametrized by the primal objective
term $C$:
\beq \label{eq:sdpp}
\min_X \, \langle C,X \rangle \;\; \mbox{s.t.} \;\; \langle A_i, X \rangle = b_i, \;\; i=1,2,\dotsc,p, \;\; X \succeq 0, \\
\eeq
\beq \label{eq:sdpd}
\max_{\zeta,S} \, b^T\zeta \;\; \mbox{s.t.} \;\; \sum_{i=1}^p \zeta_i A_i + S = C, \;\; S \succeq 0.
\eeq
We denote solutions of these problems by $X(C)$ and $(\zeta(C),S(C))$,
respectively. (Our interest is in the application to the SDP pair
\eqnok{eq:50}, \eqnok{eq:50a}, but we have simplified the notation
here.)

We show first that the solutions to \eqnok{eq:sdpp} are uniformly bounded in a
neighborhood of a $C$ for which a strictly feasible point for the dual
\eqnok{eq:sdpd} exists. The result is an almost immediate consequence
of \cite[Theorem~4.1]{Tod01a}.
\begin{lemma} \label{lem:bd} Consider the primal-dual pair
  \eqnok{eq:sdpp}, \eqnok{eq:sdpd} of semidefinite programs: Suppose
  that \eqnok{eq:sdpp} is feasible (with feasible point $\hat{X}$),
  and that at some matrix $C_0 \in S \R^{n \times n}$, there exists a
  strictly feasible point $(\hat{\zeta},\hat{S})$ for \eqnok{eq:sdpd},
  where the eigenvalues of $\hat{S}$ are bounded below by
  $\sigma>0$. Then there exists a constant $\delta>0$ such that for
  all matrices $C \in S \R^{n \times n}$ with $\|C-C_0\|_F \le
  \delta$, \eqnok{eq:sdpp} has a nonempty solution set, and all
  solutions $X(C)$ are bounded as follows:
\[
\| X (C) \|_F \le \frac{2}{\sigma} (\langle \hat{X}, \hat{S} \rangle + 
\delta \|\hat{X} \|_F).
\]
% That is, there is a neighborhood of $C$ within which the solution
% $X(C)$ of \eqnok{eq:sdpp} is uniformly bounded.
Moreover the optimal values of the problems \eqnok{eq:sdpp} and
\eqnok{eq:sdpd} are equal.
\end{lemma}
\begin{proof}
  We have for any $X$ feasible for the primal (note that the primal
  feasible region does not depend on $C$) that
\begin{align*}
\langle C,X \rangle &= \langle C_0,X \rangle +  \langle C-C_0,X \rangle \\
&= \langle \sum_{i=1}^p \hat{\zeta}_i A_i + \hat{S} ,X \rangle 
+  \langle C-C_0,X \rangle \\
&= b^T \hat{\zeta} + \langle \hat{S} ,X \rangle 
+  \langle C-C_0,X \rangle.
\end{align*}
Since $b^T \hat{\zeta}$ is independent of $X$, we can obtain an
equivalent to the primal problem by replacing its objective $\langle
C,X \rangle$ by $\langle \hat{S} + C-C_0,X \rangle$.  Using the
assumed feasible point $\hat{X}$ of \eqnok{eq:sdpp} (note that there
is no dependence of $\hat{X}$ on $C$), we can formulate
\eqnok{eq:sdpp} equivalently as follows:
\begin{subequations} \label{eq:sdpp.1}
\begin{align}
\min_X \, \langle \hat{S} + C-C_0 ,X \rangle \;\; \mbox{s.t.} \;\;& \langle A_i, X \rangle = b_i, \;\; i=1,2,\dotsc,p, \;\; X \succeq 0, \\
&\langle \hat{S} + C-C_0 ,X \rangle \le 
\langle \hat{S} + C-C_0 , \hat{X} \rangle.
\end{align}
\end{subequations}
By choosing $\delta \in (0, \sigma/2]$, we have that all eigenvalues
of $\hat{S} + C-C_0$ are bounded below by $\sigma/2$. Hence from
``Fact 14'' of \cite{Tod01a}, we have that
\[
\langle \hat{S} + C-C_0 ,X \rangle  \le 
\langle \hat{S} + C-C_0 , \hat{X} \rangle \; \Rightarrow \;
\|X\|_F \le \frac{2}{\sigma} \langle \hat{S} + C-C_0 , \hat{X} \rangle \le
\frac{2}{\sigma}  (
\langle \hat{S} , \hat{X} \rangle + \delta \| \hat{X} \|_F).
\]
Hence, \eqnok{eq:sdpp.1} involves the minimization of a continuous
function over a nonempty compact set, so the solution set exists, and
moreover, all solutions are bounded as claimed.

The last claim can be derived exactly as in
\cite[Theorem~4.1]{Tod01a}.
\end{proof}

The next result examines the solution of a sequence of parametrized
SDPs.
\begin{theorem} \label{th:limitsol} Given $A_i$, $i=1=1,2,\dotsc,p$
  and $b \in \R^p$ as in Lemma~\ref{lem:bd}, such that \eqnok{eq:sdpp}
  is feasible, let $\bar{C} \in S\R^{n \times n}$ be such that there
  exists a strictly feasible point for \eqnok{eq:sdpd} when
  $C=\bar{C}$. Consider a sequence $\{ C_k \}$ with $C_k \in S \R^{n
    \times n}$ and $C_k \to \bar{C}$. Then the following is true:
\begin{itemize}
\item[(i)] There is a constant $\beta >0$ and index $K$ such that
  \eqnok{eq:sdpp} with $C=C_k$ has nonempty solution set for all $k
  \ge K$, and $\|X(C_k) \|_F \le \beta$ for all such solutions.
\item[(ii)] If $\{X(C_k)\}$ is a sequence of solutions of
  \eqnok{eq:sdpp} with $C=C_k$, then this sequence has at least one
  accumulation point, and all such accumulation points are solutions
  of \eqnok{eq:sdpp} with $C=\bar{C}$.
\end{itemize}
\end{theorem}
\begin{proof}
  The first claim (i) is an immediate consequence of
  Lemma~\ref{lem:bd}. For (ii), note that boundedness of $X(C_k)$
  ensures existence of accumulation points. Suppose that $\bar{X}$ is
  such a point and assume WLOG that $X(C_k) \to \bar{X}$. Note first
  that $\bar{X}$ is feasible for \eqnok{eq:sdpp} regardless of $C$. If
  $\bar{X}$ were not optimal for \eqnok{eq:sdpp} with $C=\bar{C}$,
  then there would exist another feasible matrix $\tilde{X}$ with
  $\langle \bar{C},\tilde{X} \rangle < \langle \bar{C},\bar{X}
  \rangle$. But since 
\[
\lim_k \, \langle C_k,\tilde{X} \rangle = \langle \bar{C},\tilde{X}
\rangle < \langle \bar{C},\bar{X} \rangle = \lim_k \, \langle C_k, X(C_k) \rangle,
\]
we have that $\langle C_k,\tilde{X} \rangle < \langle C_k,X(C_k) \rangle$
for all $k$ sufficiently large, contradicting optimality of
$X(C_k)$. Hence (ii) is true.
\end{proof}

We next prove some elementary and useful facts about the value
function of \eqnok{eq:sdpp}, which we denote by $t(C)$.
\begin{lemma} \label{lem:subdiff} Suppose that \eqnok{eq:sdpp} is
  feasible, and let $\bar{C} \in S\R^{n \times n}$ be such that there
  exists a strictly feasible point for \eqnok{eq:sdpd} when
  $C=\bar{C}$.  Then there exists a neighborhood $\cN$ of $\bar{C}$
  within which the following claims are true.
\begin{itemize}
\item[(i)] $t(\cdot)$ is a concave function.
\item[(ii)] For all $C \in \cN$ and all $\DC \in S \R^{n \times n}$,
  we have
\begin{equation} \label{eq:tcdc}
t(C + \DC) \le t(C) + \langle X(C), \DC \rangle,
\end{equation}
where $X(C)$ is any solution of \eqnok{eq:sdpp}.
\item[(iii)] Any $X(C)$ that solves \eqnok{eq:sdpp} belongs to the
  Clarke subdifferential of $t(\cdot)$ at $C$.
\item[(iv)] $t(\cdot)$ is Lipschitz continuous in $\cN$.
\end{itemize}
\end{lemma}
\begin{proof}
The proof of (i) is elementary.

For (ii), note first from Lemma~\ref{lem:bd} that we can choose $\cN$
so as to ensure that a solution $X(C)$ to \eqnok{eq:sdpp} exists for
all $C \in \cN$. We have (denoting by $\Psi$ the feasible set for
\eqnok{eq:sdpp}) that
\[
t(C +\DC) = \min_{X \in \Psi} \, \langle C + \DC, X \rangle 
 \le \langle C + \DC, X(C) \rangle 
= t(C) + \langle \DC, X(C) \rangle,
\]
as required.

For (iii), note that (by taking the negative of the inequality
\eqnok{eq:tcdc}) $-X(C)$ is a subgradient of the convex function
$(-t)(C)$, and so $-X(C)$ belongs to the Clarke subdifferential of
$(-t)(C)$. It follows from \cite[p.~128, Exercise 8(c)]{BorL00}
(with $\lambda=-1$) that $X(C)$ belongs to the Clarke subdifferential
of $t(C)$, as claimed.

For (iv), note from Lemma~\ref{lem:bd} that we can choose $\cN$ such that
$\| X(C) \|$ is uniformly bounded for all $C \in \cN$ (by $\beta>0$,
say). Denoting by $C_1$ and $C_2$ any two points in $\cN$, we have
from \eqnok{eq:tcdc} that
\begin{align*}
t(C_2) & \le t(C_1) + \langle X(C_1), C_2-C_1 \rangle, \\
t(C_1) & \le t(C_2) + \langle X(C_2), C_1-C_2 \rangle.
\end{align*}
Thus
\[
| t(C_1) - t(C_2) | \le \max \left( \| X(C_1) \|, \|X(C_2) \|
\right) \|C_1-C_2\| \le \beta \|C_1-C_2\|,
\]
proving the Lipschitz property.
\end{proof}

\vspace{0.5cm}

%\section{Acknowledgements}
{\bf Acknowledgments} 

We thank Saira Mian for helpful discussions about the application to
chromosome arrangement in cell nuclei. We are grateful to the
Institute for Mathematics and its Applications at the University of
Minnesota for supporting visits by both authors while this research
was conducted.

\bibliographystyle{plain} 
\bibliography{ellipses}

\end{document}